\documentclass[a4paper,10pt]{article}
\usepackage{amssymb,amsmath,amsthm}
\usepackage{graphicx}
\usepackage{amsfonts}
\usepackage{latexsym}
\usepackage[english]{babel}
\usepackage{hyperref}  % hyperref e showkeys a volte sono incompatibili
\usepackage{color} 

\usepackage{anysize}
\marginsize{3.0cm}{3.0cm}{2.2cm}{2.2cm}

\let\OLDthebibliography\thebibliography
\renewcommand\thebibliography[1]{
  \OLDthebibliography{#1}
  \setlength{\parskip}{0pt}
  \setlength{\itemsep}{0pt plus 0.3ex} }

\numberwithin{equation}{section}

\theoremstyle{plain}
\newtheorem{theorem}{Theorem}[section]

\newtheorem{lemma}[theorem]{Lemma}

\newtheorem{definition}[theorem]{Definition}

\theoremstyle{definition}
\newtheorem{remarkbase}[theorem]{Remark}
\newenvironment{remark}{\pushQED{\qed} \remarkbase}{\popQED\endremarkbase}

\newcommand{\om}{\omega}

\newcommand{\ph}{\varphi}

\newcommand{\R}{\mathbb R}
\newcommand{\C}{\mathbb C}

\newcommand{\Lm}{\Lambda}

\newcommand{\Z}{\mathbb Z}

\newcommand{\N}{\mathbb N}
\newcommand{\T}{\mathbb T}

\renewcommand{\a }{\alpha }
\newcommand{\s }{\sigma }
\newcommand{\ii }{{\mathrm i} }
\renewcommand{\d }{\delta }
\renewcommand{\b }{\beta }

\newcommand{\vphi}{\varphi }

\renewcommand{\t }{\tau }

\newcommand{\pa}{\partial}
\newcommand{\lm}{\lambda}

\newcommand{\mS}{\mathcal{S}}
\newcommand{\mL}{\mathcal{L}}
\newcommand{\mN}{\mathcal{N}}
\newcommand{\mM}{\mathcal{M}}

\newcommand{\mR}{\mathcal{R}}
\newcommand{\mA}{\mathcal{A}}
\newcommand{\mF}{\mathcal{F}}
\newcommand{\mH}{\mathcal{H}}
\newcommand{\mC}{\mathcal{C}}
\newcommand{\mT}{\mathcal{T}}
\newcommand{\mB}{\mathcal{B}}

\newcommand{\cc}{\mathfrak{C}}
\renewcommand{\oe}{{\"o}}

\DeclareMathOperator{\Op}{Op}

\title{Controllability of quasi-linear Hamiltonian NLS equations}
\date{}
\author{\small{Pietro Baldi, Emanuele Haus, Riccardo Montalto}}

\begin{document}
\maketitle

\begin{small}
\textbf{Abstract.}
We prove internal controllability in arbitrary time, for small data, 
for quasi-linear Hamiltonian NLS equations on the circle.
We use a procedure of reduction to constant coefficients up to order zero and 
HUM method to prove the controllability of the linearized problem. 
Then we apply a Nash-Moser-H\"ormander implicit function theorem as a black box. 
\emph{MSC2010:} 35Q55, 35Q93.
\end{small}

\bigskip

\emph{Contents.}
\ref{sec:intro} Introduction ---
\ref{riduzione operatori lineari generali} Reduction of the linearized operator ---
\ref{sezione osservabilita} Observability ---
\ref{sezione controllabilita} Controllability ---
\ref{sec:proof} Proofs ---
\ref{sezione formalismo hamiltoniano} Appendix A. Quadratic Hamiltonians and linear Hamiltonian vector fields ---
\ref{appendice paradiff} Appendix B. Classical tame estimates ---
\ref{sec:WP} Appendix C. Well-posedness of linear equations ---
\ref{sec:NM} Appendix D. Nash-Moser-H\"ormander theorem.

\bigskip

\section{Introduction} \label{sec:intro}
We consider a class of 
nonlinear Schr\"odinger equations (NLS) on $\T := \R / 2 \pi \Z$ of the form 
\begin{equation}\label{NLS quasi-lineare}
\partial_t u + \ii \partial_{xx} u + {\cal N}(x, u, \pa_x u, \pa_{xx} u) = 0 \,, \quad x \in \T,
\end{equation}
for the complex-valued unknown $u = u(t,x)$.
We assume that ${\cal N}$ is a 
\emph{Hamiltonian, quasi-linear} nonlinearity 
\begin{equation} \label{1506.1}
{\cal N}(x, u, u_x, u_{xx}) =  -\ii \Big( \partial_{\overline z_0} F(x, u, u_x)  - \partial_x \{ \partial_{\overline z_1} F(x, u, u_x) \} \Big)  \,,
\end{equation}
where $u_x, u_{xx}$ denote the partial derivatives $\pa_x u, \pa_{xx} u$, 
$F : \T \times \C^2 \to \R$ is a real-valued function, 
\begin{equation} \label{regolarita Hamiltoniana}
F \Big( x, \frac{y_1 + \ii y_2}{\sqrt 2}\,, \frac{y_3 + \ii y_4}{\sqrt 2} \Big) 
= G(x, y_1, y_2, y_3, y_4) 
\quad \text{for some} \ G \in C^r(\T \times \R^4 , \R)\,,
\end{equation}
and the differential operators $\partial_{\overline z_0}, \partial_{\overline z_1}$ 
in \eqref{1506.1} are defined as 
\begin{equation} \label{2306.3}
\partial_{\overline z_0} = \frac{1}{\sqrt{2}}\, (\partial_{y_1} + \ii \partial_{y_2}), 
\quad 
\partial_{\overline z_1} = \frac{1}{\sqrt{2}}\, (\partial_{y_3} + \ii \partial_{y_4}).
\end{equation}
We assume that $G$ satisfies 
\begin{equation} \label{1506.2}
|G(x, y)| \leq C |y|^3
\quad \forall y = (y_1, y_2, y_3, y_4) \in \R^4, \ |y| \leq 1.
\end{equation}
Equation \eqref{NLS quasi-lineare} is Hamiltonian 
in the sense that it can be written as
$$
\partial_t u = \ii \nabla_{\bar  u}{\cal H}(u)
$$
where $\nabla_{\bar u} := \frac{1}{\sqrt 2}\, (\nabla_{u_1} + \ii \nabla_{u_2})$, 
$\nabla$ is the $L^2(\T)$ gradient, 
$u = \frac{1}{\sqrt 2}\, (u_1 + \ii u_2)$, 
and the real Hamiltonian ${\cal H}(u)$ is given by 
\begin{equation}\label{Hamiltoniana NLS}
{\cal H}(u) = \int_\T \big( | u_x |^2 + F(x, u, u_x) \big)\, d x\,.
\end{equation}
We underline that \eqref{NLS quasi-lineare} is, in fact, the \emph{real} Hamiltonian system
\begin{equation} \label{2306.1}
\pa_t \begin{pmatrix} u_1 \\ u_2 \end{pmatrix}
= J \begin{pmatrix} \nabla_{u_1} H(u_1, u_2) \\ \nabla_{u_2} H(u_1, u_2) \end{pmatrix} 
\end{equation}
for the real-valued unknowns $u_1, u_2$, 
where $J := \begin{pmatrix} 0 & -1 \\ 1 & 0 \end{pmatrix}$ and 
\begin{align} \label{2306.2}
H(u_1, u_2) := \mH \Big( \frac{u_1 + \ii u_2}{\sqrt 2} \Big)
= \frac12 \int_\T \big( (\pa_x u_1)^2 + (\pa_x u_2)^2 \big) \, dx 
+ \int_\T G \big( x, u_1, u_2, \pa_x u_1, \pa_x u_2 \big) \, dx.
\end{align}
As a consequence, the assumption of finite regularity of $G$, i.e. $G \in C^r$ (only finitely many times differentiable)
in \eqref{regolarita Hamiltoniana} is compatible 
with the Hamiltonian structure --- 
in particular, no analyticity assumption is needed on the Hamiltonian.

For example, if $G(x,y_1,y_2,y_3,y_4) = \frac18 a(x) (y_3^2 + y_4^2)^2$, then 
$\pa_{\bar z_1} F(x, u, u_x) = a(x) |u_x|^2 u_x$, and 
$\mN(x,u,u_x,u_{xx}) = \ii \pa_x \{ a(x) |u_x|^2 u_x \}
= \ii a_x(x) |u_x|^2 u_x + \ii a(x) (u_x^2 \bar u_{xx} + 2 |u_x|^2 u_{xx})$; 
%if $G(x,y_1,y_2,y_3,y_4) = \frac18 a(x) (y_1^2 + y_2^2)^2$, then 
%$\pa_{\bar z_0} F(x, u, u_x) = a(x) |u|^2 u$, and 
%$\mN(x,u,u_x,u_{xx}) = \ii \pa_x \{ a(x) |u|^2 u \}
%= \ii a_x(x) |u|^2 u + \ii a(x) (u^2 \bar u_x + 2 |u|^2 u_x)$. 
if $G = \frac18 (y_1^2 + y_2^2)^2$, then 
$\pa_{\bar z_0} F(x, u, u_x) = |u|^2 u$, and 
$\mN = - \ii |u|^2 u$.

\medskip

For real $s \geq 0$, let $H^s_x := H^s(\T, \C)$ 
be the usual Sobolev space of complex-valued periodic functions $u(x)$,
and let $\| u \|_s := \| u \|_{H^s_x}$ be its norm.
The main result of the paper is the following theorem about the exact, internal controllability of equation \eqref{NLS quasi-lineare}.

\begin{theorem}[Controllability] \label{thm:1}
Let $T>0$, and let $\om \subset \T$ be a nonempty open set. 
There exist positive universal constants $r_1, s_1$, with $r_1 > s_1 > 10$, 
such that, if $G$ in \eqref{regolarita Hamiltoniana} is of class $C^{r_1}$ 
and satisfies \eqref{1506.2}, 
then there exists a positive constant $\d_*$ depending on $T,\om,G$
with the following property. 

Let $u_{in}, u_{end} \in H^{s_1}(\T,\C)$ with 
\begin{equation}\label{smallness dati iniziali teorema}
\| u_{in} \|_{s_1} + \| u_{end} \|_{s_1} \leq \d_*.
\end{equation}
Then there exists a function $f(t,x)$ satisfying
\[
f(t,x) = 0 \quad \text{for all $x \notin \om$, for all $t \in [0,T]$,}
\]
belonging to $C([0,T],H^{s_1}_x) 
\cap C^1([0,T],H^{s_1-2}_x) 
\cap C^2([0,T],H^{s_1-4}_x)$
such that the Cauchy problem
\begin{equation} \label{i9}
\begin{cases}
u_t + \ii u_{xx} + \mN(x,u,u_x, u_{xx}) = f 
\quad \forall (t,x) \in [0,T] \times \T \\
u(0,x) = u_{in}(x) 
\end{cases}
\end{equation}
has a unique solution $u(t,x)$ belonging to 
$C([0,T], H^{s_1}_x) \cap C^1([0,T], H^{s_1-2}_x)
\cap C^2([0,T],H^{s_1-4}_x)$, which satisfies 
\begin{equation} \label{i10}
u(T,x) = u_{end}(x),
\end{equation}
and 
\begin{multline} \label{stimetta}
\| u,f \|_{C([0,T],H^{s_1}_x)} + \| \pa_t u, \pa_t f \|_{C([0,T],H^{s_1-2}_x)} 
+ \| \pa_{tt} u, \pa_{tt} f \|_{C([0,T],H^{s_1-4}_x)} 
\\
\leq C 
(\| u_{in} \|_{s_1} + \| u_{end} \|_{s_1})
\end{multline}
for some $C > 0$ depending on $T,\om,G$. 

Moreover the universal constant $\t_1 := r_1 - s_1 > 0$ has the following property.
For all $r \geq r_1$, all $s \in [s_1, r-\t_1]$, 
if, in addition to the previous assumptions, $G$ is of class $C^{r}$ 
and $u_{in}, u_{end} \in H^{s}_x$, 
then $u,f$ belong to $C([0,T], H^{s}_x) \cap C^1([0,T], H^{s-2}_x)
\cap C^2([0,T],H^{s-4}_x)$ and \eqref{stimetta} holds 
with another constant $C_s$ instead of $C$,
where $C_s > 0$ depends on $s, T,\om,G$. 
\end{theorem}

\begin{remark} \label{rem:piccolezza norma bassa}
Theorem \ref{thm:1} can be seen as split into two parts: first we fix the ``low'' regularity thresholds $s_1,r_1$, which are sufficient to prove the existence of a solution to the control problem. Then, in the last paragraph of the theorem, we give a statement about the higher regularity of such a solution.

Note that the smallness assumption \eqref{smallness dati iniziali teorema} in Theorem \ref{thm:1} is only in the ``low'' norm:  
we only assume $\| u_{in} \|_{s_1} + \| u_{end} \|_{s_1} \leq \d_*$, 
where the constant $\d_*>0$ does not depend on the ``high'' regularity index $s \in [s_1, r-\t_1]$.
\end{remark}

Using the same techniques used for proving Theorem \ref{thm:1}, 
we also prove the following theorem.

\begin{theorem}[Local existence and uniqueness] 
\label{thm:byproduct}
There exist positive universal constants $r_0$, $s_0$ with $r_0 > s_0 > 10$, such that, 
if $G$ in \eqref{regolarita Hamiltoniana} is of class $C^{r_0}$  
and satisfies \eqref{1506.2}, then the following property holds.
For all $T > 0$ there exists $\d_* > 0$ such that
for all $u_{in} \in H^{s_0}(\T, \C)$
satisfying $\| u_{in} \|_{s_0} \leq \d_*$,
the Cauchy problem 
\begin{equation} \label{i11}
\begin{cases}
u_t + \ii u_{xx} + \mN(x,u,u_x, u_{xx}) = 0, 
\qquad (t,x) \in [0,T] \times \T \\
u(0,x) = u_{in}(x) 
\end{cases}
\end{equation}
has one and only one solution 
$u \in C([0,T], H^{s_0}_x) \cap C^1([0,T], H^{s_0-2}_x) \cap C^2([0,T], H^{s_0-4}_x)$.
Moreover
\begin{equation} \label{stimetta bis}
\| u \|_{C([0,T],H^{s_0}_x)} + \| \pa_t u \|_{C([0,T],H^{s_0-2}_x)} 
+ \| \pa_{tt} u \|_{C([0,T],H^{s_0-4}_x)} 
\leq C \| u_{in} \|_{s_0}
\end{equation}
for some $C > 0$ depending on $T,G$. 

The universal constant $\t_0 := r_0 - s_0 > 0$ has the following property.
For all $r \geq r_0$, all $s \in [s_0, r-\t_0]$, 
if, in addition to the previous assumptions, $G$ is of class $C^{r}$ and $u_{in} \in H^{s}(\T, \C)$, 
then $u$ belongs to $C([0,T], H^{s}_x) \cap C^1([0,T], H^{s-2}_x)
\cap C^2([0,T],H^{s-4}_x )$ and \eqref{stimetta bis} holds 
with another constant $C_s$ instead of $C$,
where $C_s > 0$ depends on $s, T,G$. 
\end{theorem}

\subsection{Some related literature}

There is a vast amount of literature concerning controllability for linear or semilinear Schr\"odinger equations. Without even trying to be exhaustive, we only cite some relevant contributions to this subject, starting with the early papers by Jaffard \cite{Jaffard}, 
Lasiecka and Triggiani \cite{Lasiecka-Triggiani} and Lebeau \cite{Lebeau}, which deal with linear Schr\"odinger equations on bounded domains. 
Regarding the one-dimensional case, we mention the result of Beauchard and Coron \cite{Beauchard-Coron} for the controllability of the linear equation by a moving potential well, 
and the papers by Beauchard, Laurent, Rosier and Zhang \cite{Beauchard,Beauchard-Laurent,Laurent,Rosier-Zhang} about controllability of semilinear Schr\"odinger equations. For the semilinear case on compact surfaces, we cite the work by 
Dehman, G\'erard and Lebeau \cite{Dehman-Gerard-Lebeau}. We also mention the recent results 
by Bourgain, Burq and Zworski \cite{Bourgain-Burq-Zworski} 
and by Anantharaman and Maci\`a \cite{Anantharaman-Macia} 
concerning linear Schr\"odinger operators with rough potentials on higher-dimensional tori. 
More references in control theory for Schr\"odinger equations can be found in the detailed surveys by Laurent \cite{Laurent-survey} and Zuazua \cite{Zuazua}.

Concerning controllability theory for quasi-linear PDEs, most known results
deal with first order quasi-linear hyperbolic systems 
of the form $u_t + A(u) u_x = 0$ 
(see, for example, Coron \cite{Coron} chapter 6.2 and the many references therein).
Recent results for different kinds of quasi-linear PDEs 
are contained in Alazard, Baldi and Han-Kwan \cite{ABH} 
on the internal controllability of gravity-capillary water waves equations, 
in Alazard \cite{Alaz1, Alaz2, Alaz3} on the boundary observability and stabilization of gravity and gravity-capillary water waves, and in  Baldi, Floridia and Haus \cite{BFH, BH} 
on the internal controllability of quasi-linear perturbations of the Korteweg-de Vries equation.

\subsection{Strategy of the proof}

Because of the presence of two derivatives in the nonlinearity, 
the controllability of the \emph{quasi-linear} control problem \eqref{i9}-\eqref{i10} 
cannot be directly deduced by a perturbative argument from 
the controllability of the corresponding linear problem  
by applying some fixed point argument or the usual implicit function theorem.
A similar difficulty for a quasi-linear control problem was overcome 
in \cite{ABH} by using a suitable nonlinear iteration scheme adapted to
quasi-linear problems. 
Such a nonlinear scheme requires solving a linear control problem with variable
coefficients at each step of the iteration, 
with no loss of regularity with respect to the coefficients (i.e.,
the solution must have the same regularity as the coefficients). 
In \cite{ABH} this is achieved by means of paradifferential calculus, together with linear transformations, Ingham-type inequalities and the Hilbert uniqueness method.
As an alternative method, in \cite{BFH} it is used a Nash-Moser approach, 
which also demands the solving of a linear control problem with variable coefficients, 
but it requires weaker estimates, allowing some loss of regularity. 
The proof of such weaker estimates is easier to obtain, and it does not require the use of powerful techniques like paradifferential calculus
(for a discussion about pseudo- and paradifferential calculus in connection with the Nash-Moser theorem, see, for example, \cite{H-90}, \cite{AG}). 
The result in \cite{BFH} is slightly weaker than the one in 
\cite{ABH} regarding the regularity of the solution of the nonlinear control problem with respect to the regularity of the data (in \cite{BFH} for data in $H^s(\T)$ 
both the control and the solution are in $C([0,T], H^{s'}(\T))$ for all $s' < s$, 
while the result in \cite{ABH} reaches the corresponding optimal regularity $s' = s$).  
The version of the Nash-Moser implicit function theorem used in \cite{BFH} 
is due to H{\"o}rmander \cite{Geodesy}, and it is the sharpest version in literature 
regarding the loss of regularity in terms of the coefficients of the linearized problem
in several function spaces. %(see \cite{BH} for a discussion).
As it is observed in \cite{BH}, 
the theorem in \cite{Geodesy} is the sharpest possible in H{\"o}lder class, 
but it is not optimal in Sobolev spaces 
(this is the reason for which the optimal regularity $s' = s$ is not obtained in \cite{BFH}).
In \cite{BH} the sharpest H{\"o}rmander's version of the Nash-Moser theorem 
has been extended to Sobolev spaces
(so that $s' = s$ can be obtained both with the Nash-Moser approach
and with the quasi-linear scheme with paradifferential analysis like in \cite{ABH}).
For this reason, in the present paper we use the Nash-Moser theorem in \cite{BH}.

We mention that Nash-Moser schemes in control problems for PDEs 
have been used by Beauchard, Coron, Alabau-Boussouira and Olive in 
\cite{Beauchard,Beauchard-2008,Beauchard-Coron,ACO}. 
A discussion about Nash-Moser as a method to overcome the problem of the loss 
of derivatives in the context of controllability for PDEs
can be found in \cite{Coron}, Section 4.2.2. 
Beauchard and Laurent \cite{Beauchard-Laurent} were able to avoid the use
of the Nash-Moser theorem in semilinear control problems thanks to a regularizing effect. 
%We remark that Theorem \ref{thm:1} could also be proved without Nash-Moser schemes 
%(for example, by adapting the method of \cite{ABH}).

%\medskip

We prove Theorem \ref{thm:1} by applying the Nash-Moser-H\oe{}rmander implicit function theorem of \cite{BH} as a black box. 
To this end, one has to solve the associated linearized control problem 
(see equation \eqref{2706.1}), which is a $2\times2$ real system with variable coefficients at every order, and to prove tame estimates for the solution. 
Like in \cite{ABH,BFH}, we solve the linearized control problem in $L^2 (\T)$ by applying the Hilbert uniqueness method (HUM), see Lemma \eqref{controllabilita cal P2}. Then, in Lemma \eqref{regolarita H2 per cal P2}, we recover the additional regularity of the solution by adapting a method 
of Dehman-Lebeau \cite{Dehman-Lebeau}, also used by Laurent \cite{Laurent} and in \cite{ABH, BFH}. 
To apply the HUM method, we prove in Section \ref{sezione osservabilita} the observability of the linearized operator in \eqref{linearized operator} by a procedure of symmetrization and reduction to constant coefficients up to a bounded remainder (like in \cite{ABH,BFH}) developed in Section \ref{sezione riduzione linearizzato}; then the result follows by applying Ingham inequality (with a further simple argument 
to deal with double eigenvalues, like in \cite{ABH}). 
The procedure of symmetrization and reduction of the linearized operator is an adaptation of the one used by Feola and Procesi \cite{Feola-Procesi,Feola} in the context of KAM theory for quasi-linear NLS equations. 
We remark that a similar reduction procedure has been also developed in 
\cite{Ioo-Plo-Tol}, \cite{Baldi-Benj-Ono}, 
\cite{BBM-Airy}, \cite{BBM-auto}, \cite{BBM-mKdV}, 
\cite{AB}, \cite{ABH}, \cite{BertiMontalto}, \cite{Montalto} 
for water waves, quasi-linear KdV, Benjamin-Ono and Kirchhoff equations.

\subsection{Functional setting and the linearized problem}
\label{sezione functional setting} 
Given any open subset $\om \subset \T$, 
we introduce a function $\chi_\om \in C^\infty(\T, \R)$ whose support is contained in $\om$, 
such that $0 \leq \chi_\om(x) \leq 1$ for all $x \in \T$, 
and $\chi_\om = 1$ on some open interval contained in $\om$. 
We write the NLS control problem as a real system,
namely, writing $u  = \frac{1}{\sqrt{2}}(u_1 + \ii u_2)$, $f = \frac{1}{\sqrt{2}} (f_1 + \ii f_2)$, 
with $u_1, u_2, f_1, f_2$ all real-valued functions,
the control problem \eqref{i9}-\eqref{i10} becomes the one of finding 
$(f_1, f_2)$ such that the solution $(u_1, u_2)$ of the Cauchy problem
\begin{equation}\label{problema controllo in coordinate reali}
\begin{cases}
\partial_t u_1 + \nabla_{u_2} H(u_1, u_2) = \chi_\omega f_1 \\
\partial_t u_2 - \nabla_{u_1} H(u_1, u_2) = \chi_\omega f_2 \\
u_1 (0, \cdot ) = (u_1)_{in} \\
u_2 (0, \cdot) = (u_2)_{in}
\end{cases} 
\quad \text{satisfies} \quad 
\begin{cases}
u_1(T, \cdot) = (u_1)_{end} \\ 
u_2(T, \cdot) = (u_2)_{end}
\end{cases} 
\end{equation}
where the real Hamiltonian $H$ is defined in \eqref{2306.2}. 
We define 
\begin{equation} \label{2406.1}
P(u_1, u_2) := \begin{pmatrix} \partial_t u_1 + \nabla_{u_2} H(u_1, u_2) 
\\ 
\partial_t u_2 - \nabla_{u_1} H(u_1, u_2) 
\end{pmatrix}, 
\quad 
\chi_\om (f_1, f_2) := \begin{pmatrix} \chi_\om f_1 \\ \chi_\om f_2
\end{pmatrix}, 
\end{equation}
and 
\begin{equation} \label{2406.2}
\Phi(u_1, u_2, f_1, f_2) := \begin{pmatrix} 
P(u_1, u_2) - \chi_\om (f_1, f_2) 
\\
(u_1, u_2)(0,\cdot) 
\\
(u_1, u_2)(T,\cdot) 
\end{pmatrix},
\quad 
z_{data} := \begin{pmatrix} 
0 \\ ((u_1)_{in} , (u_2)_{in})
\\ ((u_1)_{end}, (u_2)_{end}) 
\end{pmatrix},
\end{equation}
so that problem \eqref{problema controllo in coordinate reali} reads 
\begin{equation} \label{2406.3}
\Phi(u_1, u_2, f_1, f_2) = z_{data}.	
\end{equation}
By \eqref{2406.1} and \eqref{2306.2}, the nonlinear operator $P$ is given by
\begin{equation} \label{2406.5}
P(u_1, u_2) = \begin{pmatrix}	
\pa_t u_1 - \pa_{xx} u_2 + (\pa_{y_2} G)(x, u_1, u_2, (u_1)_x, (u_2)_x ) 
- \pa_x \{ (\pa_{y_4} G)(x, u_1, u_2, (u_1)_x, (u_2)_x) \} 
\vspace{2pt} \\
\pa_t u_2 + \pa_{xx} u_1 - (\pa_{y_1} G)(x, u_1, u_2, (u_1)_x, (u_2)_x) 
+ \pa_x \{ (\pa_{y_3} G)(x, u_1, u_2, (u_1)_x, (u_2)_x) \} 
\end{pmatrix} \! .
\end{equation}

The crucial assumption to verify in order to apply the Nash-Moser theorem is the existence of a right inverse of the linearized operator. 
The linearized operator $\Phi'(u_1, u_2, f_1, f_2)[h_1, h_2, \ph_1, \ph_2]$ 
at the point $(u_1, u_2, f_1, f_2)$ 
in the direction $(h_1, h_2, \ph_1, \ph_2)$ is given by
\begin{equation} \label{2406.4}
\Phi'(u_1, u_2, f_1, f_2)[h_1, h_2, \ph_1, \ph_2] = 
\begin{pmatrix} 
P'(u_1, u_2)[h_1 , h_2] - \chi_\om (\ph_1, \ph_2) 
\\
(h_1 , h_2)(0,\cdot) 
\\
(h_1, h_2)(T,\cdot) 
\end{pmatrix}.
\end{equation}
Thus we have to prove that, given any $(u_1, u_2, f_1, f_2)$ 
and any $z = (v_1, v_2, \a_1, \a_2, \b_1, \b_2)$ in a suitable function space, 
there exists $(h_1, h_2, \ph_1, \ph_2)$ such that 
\begin{equation} \label{2706.1}
\Phi'(u_1, u_2, f_1, f_2)[h_1, h_2, \ph_1, \ph_2] = z
\end{equation}
(i.e., we have to solve the linearized control problem).
The linearized operator $P'(u_1, u_2)[h_1, h_2]$ is
\begin{align} \label{2406.6}
& P'(u_1, u_2)[h_1, h_2]
\\
& = \begin{pmatrix}
\pa_t h_1 - \pa_{xx} h_2 
+ p_2^{(11)} \pa_{xx} h_1 
+ p_2^{(12)} \pa_{xx} h_2
+ p_1^{(11)} \pa_{x} h_1
+ p_1^{(12)} \pa_{x} h_2
+ p_0^{(11)} h_1  
+ p_0^{(12)} h_2
\vspace{2pt} \\
\pa_t h_2 + \pa_{xx} h_1
+ p_2^{(21)} \pa_{xx} h_1
+ p_2^{(22)} \pa_{xx} h_2 
+ p_1^{(21)} \pa_{x} h_1
+ p_1^{(22)} \pa_{x} h_2
+ p_0^{(21)} h_1 
+ p_0^{(22)} h_2
\end{pmatrix},
\notag 
\end{align}
namely
\begin{equation}  \label{2406.7}
\bigg\{ 
\pa_t + J \pa_{xx} 
+ \begin{pmatrix} 
p_2^{(11)} & p_2^{(12)} \\ 
p_2^{(21)} & p_2^{(22)} \end{pmatrix}
\pa_{xx} 
+ \begin{pmatrix} 
p_1^{(11)} & p_1^{(12)} \\ 
p_1^{(21)} & p_1^{(22)} \end{pmatrix}
\pa_{x} 
+ \begin{pmatrix} 
p_0^{(11)} & p_0^{(12)} \\ 
p_0^{(21)} & p_0^{(22)} \end{pmatrix} \bigg\}
\begin{pmatrix} h_1 \\ h_2 \end{pmatrix}
\end{equation}
where the coefficients of the terms of order 2 are
\begin{alignat}{2} \label{pp.2}
p_2^{(11)}
& = - (\pa_{y_3 y_4} G), 
\qquad &
p_2^{(12)}
& = - (\pa_{y_4 y_4} G),
\\
p_2^{(21)}
& = (\pa_{y_3 y_3} G),
&
p_2^{(22)}
& = (\pa_{y_3 y_4} G),
\notag
\end{alignat}
those of order 1 are
\begin{alignat}{2} \label{pp.1}
p_1^{(11)}
& = (\pa_{y_2 y_3} G) - (\pa_{y_1 y_4} G) - \pa_x \{ (\pa_{y_3 y_4} G) \},
\qquad &
p_1^{(12)}
& = - \pa_x \{ (\pa_{y_4 y_4} G) \},
\\
p_1^{(21)}
& = \pa_x \{ (\pa_{y_3 y_3} G) \}, 
&
p_1^{(22)}
& = - (\pa_{y_1 y_4} G) + (\pa_{y_2 y_3} G) + \pa_x \{ (\pa_{y_3 y_4} G) \}, 
\notag
\end{alignat}
those of order 0 are
\begin{alignat}{2} \label{pp.0}
p_0^{(11)} 
& = (\pa_{y_1 y_2} G) - \pa_x \{ (\pa_{y_1 y_4} G) \}, 
\qquad &
p_0^{(12)}
& = (\pa_{y_2 y_2} G) - \pa_x \{ (\pa_{y_2 y_4} G) \}, 
\\
p_0^{(21)} 
& = - (\pa_{y_1 y_1} G) + \pa_x \{ (\pa_{y_1 y_3} G) \}, 
\qquad &
p_0^{(22)}
& = - (\pa_{y_1 y_2} G) + \pa_x \{ (\pa_{y_2 y_3} G) \},
\notag
\end{alignat}
and $(\pa_{y_i y_j}G) = (\pa_{y_i y_j}G)(x,u_1, u_2, \pa_x u_1, \pa_x u_2)$ 
for all $i,j \in \{ 1,2,3,4 \}$. 

Consider the transformation 
\begin{equation} \label{2806.5}
\begin{pmatrix} h_1 \\ h_2 \end{pmatrix} 
= {\cal C} \begin{pmatrix} h \\ \bar h \end{pmatrix}, 
\quad
\text{where}
\quad
{\cal C} := \frac{1}{\sqrt{2}} \begin{pmatrix} 1 & 1 \\ -i & i \end{pmatrix}, 
\quad 
{\cal C}^{-1} = \frac{1}{\sqrt{2}} \begin{pmatrix} 1 & i \\ 1 & -i \end{pmatrix},
\end{equation}
and similarly 
$(\ph_1, \ph_2) = {\cal C} (\ph, \bar \ph)$, 
$(v_1, v_2) = {\cal C} (v, \bar v)$, 
$(\a_1, \a_2) = {\cal C} (\a, \bar \a)$, 
$(\b_1, \b_2) = {\cal C} (\b, \bar \b)$.
With this ``vector complex'' notation, 
the linearized control problem \eqref{2706.1} becomes 
\begin{equation} \label{2706.10}
\begin{cases}
\mL [h, \bar h] - \chi_\om (\ph, \bar \ph) = (v, \bar v) \\
(h, \bar h)(0,\cdot) = (\a, \bar \a) \\
(h, \bar h)(T,\cdot) = (\b, \bar \b)
\end{cases}
\end{equation}
where $\mL := \mL(u_1, u_2) := \mC^{-1} P'(u_1, u_2) \mC$. 
More explicitly, we calculate 
\begin{equation} \label{linearized operator}
\mL
= \partial_t {\mathbb I}_2 + \ii (\Sigma + A_2) \partial_{xx} + \ii A_1 \partial_x + \ii A_0\,,
\end{equation}
where
\begin{equation}\label{2706.12} 
\mathbb{I}_2 := \begin{pmatrix} 1 & 0 \\ 0 & 1 \end{pmatrix}, \quad 
\Sigma := \begin{pmatrix} 1 & 0 \\ 0 & - 1 \end{pmatrix}, \quad 
A_k := \begin{pmatrix}
a_k & b_k\\
- \overline b_k & - \overline a_k
\end{pmatrix}, \quad k = 0,1, 2,
\end{equation}
\begin{equation} \label{2706.13}
a_k := \frac12 \Big( - \ii p_k^{(11)} - p_k^{(12)} + p_k^{(21)} - \ii p_k^{(22)} \Big), 
\quad 
b_k := \frac12 \Big( - \ii p_k^{(11)} + p_k^{(12)} + p_k^{(21)} + \ii p_k^{(22)} \Big),
\end{equation}
and $\bar a_k, \bar b_k$ are the complex conjugates of the coefficients $a_k, b_k$. 
By \eqref{2706.13} and \eqref{pp.2}, \eqref{pp.1}, \eqref{pp.0}, 
one has  
\begin{equation} \label{2806.1} 
a_2 = \bar a_2, \quad
a_1 = 2 \pa_x a_2 - \bar a_1, \quad
a_0 = \bar a_0 + \pa_{xx} a_2 - \pa_x \bar a_1, \quad 
b_1 = \pa_x b_2.
\end{equation}

\begin{remark}
The linear system \eqref{2706.10} is made by three pairs of equations in which the second equation is the complex conjugate of the first one. 
Hence \eqref{2706.10} is equivalent to 
\begin{equation} \label{2806.2}
\begin{cases}
\mL^{(sca)} h - \chi_\om \ph = v \\
h(0,\cdot) = \a \\
h(T,\cdot) = \b
\end{cases}	
\end{equation}
where 
\begin{equation} \label{2806.3}
\mL^{(sca)} := \pa_t + \ii (1 + a_2 + b_2 \cc) \pa_{xx} + \ii (a_1 + b_1 \cc) \pa_x 
+ \ii (a_0 + b_0 \cc), 
\qquad 
\cc[h] := \bar h.
\end{equation}
The complex conjugate operator $\cc : h \mapsto \bar h$ is $\R$-linear, 
and there is no problem in using it 
to shorten the notation of the real system \eqref{2706.1}.

However, instead of the \emph{scalar complex} notation \eqref{2806.2},
in the analysis of the linearized problem we will use the \emph{vector complex} notation \eqref{2706.10}, which is somewhat ``more natural'' and very common in the literature on the Schr\oe{}dinger equation. 
In any case, for linear systems the two notations are, of course, completely equivalent.
\end{remark}

For real $s \geq 0$, 
we consider the classical Sobolev space  
$$
H^s(\T) := H^s(\T, \C) 
:= \Big\{ u \in L^2(\T,\C) : \| u \|_{s}^2 := \sum_{k \in \Z} \langle k \rangle^{2 s} 
|\widehat u_k|^2 < \infty \Big\}\,,
$$
where $\langle k \rangle := (1 + |k|^2)^{\frac12}$ and 
$u(x) = \sum_{k \in \Z} \widehat u_k \, e^{\ii k x} \in L^2(\T) := L^2(\T,\C)$. 
We adopt the convention of indicating explicitly $H^s(\T,\R)$ the subspace of \emph{real-valued} functions of $H^s(\T,\C)$, and to denote, in short, by $H^s(\T)$ the whole space $H^s(\T,\C)$. 
The same convention applies to $L^2(\T,\R)$ and $L^2(\T) := L^2(\T, \C)$. 
We also consider spaces 
$H^s(\T, \mathbb K^2)$, where $\mathbb K = \R, \C$, 
and for $(u_1, u_2) \in H^s(\T, \mathbb K^2)$ we set 
$$
\| (u_1, u_2)\|_s := \| u_1\|_s + \| u_2\|_s\,. 
$$
We define the \emph{real subspace} ${\bf H}^s(\T)$ of $H^s(\T, \C^2)$ as 
\begin{equation}\label{definizione bf Hs}
{\bf H}^s(\T) := \big\{ {\bf u} = (u, \overline u) : u \in H^s(\T,\C) \big\}
\end{equation}
where $\bar u$ is the complex conjugate of $u$. 
When there is no ambiguity, we also write, in short, 
$H^s_x$ to denote $H^s(\T,\C)$ or $H^s(\T,\R^2)$, 
and the same for $L^2_x$, ${\bf H}^s_x$ and ${\bf L}^2_x$. 

We denote by $\langle \cdot, \cdot \rangle_{L^2}$ 
the standard $L^2$ scalar product in $L^2(\T,\C)$, namely 
\begin{equation} \label{0109.2}
\langle u, v \rangle_{L^2} := \int_\T u(x) \overline v(x)\, d x 
\quad \forall u, v \in L^2(\T,\C). 
\end{equation}
We define the scalar product in $L^2(\T, \R^2)$ as 
\begin{equation} \label{0109.1}
\langle(u_1, u_2), (v_1, v_2) \rangle_{L^2(\T, \R^2)} 
:= \int_\T u_1(x) v_1(x)\, d x + \int_\T u_2(x) v_2(x)\,d x,
\end{equation}
and the scalar product in ${\bf L}^2(\T)$ as 
\begin{equation}\label{prodotto scalare sottospazio reale}
\langle {\bf u}, {\bf v} \rangle_{{\bf L}^2} := \int_\T u(x) \overline v(x)\, d x + \int_\T v(x) \overline u(x)\, d x\,.
\end{equation}
Note that \eqref{prodotto scalare sottospazio reale} 
is a real scalar product on ${\bf L}^2(\T)$, 
and therefore $({\bf L}^2(\T), \langle \cdot, \cdot \rangle_{{\bf L}^2})$ 
is a real Hilbert subspace of $L^2(\T, \C^2)$. 

The transformation $\mC$ defined in \eqref{2806.5} satisfies 
\begin{equation} \label{0109.3}
\langle {\bf u} , {\bf v} \rangle_{{\bf L}^2} 
= \langle \mC \mathbf{u}, \mC \mathbf{v} \rangle_{L^2(\T, \R^2)} 
\quad \forall \mathbf{u}, \mathbf{v} \in \mathbf{L}^2(\T),
\end{equation}
and so $\mC$ is a unitary isomorphism between 
the real Hilbert space $L^2(\T, \R^2)$ 
equipped with the real scalar product \eqref{0109.1}
and 
the real Hilbert space ${\bf L}^2(\T)$
equipped with the scalar product \eqref{prodotto scalare sottospazio reale}.

Given a linear operator $R : L^2(\T,\C) \to L^2(\T,\C)$, 
we define the \emph{adjoint} operator $R^*$ as 
\begin{equation}\label{definizione operatore aggiunto}
\langle R u, v \rangle_{L^2} = \langle u, R^* v \rangle_{L^2} 
\quad \forall u, v \in L^2(\T,\C);
\end{equation}
the \emph{transpose} operator $R^T$ as 
\begin{equation}\label{definizione operatore trasposto}
\int_\T (R u) v\, d x = \int_\T u (R^T v)\, d x 
\quad \forall u, v \in L^2(\T,\C);
\end{equation}
and the \emph{conjugate} operator $\overline R$ as 
\begin{equation}\label{definizione operatore conj}
\overline R u = \overline{(R \bar u)}
\quad \forall u \in L^2(\T,\C).
\end{equation}
For an operator 
$$
{\cal R} := \begin{pmatrix}
A & B \\
\overline B & \overline A
\end{pmatrix} : {\bf L}^2(\T) \to {\bf L}^2(\T)\,,
$$
we define its adjoint ${\cal R}^*$ by 
\begin{equation}\label{definizione aggiunto bf L2}
\langle {\cal R} {\bf u}, {\bf v} \rangle_{{\bf L}^2} = \langle{\bf u}, {\cal R}^* {\bf v} \rangle_{{\bf L}^2} \quad \forall {\bf u}, {\bf v} \in {\bf L}^2(\T),
\end{equation}
namely
\begin{equation} \label{2107.2}
{\cal R}^* = \begin{pmatrix}
(\overline A)^T & B^T \\
(\overline B)^T & A^T
\end{pmatrix}
=  \begin{pmatrix}
A^* & B^T \\
B^* & A^T
\end{pmatrix}
= \begin{pmatrix}
A^* & B^T \\
\overline{B^T} & \overline{A^*}
\end{pmatrix}.
\end{equation}
For any real $s \geq 0$ and ${\bf u} = (u, \overline u) \in {\bf H}^s(\T)$, we set 
\begin{equation}\label{norma s u bf u}
\| {\bf u}\|_s := \| u \|_s \,. 
\end{equation}
Given a Banach space $(X, \| \cdot \|_X)$, and $T > 0$, 
we consider the space $C([0, T], X)$ of the continuous functions $u : [0, T] \to X$ 
equipped with the sup-norm 
\begin{equation}\label{definizione norma CT(X)}
\| u \|_{C([0,T],X)}
:= \| u \|_{C_T(X)} 
:= \sup_{t \in [0, T]} \| u(t)\|_X\,. 
\end{equation}
For $X = H^s(\T,\R)$ or $H^s(\T, \R^2)$ or $H^s(\T,\C)$ or $H^s(\T,\C^2)$ or $\mathbf{H}^s(\T)$, 
and $u \in C([0,T], X)$, we denote, in short, 
\begin{equation} \label{2107.3}
\| u \|_{T, s} 
:= \sup_{t \in [0, T]} \| u(t)\|_s \,.
\end{equation}
We also define the following notations. Given a Sobolev index $s \geq 0$, 
we write $A \lesssim_s B$ if there exists a constant $C(s) > 0$ depending on $s$ 
such that $A \leq C(s) B$. 
If the constant $C(s)$ is independent of $s$, we simply write $A \lesssim B$.

According to \eqref{problema controllo in coordinate reali}-\eqref{2406.5}, 
Theorem \ref{thm:1} follows from the following theorem.

\begin{theorem}\label{teorema controllo in coordinate reali}
Let $T>0$, and let $\om \subset \T$ be a nonempty open set. 
Let $\chi_\omega$ be a $C^\infty$ function supported in $\omega$, 
with $0 \leq \chi_\om \leq 1$ on $\T$ 
and $\chi_\om = 1$ on some open interval contained in $\om$. 
There exist positive universal constants $r_1, s_1$ such that, 
if $G$ in \eqref{regolarita Hamiltoniana} is of class $C^{r_1}$  
and satisfies \eqref{1506.2}, 
then there exists a positive constant $\d_*$ depending on $T,\om, G$
with the following property. 
Let $(u_1)_{in}, (u_1)_{end}, (u_2)_{in}, (u_2)_{end} \in H^{s_1}(\T,\R)$ with 
\[
\| (u_i)_{in} \|_{s_1} + \| (u_i)_{end} \|_{s_1} \leq \d_*, \quad i = 1, 2\,. 
\] 
Then there exist functions 
$$
f_1, f_2 \in C([0, T], H^{s_1}(\T, \R)) \cap C^1([0, T], H^{s_1 - 2}(\T, \R)) \cap C^2([0, T], H^{s_1 - 4}(\T, \R))
$$ 
such that the Cauchy problem
\begin{equation} \label{pb cauchy statement controllo reale}
\begin{cases}
\partial_t u_1 + \nabla_{u_2} H(u_1, u_2) = \chi_\omega f_1 \\
\partial_t u_2 - \nabla_{u_1} H(u_1, u_2) = \chi_\omega f_2 \\
u_1 (0, \cdot ) = (u_1)_{in} \\
u_2 (0, \cdot) = (u_2)_{in}
\end{cases} 
\end{equation}
has a unique solution $(u_1, u_2)$ with $$u_1, u_2 \in C([0, T], H^{s_1}(\T, \R)) \cap C^1([0, T], H^{s_1 - 2}(\T, \R)) \cap C^2([0, T], H^{s_1 - 4}(\T, \R))\,,$$ which satisfies 
\begin{equation} \label{i101}
u_1(T,x) = (u_1)_{end}(x), \qquad u_2(T, x) = (u_2)_{end}(x)
\end{equation}
and for $i = 1, 2$
\begin{multline} \label{stimetta2}
\| u_i, f_i\|_{T, s_1} + \| \partial_t u_i, \partial_t f_i\|_{T, s_1 - 2} + \| \partial_{tt} u_i, \partial_{tt} f_i\|_{T, s_1 - 4} 
\\
\leq C (\| (u_1)_{in}, (u_2)_{in} \|_{s_1} + \| (u_1)_{end}, (u_2)_{end} \|_{s_1})
\end{multline}
for some $C > 0$ depending on $T,\om,G$. 

Moreover the universal constant $\t_1 := r_1 - s_1 > 0$ has the following property.
For all $r \geq r_1$, all $s \in [s_1, r-\t_1]$, 
if, in addition to the previous assumptions, $G$ is of class $C^{r}$ 
and $(u_1)_{in}$, $(u_2)_{in}$, $(u_1)_{end}$, $(u_2)_{end} \in H^{s}(\T, \R)$, 
then $u,f$ belong to 
$C([0,T], H^{s}(\T, \R)) \cap C^1([0,T], H^{s-2}(\T, \R)) 
\cap C^2([0,T],H^{s-4}(\T, \R))$ and \eqref{stimetta2} holds 
with another constant $C_s$ instead of $C$,
where $C_s > 0$ depends on $s, T,\om,G$. 
\end{theorem}

Similarly, Theorem \ref{thm:byproduct} follows from the following theorem.

\begin{theorem}\label{thm:byproduct real}
Let $T>0$. There exist positive universal constants $r_0, s_0$ such that, 
if $G$ in \eqref{regolarita Hamiltoniana} is of class $C^{r_0}$ in its arguments 
and satisfies \eqref{1506.2}, 
then there exists a positive constant $\d_*$ depending on $T, G$
with the following property. 
Let $(u_1)_{in}, (u_2)_{in} \in H^{s_0}(\T,\R)$ with 
\[
\| (u_1)_{in} \|_{s_0} + \| (u_2)_{in} \|_{s_0} \leq \d_*\,. 
\] 
Then the Cauchy problem
\begin{equation} \label{pb cauchy statement controllo reale pappa}
\begin{cases}
\partial_t u_1 + \nabla_{u_2} H(u_1, u_2) = 0\\
\partial_t u_2 - \nabla_{u_1} H(u_1, u_2) = 0 \\
u_1 (0, \cdot ) = (u_1)_{in} \\
u_2 (0, \cdot) = (u_2)_{in}
\end{cases} 
\end{equation}
has a unique solution $(u_1, u_2)$ with 
$$
u_1, u_2 \in C([0, T], H^{s_0}(\T, \R)) \cap C^1([0, T], H^{s_0 - 2}(\T, \R)) \cap C^2([0, T], H^{s_0 - 4}(\T, \R))
$$ 
and 
\begin{align} \label{stimetta2 pappa}
\| u_i\|_{T, s_0} + \| \partial_t u_i \|_{T, s_0 - 2} + \| \partial_{tt} u_i \|_{T, s_0 - 4} 
\leq C (\| (u_1)_{in} \|_{s_0} + \| (u_2)_{in} \|_{s_0})\,, \quad i = 1, 2
\end{align}
for some $C > 0$ depending on $T, G$. 

Moreover the universal constant $\t_0 := r_0 - s_0 > 0$ has the following property.
For all $r \geq r_0$, all $s \in [s_0, r-\t_0]$, 
if, in addition to the previous assumptions, $G$ is of class $C^{r}$ 
and $(u_1)_{in}, (u_2)_{in} \in H^{s}(\T, \R)$, 
then $u$ belongs to $C([0,T], H^{s}(\T, \R)) \cap C^1([0,T], H^{s-2}(\T, \R))
\cap C^2([0,T],H^{s-4}(\T, \R))$ and \eqref{stimetta2 pappa} holds 
with another constant $C_s$ instead of $C$,
where $C_s > 0$ depends on $s, T,G$. 
\end{theorem}

\section{Reduction of the linearized operator}\label{sezione riduzione linearizzato}
\label{riduzione operatori lineari generali}

In view of the application of the Nash-Moser scheme, we will consider linear operators of the same form as ${\cal L} = {\cal L}(u_1, u_2)$ given in \eqref{linearized operator}. 
The aim of this section is to conjugate such operators to constant coefficients up to a bounded remainder, adapting the procedure described in \cite{Feola, Feola-Procesi}. We first fix some notation.

Let $u_1, u_2 \in C^0 \big( [0, T], H^{s + 4}(\T, \R) \big)  \cap C^1\big( [0, T], H^{s + 2}(\T, \R) \big) \cap C^2 \big( [0, T], H^{s}(\T, \R) \big)$. We define
\begin{equation}\label{regolarita u}
M_T(s ; u_1, u_2) := \max_{k = 1, 2} \, \sup_{t \in [0, T]} \big( \| u_k(t, \cdot)\|_{s + 4} + \| \partial_t u_k(t, \cdot) \|_{s + 2} + \| \partial_{tt} u_k(t, \cdot) \|_{s}\big)\,.
\end{equation}
We recall the notation defined in \eqref{2107.3}: 
given a function $v \in C \big( [0, T], H^s (\T,\R) \big)$, we denote 
$\| v \|_{T, s} := \sup_{t \in [0, T]} \| v(t, \cdot)\|_{H^s_x}$.
Also, if ${\bf v} = (v, \overline v) \in C\big( [0, T], {\bf H}^s(\T) \big)$, we set 
$$
\| {\bf v}\|_{T, s} := \| v \|_{T, s}\,.
$$
In the next Lemma we provide some estimates on the coefficients $a_i$, $b_i$, $i = 0,1,2$.  

\begin{lemma}\label{stime coefficienti linearizzato}
Let $r \geq 6$ be the regularity of $G$ in \eqref{regolarita Hamiltoniana}.
There exists $\d > 0$, depending on $G$, such that, 
if $M_T(2 ; u_1, u_2)$ defined in \eqref{regolarita u} satisfies 
\begin{equation} \label{2107.1}
M_T(2 ; u_1, u_2) \leq \delta,
\end{equation}
then for every $s \in [0, r - 6]$ one has 
$$
\| a_i\|_{T, s} \,,\, \| \partial_t a_i \|_{T, s}, \| \partial_{tt } a_i\|_{T, s}, \| b_i\|_{T, s} \,,\, \| \partial_t b_i \|_{T, s}, \| \partial_{tt } b_i\|_{T, s}
\lesssim_s M_T(s+2 ; u_1, u_2)\,.
$$
\end{lemma}

\begin{proof}
The estimates follow from the explicit expressions given in \eqref{2706.13}, \eqref{pp.2}-\eqref{pp.0} and by the composition Lemma \ref{lemma Moser}. 
\end{proof}

We consider operators of the form 
\begin{equation}\label{operatore lineare generale}
{\cal L} := \partial_t {\mathbb I}_2 + \ii (\Sigma + A_2) \partial_{xx} + \ii A_1 \partial_x + \ii A_0\,,
\end{equation}
where 
\begin{equation}\label{Sigma Ak}
\Sigma = \begin{pmatrix}
1 & 0 \\
0 & - 1
\end{pmatrix}\,, \quad A_k := \begin{pmatrix}
a_k & b_k\\
- \overline b_k & - \overline a_k
\end{pmatrix}\,, \quad k = 0,1, 2\,.
\end{equation}
We assume that the time dependent vector field $L(t) := \ii A_2 \partial_{xx} + \ii A_1 \partial_x + \ii A_0$ is Hamiltonian, therefore equations \eqref{2806.1} hold by Lemma \ref{caratterizzazione campi lineari Hamiltoniani ordine 2}.
We assume that for $S \in \N$ large enough 
\begin{equation}\label{regolarita coefficienti operatore da ridurre}
a_2, \partial_{t} a_2, \partial_{tt} a_2, b_2, \partial_t b_2, a_1, \partial_t a_1, b_1, \partial_t b_1, a_0, b_0 \in C\big([0, T], H^{S}(\T) \big) \,,
\end{equation}
and, for $s \in [0,S]$, we set 
\begin{align}
N_T(s) & := \sup_{t \in [0, T]} \max \{ \| a_2 \|_{H^s}, \| \partial_{t} a_2 \|_{H^s}, \| \partial_{tt} a_2 \|_{H^s}, \| a_1 \|_{H^s}, \|\partial_t a_1 \|_{H^s},\| a_0 \|_{H^s} \}  \nonumber\\
& \quad + \sup_{t \in [0, T]} \Big( \| b_2 \|_{H^s}, \| \partial_t b_2\|_{H^s}, \| b_1\|_{H^s}, \| \partial_t b_1\|_{H^s}, \| b_0 \|_{H^s} \Big)\,. \label{definizione NT sigma}
\end{align}
In Sections \ref{riduzione operatori lineari generali}, \ref{osservabilita operatori lineari generali}, we will consider constants $\s, S$, 
with $0 < \sigma < S$, and $\eta \in (0, 1)$, 
and assume that  
\begin{equation}\label{ansatz operatore astratto}
N_T(\sigma) \leq \eta\,. 
\end{equation}
The constant $S$ will have the role of a large and fixed regularity index, 
$\s$ will indicate the ``loss of regularity'' in terms of the coefficients 
of the linearized operator, and $\eta$ will be small enough.

\subsection{Symmetrization of ${\cal L}$ up to order zero}\label{coniugio step 1}
In this subsection we remove the off-diagonal terms from the order 2,
namely we conjugate the linear operator $\mL$ in \eqref{operatore lineare generale} 
to an operator $\mL_0$
(see \eqref{cal L1}-\eqref{A2 (1)}) where the coefficient in front of $\pa_{xx}$ 
is a \emph{diagonal} $2 \times 2$ matrix. 
As a consequence of the Hamiltonian structure, the transformation that achieves this cancellation also removes the off-diagonal terms from the order 1 (see equation \eqref{termine off diagonal ordine 1 = 0}). First we consider the $2 \times 2$ matrix valued function 
$$
\Sigma + A_2(t, x) = \begin{pmatrix}
1 + a_2(t, x) & b_2 \\
- \overline b_2 & - 1 - a_2(t, x)
\end{pmatrix}
$$
(recall that $a_2 = \overline a_2$ by Lemma \ref{caratterizzazione campi lineari Hamiltoniani ordine 2}). The eigenvalues of the above matrix are given by $\pm \lambda(t, x)\in\R$, where  
\begin{equation}\label{definizione lambda}
\lambda(t, x) : =  \sqrt{(1 + a_2)^2 - |b_2|^2}\,.
\end{equation}
Note that, by Sobolev embedding, \eqref{ansatz operatore astratto} and because $\sigma \geq 1$,
one has 
$$
\| a_2\|_{L^\infty} + \| b_2\|_{L^\infty} \lesssim \| a_2\|_{T, 1} + \| b_2\|_{T, 1} 
\lesssim \eta\,,
$$ 
so that $(1 + a_2)^2 - |b_2|^2$ is close to $1$ for $\eta \in (0, 1)$ small enough. 
Then we consider the $2 \times 2$ matrix
\begin{equation}\label{matrice autovettori}
{\cal S} = {\cal S}(t, x) := \begin{pmatrix}
\dfrac{1 + a_2 + \lambda}{\sqrt{(1 + a_2 + \lambda)^2 - |b_2|^2}} & - \dfrac{ b_2}{\sqrt{(1 + a_2 + \lambda)^2 - |b_2|^2}} \\
- \dfrac{ \overline b_2}{\sqrt{(1 + a_2 + \lambda)^2 - |b_2|^2}} & \dfrac{1 + a_2 + \lambda}{\sqrt{(1 + a_2 + \lambda)^2 - |b_2|^2}}
\end{pmatrix}\,.
\end{equation}
The columns of the matrix ${\cal S}$ are the eigenvectors corresponding to the eigenvalues $\pm \lambda$ and ${\rm det}({\cal S}(t, x)) = 1$. Then the map 
$$
{\cal S}(t) :  \quad {\bf h}(x) \mapsto {\cal S}(t, x) {\bf h}(x)
$$
is symplectic. The above matrix is invertible and its inverse is given by 
\begin{equation}\label{matrice autovettori inversa}
{\cal S}^{- 1} = {\cal S}^{- 1}(t, x) := \begin{pmatrix}
\dfrac{1 + a_2 + \lambda}{\sqrt{(1 + a_2 + \lambda)^2 - |b_2|^2}} &  \dfrac{ b_2}{\sqrt{(1 + a_2 + \lambda)^2 - |b_2|^2}} \\
\dfrac{ \overline b_2}{\sqrt{(1 + a_2 + \lambda)^2 - |b_2|^2}} & \dfrac{1 + a_2 + \lambda}{\sqrt{(1 + a_2 + \lambda)^2 - |b_2|^2}}
\end{pmatrix}\,
\end{equation}
and a direct calculation shows that 
\begin{equation}\label{cal S inverso aggiunto}
{\cal S}= {\cal S}^*\,.
\end{equation}
We compute the conjugation ${\cal S}^{- 1} {\cal L} {\cal S}$. 
Note that  
\begin{equation}\label{definizione a2 (1)}
{\cal S}^{- 1}(\Sigma + A_2) {\cal S} = \begin{pmatrix}
\lambda & 0 \\
0 & - \lambda
\end{pmatrix} =  \begin{pmatrix}
1 + a_2^{(0)} & 0 \\
0 & - 1 - a_2^{(0)}
\end{pmatrix}\,, \quad a_2^{(0)} := \lambda - 1\in\R \,
\end{equation}
and we get the linear operator 
\begin{align}
{\cal L}_0 & := {\cal S}^{- 1}{\cal L} {\cal S} = \partial_t {\mathbb I}_2 + \ii (\Sigma + {A_2^{(0)}} ) \partial_{xx} + \ii {A_1^{(0)}} \partial_x + \ii {A_0^{(0)}} \,, \label{cal L1}
\end{align}
where 
\begin{align}
& A_2^{(0)} :=  \begin{pmatrix}
a_2^{(0) } & 0 \\
0 & - a_2^{(0)}
\end{pmatrix}\,, \label{A2 (1)} \\
& A_1^{(0)} : = \begin{pmatrix}
a_1^{(0)} & b_1^{(0)} \\
- \overline b_1^{(0)} & - \overline a_1^{(0)}
\end{pmatrix} = 2 {\cal S}^{- 1} (\Sigma + A_2) (\partial_x {\cal S}) + {\cal S}^{- 1} A_1 {\cal S}\,, \label{A1 (1)}  \\
& A_0^{(0)} := \begin{pmatrix}
a_0^{(0)} & b_0^{(0)} \\
- \overline b_0^{(0)} & - \overline a_0^{(0)}
\end{pmatrix} = {\cal S}^{- 1}(\Sigma + A_2)(\partial_{xx}{\cal S}) + {\cal S}^{- 1}A_1 (\partial_x {\cal S}) + {\cal S}^{- 1}(\partial_t {\cal S}) + {\cal S}^{- 1} A_0 {\cal S}\,. \label{A0 (1)}
\end{align}
Since the linear transformation ${\cal S}(t) : {\bf h}(x) \mapsto {\cal S}(t, x){\bf h}(x)$ is symplectic, the time dependent linear vector field $L_0(t) :=  \ii (\Sigma + A_2^{(0)} ) \partial_{xx} + \ii A_1^{(0)} \partial_x + \ii A_0^{(0)}$ is still Hamiltonian. 
Then, by Lemma \ref{caratterizzazione campi lineari Hamiltoniani ordine 2}, one has 
\begin{equation}\label{termine off diagonal ordine 1 = 0}
b_1^{(0)} = \partial_x b_2^{(0)} = 0\,,
\end{equation}
hence 
\begin{equation}\label{forma finale A1 (1)}
A_1^{(0)} = \begin{pmatrix}
a_1^{(0)} & 0\\
0 & - \overline a_1^{(0)}
\end{pmatrix} = 2 {\cal S}^{- 1} (\Sigma + A_2) (\partial_x {\cal S}) + {\cal S}^{- 1} A_1 {\cal S}.
\end{equation}
Note that \eqref{termine off diagonal ordine 1 = 0} can also be proved by a direct calculation.

\begin{lemma}\label{stime dopo T cal S}
There exists $\eta \in (0, 1)$ small enough, $\sigma > 0$  such that 
if $N_T(\sigma) \leq \eta$, then for any $0 \leq s \leq S - \sigma$ 
(where $S$ is defined in \eqref{regolarita coefficienti operatore da ridurre})
\begin{align}
&  \|{\cal S}^{\pm 1} - {\rm Id}\|_{T, s} \lesssim_s N_T(s + \sigma)\,.  \label{stima cal S pm 1} 
\end{align}
As a consequence 
\begin{align}
 \|\big( {\cal S}^{\pm 1} - {\rm Id} \big) {\bf h}\|_{T, s} \lesssim_s \eta \| {\bf h}\|_{T, s} +  N_T(s + \sigma) \| {\bf h}\|_{T, 0}\,. \label{azione stima tame cal S}
\end{align}
Furthermore, 
\begin{align}
& \| a_2^{(0)} \|_{T, s}, \| \partial_t a_2^{(0)}\|_{T, s}, \| \partial_{tt } a_2^{(0)}\|_{T, s} \lesssim_s N_T(s + \sigma)\,, \label{stima a2 (1)} \\
& \| a_1^{(0)}\|_{T, s}, \| \partial_t a_1^{(0)}\|_{T, s} \,, \| a_0^{(0)}\|_{T, s}, \| b_0^{(0)}\|_{T, s} \lesssim_s N_T(s + \sigma)\,.
\end{align}
\end{lemma}

\begin{proof}
Use definitions \eqref{matrice autovettori}, \eqref{definizione a2 (1)}, \eqref{A1 (1)} 
and apply Lemmas \ref{lemma interpolazione}, \ref{lemma Moser}.   
\end{proof}

\subsection{Change of the space variable}\label{primo cambio di variabile riduzione}
The aim of this subsection is to remove the $x$-dependence from the highest order term of the operator ${\cal L}_0$ defined in \eqref{cal L1}
(namely, to conjugate $\mL_0$ to an operator where the coefficient of $\pa_{xx}$ does not depend on the space variable $x$). 
For this purpose, we consider $t$-dependent families of diffeomorphisms of the torus $\T$ of the form 
$$
x \mapsto x + \alpha(t, x)\,,\quad \alpha : [0, T] \times \T \to \R\,, \qquad |\alpha_x(t, x)| \leq 1/ 2\,. 
$$
The above diffeomorphism is invertible and its inverse is given by 
$$
y \mapsto y + \widetilde \alpha(t, y)\,. 
$$
Then we define the linear operator ${\cal A}$ as 
\begin{equation}\label{definizione cal A}
{\cal A} := \sqrt{1 + \alpha_x} \,\,A_\alpha\,,\qquad \,\, A_\alpha h(t, x) :=  h(t, x + \alpha(t, x))\,.
\end{equation}
Using the fact that 
\begin{equation}\label{lalalala}
\dfrac{1}{{1 + \alpha_x(t, y + \widetilde \alpha(t, y))}} = {1 + \widetilde \alpha_y(t, y)}
\end{equation}
one gets that the inverse of the operator ${\cal A}$ has the form 
\begin{equation}\label{definizione cal A inverso}
{\cal A}^{- 1} = {\cal A}^* = \sqrt{1 + \widetilde \alpha_y} \,\,A_{\widetilde \alpha}\,, \qquad A_{\widetilde \alpha} h(t, y) := A_\alpha^{- 1} h(t, y) =  h(t, y + \widetilde \alpha(t, y))\,.
\end{equation}
A direct calculation shows that ${\cal A} {\mathbb I}_2$ is a symplectic map. 
The conjugation of the differential operators $\pa_t,\pa_x,\pa_{xx}$ and of multiplication operators $a = a(t, x) : h \mapsto a h$ are given by  
\begin{align}
& {\cal A}^{- 1} \partial_t {\cal A} = \partial_t + ( A_{\widetilde \alpha} \alpha_t ) \partial_y + \Big( A_{\widetilde \alpha} \frac{\alpha_{tx}}{2(1+\alpha_x)} \Big) \,,
\qquad {\cal A}^{- 1} a {\cal A} = (A_{\widetilde \alpha} a) \label{regola coniugazione cal A}\\
& {\cal A}^{- 1} \partial_x {\cal A} = [ 1 + (A_{\widetilde \alpha} \alpha_x) ] \pa_y + \Big( A_{\widetilde \alpha} \frac{\alpha_{xx}}{2(1+\alpha_x)} \Big)  \\
& {\cal A}^{- 1} \partial_{xx} {\cal A} = \{ A_{\widetilde \alpha} [(1 + \alpha_x)^2] \} \pa_{yy} + 2 ( A_{\widetilde \alpha} \alpha_{xx} ) \pa_y + \Big( A_{\widetilde \alpha} \frac{2 \alpha_{xxx} (1 + \alpha_x) - \alpha_{xx}^2}{4 (1+\alpha_x)^2} \Big)  \, .
\end{align} 
Conjugating the operator ${\cal L}_0$ in \eqref{cal L1} by means of the symplectic map ${\cal A}{\mathbb I}_2$ we get the operator 
\begin{align}
& {\cal L}_1 := {\cal A}^{- 1}{\mathbb I}_2 {\cal L}_0 {\cal A} {\mathbb I}_2 = \partial_t {\mathbb I}_2 + \ii A_2^{(1)} \partial_{yy} + \ii A_1^{(1)} \partial_y + \ii A_0^{(1)} \,,   \label{cal L2}
\end{align}
where, taking into account \eqref{termine off diagonal ordine 1 = 0},
$$
A_2^{(1)} := \begin{pmatrix}
a_2^{(1)} & 0 \\
0 & - a_2^{(1)} 
\end{pmatrix}\,, \qquad A_1^{(1)} := \begin{pmatrix}
a_1^{(1)} & 0 \\
0& - \overline a_1^{(1)}
\end{pmatrix}\,, \quad A_0^{(1)} := \begin{pmatrix}
a_0^{(1)} & b_0^{(1)} \\
- \overline b_0^{(1)} & - \overline a_0^{(1)}
\end{pmatrix}
$$
and
\begin{align}
& a_2^{(1)} := A_{\widetilde \alpha} [ ( 1 + a_2^{(0)} )( 1 + \alpha_x )^2 ]   \,,    \label{a2 (2)}\\
& a_1^{(1)} := A_{\widetilde \alpha} [ 2 ( 1 + a_2^{(0)} ) \alpha_{xx} + a_1^{(0)}( 1 + \alpha_x ) - \ii \alpha_t ] \, , \label{a1 (2)} \\
& a_0^{(1)} := A_{\widetilde \alpha} \Big\{ \frac{ ( 1 + a_2^{(0)} )[ 2 \alpha_{xxx} (1 + \alpha_x) - \alpha_{xx}^2 ] }{ 4( 1 + \alpha_x )^2 } + \frac{ a_1^{(0)} \alpha_{xx} - \ii \alpha_{tx} }{ 2( 1 + \alpha_x ) } + a_0^{(0)} \Big\}   \,,  \label{a0 (2)} \\
& b_0^{(1)} := A_{\widetilde \alpha} b_0^{(0)} \label{b0 (2)}\,.
\end{align}
Our purpose is to find $\alpha: [0, T] \times \T \to \R$ and a function $m_2 : [0, T] \to \R$ so that 
\begin{equation}\label{eq omologica primo cambio di variabile}
a_2^{(1)}(t, y) = m_2(t)\,, \qquad \forall (t, y) \in [0, T] \times \T\,.
\end{equation}
Thus, we have to solve
\begin{equation}\label{eq omologica primo cambio di variabile A}
(1 + a_2^{(0)})(1 + \alpha_x)^2 = m_2\,.
\end{equation}
Since $a_2^{(0)}$ is a real-valued function, the solutions are given by 
\begin{equation}\label{definizione m1 beta}
m_2 := \Big( \frac{1}{2 \pi} \int_\T \frac{d x}{(1 + a_2^{(0)})^{\frac12}}\Big)^{- 2}\,, \quad \alpha := \partial_x^{- 1} \Big( m_2^{\frac12}(1 + a_2^{(0)})^{- \frac12}  - 1\Big)\,,
\end{equation}
where $\pa_x^{-1}$ is the Fourier multiplier $\pa_x^{-1} e^{\ii jx} = (1 / \ii j) e^{\ii jx}$ for $j \in \Z$, $j \neq 0$, and $\pa_x^{-1} 1 = 0$.
Note that $m_2: [0,T] \to \R$ is a real-valued function.
The operator ${\cal L}_1$ in \eqref{cal L2} has then the form 
\begin{equation}\label{forma finale cal L2}
{\cal L}_1 = \partial_t {\mathbb I}_2 + \ii m_2 \Sigma \partial_{yy} + \ii A_1^{(1)} \partial_y + \ii A_0^{(1)} \,,
\end{equation}
where $\Sigma$ is defined in \eqref{Sigma Ak}.

\begin{lemma}
There exists $\eta \in (0, 1)$ small enough and $\sigma \in \N$ large enough, such that if $N_T( \sigma) \leq \eta$ (see \eqref{definizione NT sigma}), 
then, for any $0 \leq s \leq S - \sigma$,
\begin{align}
& \| m_2 - 1\|_{C^2_T} \lesssim \eta\,, \label{stima m2} \\
& \| \alpha\|_{T, s}, \| \partial_t \alpha\|_{T, s}, \| \partial_{tt} \alpha \|_{T, s}, \| \widetilde\alpha\|_{T, s}, \| \partial_t \widetilde \alpha\|_{T, s}, \| \partial_{tt} \widetilde \alpha \|_{T, s} \lesssim  N_T(s + \sigma)\,. \label{stime alpha alpha tilde} 
\end{align}
The transformations ${\cal A}^{\pm 1}$ map $C([0, T], H^s(\T)) \to C([0, T], H^s(\T))$ 
and they satisfy the estimate
\begin{align}
& \|{\cal A}^{\pm 1}  h \|_{T, s} \lesssim_s  \| h \|_{T, s} + N_T( s + \sigma) \| h \|_{ T, 0}\,, \qquad \forall h \in C([0, T], H^{s}(\T))\,. \quad \label{stima cal a pm1} 
\end{align}
The functions $ a_1^{(1)}, a_0^{(1)}, b_0^{(1)}$ satisfy
\begin{align}
& \| a_1^{(1)}\|_{T, s}, \| \partial_t a_1^{(1)}\|_{T, s}, \| a_0^{(1)}\|_{T, s}, \| b_0^{(1)}\|_{T, s} \lesssim_s N_T(s + \sigma) \label{stima a1a0b0 (1)}\,. 
\end{align}
\end{lemma}

\begin{proof}
The Lemma follows by the explicit expressions of the coefficients, applying Lemmas \ref{lemma interpolazione}, \ref{lemma derivata alpha tilde alpha astratto}, \ref{Lemma astratto cambio di variabile}.
\end{proof}

\subsection{Reparametrization of time}
\label{sezione riparametrizzazione tempo}

In this subsection we remove also the dependence on time from the highest order
(namely we conjugate the operator $\mL_1$ in \eqref{forma finale cal L2} 
to an operator where the coefficient of $\pa_{xx}$ is a constant matrix, 
independent of $(t,x)$, see \eqref{cal L3}).
We consider a diffeomorphism of the time interval $[0, T]$, 
\begin{equation}\label{definizione riparametrizzazione tempo}
\beta : [0, T] \to [0, T]\,, \quad \beta(0) = 0\,, \quad \beta(T) = T
\end{equation}
with inverse $\beta^{- 1}$. We define the operators ${\cal B}^{\pm 1}$ induced by the diffeomorphisms $\beta^{\pm 1}$ as 
\begin{equation}\label{definizione cal B}
{\cal B} h(t, x) : = h(\beta(t), x)\,, \qquad {\cal B}^{- 1}h(\tau, x) := h(\beta^{- 1}(\tau), x)\,.
\end{equation}
The following conjugation rules hold: 
\begin{align}
& {\cal B}^{- 1} a {\cal B} = ({\cal B}^{- 1} a)\,, \qquad {\cal B}^{- 1} \partial_t {\cal B} = ({\cal B}^{- 1} \beta' ) \partial_\tau\,, \quad {\cal B}^{- 1} \partial_x^m {\cal B} = \partial_x^m\,, \quad m \in \N.  \label{regola coniugazione cal B}
\end{align}
Conjugating the operator ${\cal L}_1$ in \eqref{forma finale cal L2}, we get 
\begin{align}
{\cal B}^{- 1}{\mathbb I}_2 {\cal L}_1 {\cal B}{\mathbb I}_2 & = ({\cal B}^{- 1} \beta' ) \partial_\tau {\mathbb I}_2 + \ii ({\cal B}^{- 1} m_2) \Sigma \partial_{xx} + \ii ({\cal B}^{- 1}{\mathbb I}_2 A_1^{(1)}) \partial_x + \ii ({\cal B}^{- 1}{\mathbb I}_2 A_0^{(1)})\,.
\end{align}
Our aim is to choose $\beta$ so that the coefficients of $\partial_\tau {\mathbb I}_2$ and $\ii \Sigma \partial_{xx}$ are proportional, namely we have to look for a diffeomorphism $\beta : [0, T] \to [0, T]$ and a constant $\mu\in\R$ such that 
\begin{equation}\label{equazione per beta tempo}
\beta' (t) = \frac{1}{\mu} m_2(t)\,, \qquad \forall t\in[0,T]. 
\end{equation}
Then, integrating in time from $0$ to $T$, 
by \eqref{definizione riparametrizzazione tempo} we fix the value of $\mu$ 
and define $\beta(t)$ as 
\begin{equation}\label{definizione mu 2}
\mu := \frac{1}{T} \int_0^T m_2(t)\, d t, \qquad 
\beta(t) := \frac{1}{\mu} \int_0^t m_2(s)\, d s\,. 
\end{equation}
Defining 
\begin{equation}\label{definizione rho (t)}
\rho(\tau) := ({\cal B}^{- 1} \beta') (\tau) = \mu^{- 1} ({\cal B}^{- 1} m_2) (\tau), \qquad \tau \in [0, T]\,,
\end{equation}
we get 
\begin{align}
&  {\cal B}^{- 1}{\mathbb I}_2 {\cal L}_1 {\cal B}{\mathbb I}_2 = \rho {\cal L}_2\,,\quad  {\cal L}_2 := \partial_\tau {\mathbb I}_2 + \ii \mu \Sigma \partial_{yy} + \ii A_1^{(2)} \partial_y + \ii A_0^{(2)} \,,     \label{cal L3} \\
& A_1^{(2)} := \begin{pmatrix}
a_1^{(2)} &0 \\
0 & - \overline a_1^{(2)}
\end{pmatrix}\,, \quad A_0^{(2)} := \begin{pmatrix}
a_0^{(2)} & b_0^{(2)} \\
- \overline b_0^{(2)} & - \overline a_0^{(2)}
\end{pmatrix} \, , \label{Ai (3)} \\
& a_1^{(2)} := \rho^{- 1} ({\cal B}^{- 1} a_1^{(1)})\,, \quad a_0^{(2)} := \rho^{- 1} ({\cal B}^{- 1} a_0^{(1)})\,, \quad b_0^{(2)} := \rho^{- 1} ({\cal B}^{- 1} b_0^{(1)}) \,. \label{ai (3)}
\end{align}
Note that the vector field $L_2(t) := \ii \mu \Sigma \partial_{yy} + \ii A_1^{(2)} \partial_y + \ii A_0^{(2)}$ is still Hamiltonian, since reparametrizations of time preserve the Hamiltonian structure. 
We also remark that, changing the time variable in the integral, one has 
\begin{equation} \label{0807.1}
\int_0^T \langle \mB {\bf u}(t) , {\bf v}(t) \rangle_{{\bf L}^2} \, dt 
= \int_0^T \langle {\bf u}(\t) , \rho^{-1}(\t) \mB^{-1} {\bf v}(\t) \rangle_{{\bf L}^2} \, d\t
\quad \forall {\bf u}, {\bf v} \in {\bf L}^2,
\end{equation}
namely the transpose of $\mB$ with respect to the \emph{time-space} scalar product 
$\int_0^T \langle \cdot, , \cdot \rangle_{{\bf L}^2} \, dt$ is 
\begin{equation} \label{0807.2}
\mB_* = \rho^{-1} \mB^{-1}.	
\end{equation}
 
\begin{lemma}
There exists $\eta \in (0, 1)$ small enough, $\sigma \in \N$ large enough such that if $N_T(\sigma) \leq \eta$, then for any $0 \leq s \leq S - \sigma$, the following holds: 
\begin{align}
&|\mu - 1|\,, \ \|\beta^{\pm 1} - 1 \|_{C^3_T} \lesssim \eta  
\label{stima mu 2} \\
& \|{\cal B}^{\pm 1}  h \|_{T, s} \lesssim  \| h \|_{T, s} \quad \forall h \in C([0, T], H^s(\T)) 
\label{stima cal B pm1} \\
& \|\rho^{\pm 1} - 1 \|_{C^1_T} \lesssim \eta 
\label{stima rho step 2} \\
& \| a_1^{(2)}\|_{T, s}, \| \partial_t a_1^{(2)}\|_{T, s}, \| a_0^{(2)}\|_{T, s}, \| b_0^{(2)}\|_{T, s} \lesssim N_T(s + \sigma)\,. \label{stima a (3) b (3)} 
\end{align}
\end{lemma}

\begin{proof}
Estimate \eqref{stima mu 2} for $\mu$ and $\beta^{\pm 1}$ follows from definitions \eqref{definizione mu 2} and estimate \eqref{stima m2} for $m_2$. 
Estimate \eqref{stima cal B pm1} for ${\cal B}^{\pm 1}$ follows directly from definition \eqref{definizione cal B}, computing the norm $\| \cdot\|_{T, s}$. 
Estimates \eqref{stima rho step 2}, \eqref{stima a (3) b (3)} for $\rho^{\pm 1}$ follow by the explicit expressions \eqref{definizione rho (t)}, \eqref{Ai (3)}, applying Lemma \ref{lemma interpolazione} and estimates \eqref{stima m2}, \eqref{stima mu 2}, \eqref{stima cal B pm1}, \eqref{stima a1a0b0 (1)}, \eqref{stima cal B pm1}, \eqref{stima rho step 2}. 
\end{proof}

\subsection{Translation of the space variable}\label{step 4 coniugazione}
In this subsection we remove the space average from the order 1 coefficient $a_1^{(2)}$
(namely we conjugate the operator $\mL_2$ in \eqref{cal L3} to an operator 
where the coefficient in front of $\pa_x$ is a $2 \times 2$ diagonal matrix  
whose entries are functions with zero space average, 
see \eqref{forma finale cal L4}, \eqref{Ai (4)}).
We consider the change of the space variable $z = y + p(\tau)$, where $p : [0, T] \to \R$, 
and define the operators 
\begin{equation}\label{operatori traslazione dello spazio}
{\cal T} h(\tau, y) := h(\tau , y + p(\tau) )\,, \qquad {\cal T}^{- 1} h(\tau, z) = {\cal T}^* h(\tau, z) = h(\tau, z - p(\tau))\,.
\end{equation}
A direct calculation shows that ${\cal T}$ is symplectic. 
Moreover, one has 
\begin{equation}\label{regola coniugazione cal T}
{\cal T}^{- 1} \partial_\tau {\cal T} = \partial_\tau + p' \partial_z\,, \qquad {\cal T}^{- 1} a {\cal T} = ({\cal T}^{- 1} a)\,, \qquad {\cal T}^{- 1} \partial_y^m {\cal T} = \partial_z^m\,, \quad m \in \N\,.
\end{equation}
Then 
\begin{align}
& {\cal L}_3 := {\cal T}^{- 1} {\mathbb I}_2 {\cal L}_2 {\cal T} {\mathbb I}_2 = \partial_\tau {\mathbb I}_2 + \ii \mu \Sigma \partial_{zz} + \ii A_1^{(3)} \partial_z + \ii A_0^{(3)}   
\label{cal L4}
\end{align}
with 
\begin{align}
& A_1^{(3)} := \begin{pmatrix}
a_1^{(3)} & 0 \\
0 & - \overline a_1^{(3)}
\end{pmatrix}\,, \qquad A_0^{(3)} := \begin{pmatrix}
a_0^{(3)} & b_0^{(3)} \\
- \overline b_0^{(3)} & - \overline a_0^{(3)}
\end{pmatrix} \,,\label{Ai (4)} \\
& a_1^{(3)} := - \ii p' + ({\cal T}^{- 1} a_1^{(2)}) \,,\qquad a_0^{(3)} := ({\cal T}^{- 1} a_0^{(2)})\,, \qquad b_0^{(3)} := ( {\cal T}^{- 1} b_0^{(2)} )\,.\label{coefficienti cal L4} 
\end{align}
Our aim is to choose the function $p$ so that 
\begin{equation}\label{eq omologica p}
 \int_\T a_1^{(3)}(\tau, z)\, d z = 0\,, \qquad \forall \tau \in [0, T]\,.
\end{equation}
Performing the change of variable $y = z - p(\tau)$, the above equation becomes (multiplying by $\ii$)
\begin{equation}\label{farina.00}
2 \pi p'(\tau) + \ii \int_\T a_1^{(2)}(\tau, y)\, d y = 0\,.
\end{equation}
By Lemma \ref{caratterizzazione campi lineari Hamiltoniani ordine 2}, we have that $a_1^{(2)} =  2 (\partial_x \mu) - \overline a_1^{(2)} = - \overline a_1^{(2)}$ (recall that $\mu$ is a constant), implying that $a_1^{(2)} : [0, T] \times \T \to \ii \R$, and then $\ii a_1^{(2)} :  [0, T] \times \T \to  \R$. Hence we can solve equation \eqref{farina.00} by setting 
\begin{equation}\label{definizione p}
p(\tau ) := - \frac{1}{2 \pi} \int_0^\tau \int_\T \ii a_1^{(2)}(\zeta, y)\, d y\, d \zeta\,, \qquad \tau \in [0, T]\,
\end{equation}
and we get that $p : [0, T] \to \R$ is a real-valued function. Renaming the variables $\tau = t$, $z = x$ we have
\begin{equation}\label{forma finale cal L4}
{\cal L}_3 = \partial_t {\mathbb I}_2 + \ii \mu \Sigma \partial_{xx} + \ii A_1^{(3)} \partial_x + \ii A_0^{(3)}\,, \qquad  \int_\T a_1^{(3)}(t, x)\, d x = 0\,, \qquad \forall t \in [0, T]\,.
\end{equation}

\begin{lemma}
There exists $\eta \in (0, 1)$ small enough and $\sigma \in \N$ large enough such that if $N_T(\sigma) \leq \eta$, then for any $0 \leq s \leq S - \sigma$, the following estimates hold: 
\begin{align}
& \|p \|_{C^2_T} \lesssim \eta\,.\label{stima p} \\
\end{align}
The transformations ${\cal T}^{\pm 1}$ map $C([0, T], H^s(\T)) \to C([0, T], H^s(\T))$ 
and they satisfy
\begin{align}
& \|{\cal T}^{\pm 1}  h \|_{T, s} \lesssim  \| h \|_{T, s}
\quad \forall h \in C([0, T], H^s(\T)), \  \forall s \geq 0\,.  \label{stima cal T pm1}
\end{align}
Furthermore
\begin{align}
& \| a_1^{(3)}\|_{T, s}, \| \partial_t a_1^{(3)}\|_{T, s}, \| a_0^{(3)}\|_{T, s}, \| b_0^{(3)}\|_{T, s} \lesssim_s N_T(s + \sigma)\,. \label{stime a (4) b (4)} 
\end{align}
\end{lemma}

\begin{proof}
The lemma follows from definitions \eqref{operatori traslazione dello spazio}, \eqref{coefficienti cal L4},  \eqref{definizione p}, applying Lemmas \ref{lemma interpolazione}, \ref{lemma derivata alpha tilde alpha astratto}, \ref{Lemma astratto cambio di variabile} and using estimates 
\eqref{stima a (3) b (3)}.
\end{proof}

\subsection{Elimination of order one}\label{step 5 riduzione}
In this last subsection, we remove completely the order 1
(namely we conjugate the operator $\mL_3$ in \eqref{forma finale cal L4}
to an operator where the term $\pa_x$ is not present). 
We consider the multiplication operator by the matrix valued function 
\begin{equation}\label{definizione cal m}
{\cal M} := \begin{pmatrix}
 v & 0 \\
0 & \overline v
\end{pmatrix}\,, \qquad v: [0, T] \times \T \to \C\,,
\end{equation}
where $v$ is a function sufficiently close to $1$, to be determined. 
The inverse ${\cal M}^{- 1}$ and the adjoint ${\cal M}^*$ are 
\begin{equation}\label{definizione cal M inverso}
{\cal M}^{- 1} = \begin{pmatrix}
v^{- 1} & 0 \\
0 &\overline v^{- 1}
\end{pmatrix}\,, \qquad {\cal M}^* = \begin{pmatrix}
\overline{v} & 0 \\
0 & v
\end{pmatrix}
\end{equation}
We compute 
\begin{align}
& {\cal L}_4 := {\cal M}^{- 1} {\cal L}_3 {\cal M} = \partial_t {\mathbb I}_2 + \ii \mu \Sigma \partial_{xx} + \ii A_1^{(4)} \partial_x + \ii A_0^{(4)}   \label{cal L5}
\end{align}
with 
\begin{align}
& A_1^{(4)} := \begin{pmatrix}
a_1^{(4)} & 0 \\
0 & - \overline a_1^{(4)}
\end{pmatrix}\,, \qquad A_0^{(4)} := \begin{pmatrix}
a_0^{(4)} & b_0^{(4)} \\
- \overline b_0^{(4)} & - \overline a_0^{(4)}
\end{pmatrix} \,,\label{Ai (5)} \\
& a_1^{(4)} := a_1^{(3)} + 2 \mu v^{-1} v_x \,,\qquad a_0^{(4)} := a_0^{(3)} + v^{-1} (\mu v_{xx} + a_1^{(3)} v_x - \ii v_t ) \,, \qquad b_0^{(4)} := b_0^{(3)} \,.\label{coefficienti cal L5} 
\end{align}
To remove the first order term we need to solve the equation 
\begin{equation}\label{equazione omologica moltiplicazione}
a_1^{(3)} + 2 \mu v^{-1} v_x = 0\,.
\end{equation}
We look for solutions of the form $v = {\rm exp}(q)$ and we get
$
a_1^{(3)} + 2 \mu q_x = 0
$,
which, recalling \eqref{eq omologica p}, has the solution 
$
q = - (2\mu)^{-1} \partial_x^{- 1} a_1^{(3)}
$.
Hence we set
\begin{equation}\label{definizione v}
v := {\rm exp}\Big( - \frac{\pa_x^{-1} a_1^{(3)}}{2 \mu} \Big)\, ,
\end{equation}
which solves \eqref{equazione omologica moltiplicazione} and gives
\begin{equation}\label{definizione cal R}
{\cal L}_4 = \partial_t {\mathbb I}_2 + \ii \mu \Sigma \partial_{xx} + {\cal R}\,, \qquad {\cal R}:= \ii A_0^{(4)}.
\end{equation}
We remark that, by the Hamiltonian structure, $a_1^{(3)} = - \overline{a}_1^{(3)}$, 
therefore 
$$
\overline v = \overline{{\rm exp}\Big( - \frac{\pa_x^{-1} a_1^{(3)}}{2 \mu} \Big)} = {\rm exp}\Big( - \frac{\pa_x^{-1} \overline a_1^{(3)}}{2 \mu} \Big) = {\rm exp}\Big(  \frac{\pa_x^{-1} a_1^{(3)}}{2 \mu} \Big) = v^{- 1}\,.
$$
Recalling \eqref{definizione cal M inverso} one gets 
\begin{equation}\label{uguaglianza cal M aggiunto inverso}
{\cal M}^{- 1} = {\cal M}^*\,.
\end{equation}
\begin{lemma}\label{stime cal L5}
There exist $\eta \in (0, 1)$ small enough, $\sigma \in \N$ large enough such that, 
if $N_T( \sigma) \leq \eta$, for any $0 \leq s \leq S - \sigma$, 
the function $v$ defined in \eqref{definizione v} satisfies the estimate 
\begin{equation}\label{stima v meno id}
\| v^{\pm 1} - 1 \|_{T, s} , \| \partial_t v^{\pm 1}\|_{T, s} \lesssim_s N_T(s + \sigma)\,.
\end{equation}
As a consequence, the transformations ${\cal M}^{\pm 1}$ satisfy 
\begin{equation}\label{stima cal M}
\| {\cal M}^{\pm 1 }{\bf h}\|_{T, s} \lesssim_s \Big(\| {\bf h} \|_{T, s} + N_T(s + \sigma) \| {\bf h}\|_{T, 0}\Big)\,, \quad \forall {\bf h} = (h, \overline h) \in C([0, T], {\bf H}^s_x)\,. 
\end{equation}
The multiplication operator 
\begin{equation}\label{cal R r1 r2}
{\cal R} = \begin{pmatrix}
\ii a_0^{(4)} & \ii b_0^{(4)} \\
- \ii \overline b_0^{(4)} & - \ii \overline a_0^{(4)}
\end{pmatrix} :=
\begin{pmatrix}
r_1 & r_2 \\
\overline r_1 & \overline r_2
\end{pmatrix}
\end{equation}
satisfies 
\begin{equation}\label{stima cal R linearizzato ridotto}
\| r_1 \|_{T, s}, \| r_2 \|_{T, s} \lesssim_s N_T(s + \sigma)\,. 
\end{equation}
\end{lemma}

\begin{proof}
The lemma follows by recalling definitions \eqref{definizione cal m}, \eqref{definizione cal M inverso}, \eqref{definizione v}, \eqref{definizione cal R}, applying Lemma \ref{lemma interpolazione} and estimates \eqref{stima mu 2}, \eqref{stime a (4) b (4)}. 
\end{proof}

\section{Observability}\label{sezione osservabilita}
\label{osservabilita operatori lineari generali}
In this section we prove the observability for linear operators ${\cal L}$ of the form \eqref{operatore lineare generale}. The proof is split in several lemmas. 

\begin{lemma}[Ingham]
\label{Ingham inequality}
Let $T > 0$. Then there exists a constant $C_1(T) > 0$ such that for any $\mu \geq \frac12$ and for any $w = (w_j)_{j \in \N} \in \ell^2(\N, \C)$, one has 
$$
\int_0^T \Big| \sum_{j \in \N} w_j e^{\ii \mu j^2 t} \Big|^2 \, d t \geq C_1(T) \sum_{j \in \N} |w_j|^2\,.
$$
\end{lemma}

\begin{proof}
This result is classical. For a proof, see for instance Theorem 4.3 in Section 4.1 of \cite{Micu-Zuazua}. To prove that the constant $C_1(T)$ does not depend on $\mu\in[\frac12,+\infty)$ it is enough to follow the proof in \cite{Micu-Zuazua} and use the lower bound $|\mu j^2 - \mu k^2| \geq\frac12$ for all pairs of distinct nonnegative integers $j\neq k$.
\end{proof}

\begin{lemma}[Observability for $\partial_t + \ii \mu \partial_{xx}$]
Let $T > 0$ and $\omega \subset \T$ be a non-empty open set. Then there exists a constant $C_2 := C_2(T, \omega) > 0$ such that for any $\mu \geq \frac12$, the following holds: for any $u_T \in L^2(\T)$ the solution $u$ of the backward Cauchy problem 
\begin{equation}\label{problema di cauchy parte costante}
\partial_t u + \ii \mu \partial_{xx} u = 0\,, \qquad u(T, \cdot) = u_T(\cdot)
\end{equation}
satisfies the estimate
$$
\int_0^T \int_\omega |u(t, x)|^2\, d x\, d t \geq C_2 \| u_T\|_{0}^2\,.
$$
\end{lemma}

\begin{proof}
The proof of this result is standard. For instance, it can be deduced by adapting the proof of Proposition 6.5 in \cite{ABH} to the present, simpler case. 
We give here the proof for completeness.

We fix an open interval $\om_0=(a,b)\subset\om$. 
We choose $b-a$ smaller than a suitable universal constant, so that
\begin{equation}\label{diseq seni}
\left| \frac{\sin(n(b-a))}{n} \right| = (b-a) \left| \frac{\sin(n(b-a))}{n (b-a)} \right| \leq (b-a) \frac{\sin(b-a)}{b-a} = \sin(b-a) \qquad \forall n\geq1.
\end{equation}
Let $u_T=\sum_{n\in\Z} w_n e^{\ii nx}$, so that $\| u_T\|_{L^2_x}^2 = \sum_{n\in\Z} |w_n|^2$. We compute
\[
u(t,x)=\sum_{n\in\Z} w_n e^{\ii nx} e^{\ii\mu n^2 (t-T)} = \sum_{n\in\N} z_n(x) e^{\ii\mu n^2 t}
\]
where
\[
z_n(x) :=
\begin{cases}
e^{-\ii\mu n^2 T} (w_n e^{\ii nx} + w_{-n} e^{-\ii nx}) \quad & \text{for} \ n\geq 1, \\
w_0 & \text{for} \ n=0.
\end{cases}
\]
By Lemma \ref{Ingham inequality} we get
\[
\int_0^T \int_{\om} |u(t,x)|^2 \; dx \; dt \geq C_1(T) \sum_{n\in\N} \int_{\om_0} |z_n(x)|^2 \; dx.
\]
It remains to prove that
\begin{equation}\label{claim bn}
\sum_{n\in\N} \int_{\om_0} |z_n(x)|^2 \; dx \geq C(\om_0) \sum_{n\in\Z} |w_n|^2
\end{equation}
for some constant $C(\om_0)$ depending only on $\om_0$. We have
\begin{equation}\label{b0a0}
\int_{\om_0} |z_0(x)|^2 \; dx = (b-a) |w_0|^2.
\end{equation}
For $n\geq1$, we compute
\begin{align*}
\int_{\om_0} |z_n(x)|^2 \; dx 
& = \int_{\om_0} \big( |w_n|^2 + |w_{-n}|^2 + w_n \bar w_{-n} e^{2\ii nx} + \bar w_n w_{-n} e^{-2\ii nx} \big) dx \\
& \geq (b-a)\{ |w_n|^2 + |w_{-n}|^2 \} - |w_n||w_{-n}| \left( \left| \int_{\om_0} e^{2\ii nx} \; dx \right| + \left| \int_{\om_0} e^{-2\ii nx} \; dx \right| \right) \\
& = (b-a)\{ |w_n|^2 + |w_{-n}|^2 \} - 2 |w_n||w_{-n}| \left| \frac{\sin(n(b-a))}{n} \right| \\
& \geq \left\{ b-a - \left| \frac{\sin(n(b-a))}{n} \right| \right\} \big( |w_n|^2 + |w_{-n}|^2 \big).
\end{align*}
Finally, we use \eqref{diseq seni} and we deduce
\begin{equation}\label{bnan}
\int_{\om_0} |z_n(x)|^2 \; dx \geq \big\{ b-a - \sin(b-a) \big\} \big( |w_n|^2 + |w_{-n}|^2 \big).
\end{equation}
Note that $b-a - \sin(b-a)>0$ is a constant depending only on $\om_0$.
Summing \eqref{bnan} over $n\in\N$ and adding $\eqref{b0a0}$, we get \eqref{claim bn}, which concludes the proof.
\end{proof}

\begin{lemma}[Observability for ${\cal L}_4 = \partial_t {\mathbb I}_2 + \ii \mu \Sigma \partial_{xx} + {\cal R}$]
\label{osservabilita secondo lemma}
Let $T > 0$, $\omega \subset \T$ be a non-empty open set and ${\cal L}_4$ the operator defined in \eqref{cal L5}. Then there exist $\eta \in (0, 1)$ small enough and $\sigma \in \N$ large enough such that if $N_T(\sigma) \leq \eta$ then the following holds: let ${\bf u}_T \in {\bf L}^2(\T)$ 
and let ${\bf u}(t, x)$ be the solution of the backward Cauchy problem 
\begin{equation}\label{Cauchy backward cal L5}
\partial_t {\bf u} + \ii \mu \Sigma \partial_{xx} {\bf u} + {\cal R} {\bf u} = 0\,, \qquad {\bf u}(T, \cdot) = {\bf u}_T\,.
\end{equation}
Then there exists a constant $C_3 := C_3(T, \omega) > 0$ (independent of ${\bf u}_T$) such that 
$$
\int_0^T \int_\omega |{\bf u}(t, x)|^2\, d x \, dt \geq C_3 \| {\bf u}_T \|_{0}^2\,.
$$
\end{lemma}

\begin{proof}
Let ${\bf u}_1$ be the solution of 
$$
\partial_t {\bf u}_1 + \ii \mu \Sigma \partial_{xx} {\bf u}_1 = 0\,, \qquad 
{\bf u}_1(T, \cdot) = {\bf u}_T\,.
$$
If ${\bf u}_1 = (u_1, \overline u_1)$ and ${\bf u}_T = (u_T, \overline u_T)$, 
then $u_1$ solves \eqref{problema di cauchy parte costante}. Therefore 
\begin{equation}\label{inequality bf u1 osservabilita}
\int_0^T \int_\omega |{\bf u}_1(t, x)|^2\, d x\, d t \geq C_2 \| {\bf u}_T\|_{0}^2\, \qquad \text{and} \qquad \| {\bf u}_1\|_{T, 0} = \| {\bf u}_T\|_0\,.
\end{equation}
Then the function ${\bf u}_2 := {\bf u} - {\bf u}_1$ solves the Cauchy problem 
$$
\partial_t {\bf u}_2 + \ii \mu \Sigma \partial_{xx} {\bf u}_2 +{\cal R}{\bf u}_2 = - {\cal R} {\bf u}_1\,, \qquad {\bf u}_2(T, \cdot) = 0\,. 
$$
By Lemma \ref{buona positura equazione lineare 1}, \eqref{stima cal R linearizzato ridotto}, \eqref{inequality bf u1 osservabilita}, since $N_T(\sigma) \leq \eta$,
\begin{equation}\label{stima bf u2 osservabilita cal L5}
\| {\bf u}_2\|_{T, 0} \lesssim  \| {\cal R} {\bf u}_1\|_{T, 0} 
\lesssim N_T(\sigma) \| {\bf u}_T\|_{0} 
\lesssim \eta \| {\bf u}_T\|_{0}  \,.
\end{equation}
Therefore, using the elementary inequality $(a + b)^2 \geq \frac12 a^2 - b^2$,
\begin{align*}
\int_0^T \int_\omega |{\bf u}(t, x)|^2\, dx \, d t & \geq \frac12 \int_0^T \int_\omega |{\bf u}_1(t, x)|^2\, dx \, d t -  \int_0^T \int_\omega |{\bf u}_2(t, x)|^2\, dx \, d t \\
& \stackrel{\eqref{inequality bf u1 osservabilita}}{\geq} \frac{C_2}{2} \| {\bf u}_T\|_{0}^2 -  \int_0^T \int_\T |{\bf u}_2(t, x)|^2\, dx \, d t \\
& \geq \frac{C_2}{2} \| {\bf u}_T\|_{0}^2 - T \| {\bf u}_2\|_{T, 0}^2  \\
& \stackrel{\eqref{stima bf u2 osservabilita cal L5}}{\geq} \frac{C_2}{2} \| {\bf u}_T\|_{0}^2 - T \eta^2 \| {\bf u}_T\|_{0}^2 \geq \frac{C_2}{4} \| {\bf u}_T\|_{0}^2
\end{align*}
by taking $\eta \in (0, 1)$ small enough, then the claimed inequality holds by taking $C_3 := C_2/4$. 
\end{proof}

\begin{lemma}[Observability for ${\cal L}_3 = \partial_t {\mathbb I}_2 + \ii \mu \Sigma \partial_{xx} + \ii A_1^{(3)} \partial_x + \ii A_0^{(3)}$]
\label{osservabilita cal L4}
Let $T > 0$, $\omega \subset \T$ be a non-empty open set and ${\cal L}_3$ be the operator defined in \eqref{forma finale cal L4}. Then there exist $\eta \in (0, 1)$ small enough and $\sigma \in \N$ large enough such that if $N_T(\sigma) \leq \eta$ then the following holds: let ${\bf u}_T \in {\bf L}^2(\T)$ and ${\bf u}(t, x)$ be the solution of the backward Cauchy problem 
\begin{equation}\label{Cauchy problem cal L4}
\partial_t {\bf u} + \ii \mu \Sigma \partial_{xx} {\bf u} + \ii A_1^{(3)}(t, x) \partial_x {\bf u} + \ii A_0^{(3)}(t, x) {\bf u} = 0\,, \qquad {\bf u}(T, \cdot) = {\bf u}_T\,.
\end{equation}
Then there exists a constant $C_4 := C_4(T, \omega) > 0$ (independent of ${\bf u}_T$) 
such that 
$$
\int_0^T \int_\omega |{\bf u}(t, x)|^2\, d x \, d t \geq C_4 \| {\bf u}_T \|_{0}^2\,.
$$
\end{lemma}
\begin{proof}
Lemma \ref{buona positura equazione lineare 2} guarantees that if ${\bf u}_T \in {\bf L}^2(\T)$, then the Cauchy problem \eqref{Cauchy problem cal L4} admits a unique solution ${\bf u} \in C([0, T], {\bf L}^2(\T))$. In Section \ref{step 5 riduzione}, we have proved that the operator ${\cal L}_3$ in \eqref{forma finale cal L4} is conjugated to the operator ${\cal L}_4$ in \eqref{cal L5} by using the operator ${\cal M}$ defined in \eqref{definizione cal m}. Therefore ${\bf u}$ solves the Cauchy problem 
$$
{\cal L}_3 {\bf u} = 0\,, \qquad {\bf u}(T, \cdot) = {\bf u}_T
$$
if and only if $\widetilde{\bf u}(t, \cdot) := {\cal M}^{- 1}(t) {\bf u}(t, \cdot)$ solves the Cauchy problem 
$$
{\cal L}_4 \widetilde{\bf u} = 0\,, \qquad \widetilde{\bf u}(T, \cdot) = {\cal M}^{- 1}(T){\bf u}_T\,.
$$
By Lemma \ref{osservabilita secondo lemma} we get the inequality for $\widetilde{\bf u}$ 
\begin{equation}\label{ragu 1}
\int_0^T \int_\omega |\widetilde{\bf u}(t, x)|^2\,d x \, dt \geq C_3 \| \widetilde{\bf u}_T\|_{0}^2\,.
\end{equation}
By estimate \eqref{stima v meno id} of Lemma \ref{stime cal L5}, using that $C([0, T] \times \T)$ is embedded into $C([0, T], H^1(\T))$ one has that, for some $\sigma \in \N$ large enough, the function $v(t,x)$, defined in \eqref{definizione v} and determining the operator $\cal M$, satisfies
$$
\|v^{\pm 1} - 1\|_{L^\infty_T L^\infty_x} \lesssim N_T(\sigma) \lesssim \eta\,.
$$
Hence, for any function ${\bf h} = (h , \overline h) : [0, T] \times \T \to \C^2$, for $\eta$ small enough, we get  for any $(t, x) \in [0, T] \times \T$
\begin{align}
 |{\cal M}^{- 1}(t) {\bf h}(t, x)| & \leq (1 + \| v^{- 1} - 1\|_{L^\infty_T L^\infty_x}) |{\bf h}(t, x)|   \leq (1 + C \eta) |{\bf h}(t, x)| \leq 2 |{\bf h}(t, x)|\,,  \label{pappap0} \\
 |{\cal M}^{- 1}(t) {\bf h}(t, x)| &  \geq |{\bf h}(t, x)| - \| v^{- 1} - 1 \|_{L^\infty_T L^\infty_x} |{\bf h}(t, x)| \geq (1 - C \eta) |{\bf h}(t, x)| \geq \frac12 |{\bf h}(t, x)|\,.\label{pappap1}
\end{align}
Using that $\widetilde{\bf u}(t, x) = {\cal M}^{- 1}(t){\bf u}(t, x)$, the two inequalities above imply 
$$
\int_0^T \int_\omega |\widetilde{\bf u}(t, x)|^2\, d x\, d t \leq 4 \int_0^T \int_\omega |{\bf u}(t, x)|^2\,dx\,dt\,, \quad \| \widetilde{\bf u}_T\|_{0}^2 \geq \frac14 \| {\bf u}_T\|_{0}^2\,,
$$
and then the claimed inequality follows by \eqref{ragu 1} and by setting $C_4 := C_3/16$.
\end{proof}

\begin{lemma}[Observability for ${\cal L}_2 = \partial_t {\mathbb I}_2 + \ii \mu \Sigma \partial_{xx} + \ii A_1^{(2)} \partial_x + \ii A_0^{(2)}$]
\label{osservabilita cal L3}
Let $T > 0$, let $\omega \subset \T$ be a non-empty open set and ${\cal L}_2$ be the operator defined in \eqref{cal L3}. Then there exist $\eta \in (0, 1)$ small enough and $\sigma \in \N$ large enough such that if $N_T(\sigma) \leq \eta$ then the following holds: let ${\bf u}_T \in {\bf L}^2(\T)$ and ${\bf u}(t, x)$ be the solution of the backward Cauchy problem 
\begin{equation}\label{Cauchy problem cal L3}
\partial_t {\bf u} + \ii \mu \Sigma \partial_{xx} {\bf u} + \ii A_1^{(2)}(t, x) \partial_x {\bf u} + \ii A_0^{(2)}(t, x) {\bf u} = 0\,, \qquad {\bf u}(T, \cdot) = {\bf u}_T\,.
\end{equation}
Then there exists a constant $C_5 := C_5(T, \omega) > 0$ 
(independent of ${\bf u}_T$) such that 
$$
\int_0^T \int_\omega |{\bf u}(t, x)|^2\, d x \, d t \geq C_5 \| {\bf u}_T \|_{0}^2\,.
$$
\end{lemma}

\begin{proof}
Lemma \ref{buona positura equazione lineare 3} guarantees that if ${\bf u}_T \in {\bf L}^2(\T)$ then there exists a unique solution ${\bf u} \in C([0, T],$ ${\bf L}^2(\T))$ of the Cauchy problem \eqref{Cauchy problem cal L3}. In Section \ref{step 4 coniugazione}, we have proved that the transformation ${\cal T}$ defined in \eqref{operatori traslazione dello spazio} conjugates the operator ${\cal P}_4$ defined in \eqref{cal L3} to the operator ${\cal P}_5$ given in \eqref{forma finale cal L4}, hence ${\bf u}$ solves the Cauchy problem 
$$
{\cal L}_2 {\bf u} = 0\,, \qquad {\bf u}(T, \cdot) = {\bf u}_T
$$
if and only if $\widetilde{\bf u}(t, x) := {\cal T}^{- 1}(t){\bf u}(t, x)$ solves the Cauchy problem 
$$
{\cal L}_3 \widetilde{\bf u} = 0\,, \qquad \widetilde{\bf u}(T, \cdot) = {\cal T}^{- 1}(T) {\bf u}_T\,.
$$
Then by Lemma \ref{osservabilita cal L4}, applied to a time interval $\omega_1 := (\alpha_1, \beta_1) \subset \omega$, the function $\widetilde{\bf u}$ satisfies the property 
\begin{equation}\label{pipi 0}
\int_0^T \int_{\omega_1} |\widetilde{\bf u}(t, x)|^2\, d x\, dt \geq C_4(T, \omega_1) \| \widetilde{\bf u}_T\|_{0}^2\,.
\end{equation}
Performing the change of variables $y = x - p(T)$ (where $p(t)$, defined in \eqref{definizione p}, is the function determining the operator $\cal T$), one has 
\begin{align}
\| \widetilde{\bf u}_T\|_{0}^2 & = \int_\T |{\bf u}_T(x - p(T))|^2\, d x = \int_\T |{\bf u}_T(y)|^2\,d y = \| {\bf u}_T\|_{0}^2\,. \label{pipi 1}
\end{align}
By the change of variables $y = x - p(t)$, 
\begin{align}
\int_0^T \int_{\omega_1} |\widetilde{\bf u}(t, x)|^2\, d x\, d t & = \int_0^T \int_{\omega_1} |{\bf u}(t, x - p(t))|^2\,dx\,d t = \int_0^T \int_{\alpha_1 - p(t)}^{\beta_1 - p(t)} |{\bf u}(t, y)|^2\,d y \, d t \,. \label{pipi 11}
\end{align}
By estimate \eqref{stima p}, for all $t \in [0, T]$, $[\alpha_1 - p(t), \beta_1 - p(t)] \subseteq [\alpha_1 - C \eta, \beta_1 + C \eta] \subset \omega$ if $\eta$ is small enough. 
Therefore, by \eqref{pipi 11},
\begin{align}
\int_0^T \int_{\omega_1} |\widetilde{\bf u}(t, x)|^2\, d x\, d t \leq \int_0^T \int_\omega |{\bf u}(t, x)|^2\,d x\, dt \,. \label{pipi 2}
\end{align}
The claimed inequality follows by \eqref{pipi 0}, \eqref{pipi 1}, \eqref{pipi 2}, with $C_5:=C_4(T, \omega_1)$. 
\end{proof}
\begin{lemma}[Observability for ${\cal L}_1 = \partial_t {\mathbb I}_2 + \ii m_2 \Sigma \partial_{yy} + \ii A_1^{(1)} \partial_y + \ii A_0^{(1)}$]
\label{osservabilita cal L2}
Let $T > 0$, $\omega \subset \T$ be a non-empty open set and ${\cal L}_1$ be the operator defined in \eqref{forma finale cal L2}. Then there exist $\eta \in (0, 1)$ small enough and $\sigma \in \N$ large enough such that if $N_T(\sigma) \leq \eta$ then the following holds: let ${\bf u}_T \in {\bf L}^2(\T)$ and ${\bf u}(t, x)$ be the solution of the backward Cauchy problem 
\begin{equation}\label{Cauchy problem cal L2}
\partial_t {\bf u} + \ii m_2(t) \Sigma \partial_{xx} {\bf u} + \ii A_1^{(1)}(t, x) \partial_x {\bf u} + \ii A_0^{(1)}(t, x) {\bf u}= 0\,, \qquad {\bf u}(T, \cdot) = {\bf u}_T(\cdot)\,.
\end{equation}
Then there exists a constant $C_6 := C_6(T, \omega) > 0$ 
(independent of ${\bf u}_T$) such that 
$$
\int_0^T \int_\omega |{\bf u}(t, x)|^2\, d x \, d t \geq C_6 \| {\bf u}_T \|_{0}^2\,.
$$
\end{lemma}

\begin{proof}
Lemma \ref{buona positura equazione lineare 4} guarantees that if ${\bf u}_T \in {\bf L}^2(\T)$ then there exists a unique solution ${\bf u} \in C([0, T],$ ${\bf L}^2(\T))$ of the Cauchy problem \eqref{Cauchy problem cal L2}. In Section \ref{sezione riparametrizzazione tempo}, we have proved that the transformation ${\cal B}$ defined in \eqref{definizione cal B} conjugates the operator ${\cal L}_1$ defined in \eqref{forma finale cal L2} to the operator $\rho {\cal L}_2$ where the function $\rho$ is defined by \eqref{definizione rho (t)} and the operator ${\cal L}_2$ is given in \eqref{cal L3}. Hence ${\bf u}$ solves the Cauchy problem 
$$
{\cal L}_1 {\bf u} = 0\,, \qquad {\bf u}(T, \cdot) = {\bf u}_T
$$
if and only if $\widetilde{\bf u}(t, x) := {\cal B}^{- 1}{\bf u}(t, x)$ solves 
$$
{\cal L}_2 \widetilde{\bf u} = 0\,, \qquad \widetilde{\bf u}(T, \cdot) =  {\bf u}_T\,
$$
(we use that ${\cal B}^{- 1} {\bf u}_T = {\bf u}_T$ since ${\cal B}$ acts only in time).
Then, by Lemma \ref{osservabilita cal L3}, the function $\widetilde{\bf u}$ satisfies
\begin{equation}\label{cacca 0}
\int_0^T \int_{\omega} |\widetilde{\bf u}(t, x)|^2\, d x\, dt \geq C_5 \| {\bf u}_T\|_{0}^2\,.
\end{equation}
Performing the change of the time variable $\tau = \beta^{- 1}(t)$ (recall \eqref{definizione riparametrizzazione tempo}), we get for $\eta$ small enough
\begin{align}
\int_0^T \int_{\omega} |\widetilde{\bf u}(t, x)|^2\, d x\, d t & = \int_0^T \int_{\omega} |{\bf u}(\beta^{- 1}(t), x )|^2\,dx\,d t = \int_0^T \int_\omega |{\bf u}(\tau, x)| ^2 \beta'(\tau)\,d x\, d \tau \nonumber\\
& \stackrel{\eqref{stima mu 2}}{\leq} (1 + C \eta) \int_0^T \int_\omega |{\bf u}(\tau, x)| ^2 \,d x\, d \tau \leq 2 \int_0^T \int_\omega |{\bf u}(\tau, x)| ^2 \,d x\, d \tau \,. \label{cacca 11}
\end{align}
The claimed inequality follows by \eqref{cacca 0}, \eqref{cacca 11} and setting $C_6 := C_5/2$. 
\end{proof}

\begin{lemma}[Observability for ${\cal L}_0 = \partial_t {\mathbb I}_2 + \ii (\Sigma + {A_2^{(0)}} ) \partial_{xx} + \ii {A_1^{(0)}} \partial_x + \ii A_0^{(0)}$]
\label{osservabilita cal L1a}
Let $T > 0$, let $\omega \subset \T$ be a non-empty open set and ${\cal L}_0$ be the operator defined in \eqref{cal L1}. Then there exist $\eta \in (0, 1)$ small enough and $\sigma \in \N$ large enough such that if $N_T(\sigma) \leq \eta$ then the following holds: let ${\bf u}_T \in {\bf L}^2(\T)$ and ${\bf u}(t, x)$ be the solution of the backward Cauchy problem 
\begin{equation}\label{Cauchy problem cal L1}
\partial_t {\bf u} + \ii  (\Sigma + A_2^{(0)})\partial_{xx} {\bf u} + \ii A_1^{(0)} \partial_x {\bf u} + \ii A_0^{(0)} {\bf u} = 0\,, \qquad {\bf u}(T, \cdot) = {\bf u}_T 	\,.
\end{equation}
Then there exists a constant $C_7 := C_7(T, \omega) > 0$ 
(independent of ${\bf u}_T$) such that 
$$
\int_0^T \int_\omega |{\bf u}(t, x)|^2\, d x \, d t \geq C_7 \| {\bf u}_T \|_{0}^2\,.
$$
\end{lemma}

\begin{proof}
Lemma \ref{buona positura equazione lineare 5} guarantees that if ${\bf u}_T \in {\bf L}^2(\T)$ then there exists a unique solution ${\bf u} \in C([0, T],$ ${\bf L}^2(\T))$ of the Cauchy problem \eqref{Cauchy problem cal L1}.
In Section \ref{primo cambio di variabile riduzione}, we have proved that the transformation ${\cal A}$ defined in \eqref{definizione cal A} conjugates the operator ${\cal L}_0$ defined in \eqref{cal L1} to the operator $ {\cal L}_1$ defined in \eqref{forma finale cal L2}. Hence ${\bf u}$ solves the Cauchy problem 
$$
{\cal L}_0 {\bf u} = 0\,, \qquad {\bf u}(T, \cdot) = {\bf u}_T
$$
if and only if $\widetilde{\bf u}(t, x) := {\cal A}^{- 1}{\bf u}(t, x)$ solves 
${\cal L}_1 \widetilde{\bf u} = 0$, $\widetilde{\bf u}(T, \cdot) = {\cal A}^{- 1}(T){\bf u}_T$.
Applying Lemma \ref{osservabilita cal L2} to the time interval $\omega_1 := (\alpha_1, \beta_1) \subset \omega$ one gets  
\begin{equation}\label{gargamella 0}
\int_0^T \int_{\omega_1} |\widetilde{\bf u}(t, x)|^2\, d x\, dt \geq C_6(T, \omega_1) \| \widetilde{\bf u}_T\|_{0}^2\,.
\end{equation}
Recalling \eqref{lalalala}, \eqref{definizione cal A inverso} and performing the change of variable $x = y + \widetilde \alpha(T, y)$, one has 
\begin{align}
& \| \widetilde{\bf u}_T\|_{0}^2 = \int_\T (1 + \widetilde \alpha_y(T, y)) |{\bf u}_T( y + \widetilde \alpha(T, y))|^2\, d y \nonumber\\
& = \int_\T \Big(1 + \widetilde \alpha_y(T, x + \alpha(T, x)) \Big)\Big( 1 + \alpha_x(T,x) \Big) |{\bf u}_T( x)|^2\,d x 
= \int_\T |{\bf u}_T(x)|^2\,d x = \| {\bf u}_T\|_{0}^2\,. \label{gargamella 1}
\end{align}
By \eqref{stime alpha alpha tilde} (applied with $s_0 \geq 1$), and using the standard Sobolev embedding, we get that for some $\sigma \in \N$ large enough
$$
\| \widetilde \alpha\|_{L^\infty_T L^\infty_x} \lesssim N_T(\sigma) \lesssim \eta.
$$ 
Hence, for some constant $C > 0$, 
$$
\big\{ (t, y + \tilde \alpha(t, y)) : t \in [0, T]\,, \ y \in \omega_1 \big\} 
\subset [0, T] \times [\alpha_1 - C \eta , \beta_1 + C \eta ] \subset [0, T] \times \omega
$$
for $\eta \in (0, 1)$ small enough. 
Then, using the change of variables $x = y + \widetilde \alpha(t, y)$ and \eqref{lalalala}, 
\begin{align}
\int_0^T \int_{\omega_1} |\widetilde{\bf u}(t, y)|^2\,d y \, d t 
& = \int_0^T \int_{\omega_1}(1 + \widetilde \alpha_y(t, y)) |{\bf u}(t, y + \widetilde \alpha(t, y))|^2\,d y \, d t  
\nonumber\\& 
\leq \int_0^T \int_\omega |{\bf u}(t, x)|^2\,d x\,dt \,. \label{gargamella 2}
\end{align}
The claimed inequality follows by \eqref{gargamella 0}, \eqref{gargamella 1}, \eqref{gargamella 2} by choosing $C_7:= C_6(T, \omega_1)$. 
\end{proof}

\begin{lemma}[Observability for ${\cal L} = \partial_t {\mathbb I}_2 + \ii (\Sigma + {A_2} ) \partial_{xx} + \ii {A_1} \partial_x + \ii A_0$]
\label{osservabilita cal L1}
Let $T > 0$, let $\omega \subset \T$ be a non-empty open set and let ${\cal L}$ be the operator defined in \eqref{operatore lineare generale}. 
Then there exist $\eta \in (0, 1)$ small enough and $\sigma \in \N$ large enough such that if $N_T(\sigma) \leq \eta$ then the following holds: let ${\bf u}_T \in {\bf L}^2(\T)$ and ${\bf u}(t, x)$ be the solution of the backward Cauchy problem 
\begin{equation}\label{Cauchy problem cal L}
\partial_t {\bf u} + \ii  (\Sigma + A_2)\partial_{xx} {\bf u} + \ii A_1 \partial_x {\bf u} + \ii A_0 {\bf u} = 0\,, \qquad {\bf u}(T, \cdot) = {\bf u}_T(\cdot)\,.
\end{equation}
Then there exists a constant $C_8 := C_8(T, \omega) > 0$ (independent of ${\bf u}_T$) such that 
$$
\int_0^T \int_\omega |{\bf u}(t, x)|^2\, d x \, d t \geq C_8 \| {\bf u}_T \|_{0}^2\,.
$$
\end{lemma}

\begin{proof}
Lemma \ref{buona positura equazione lineare 6} guarantees that if ${\bf u}_T \in {\bf L}^2(\T)$ then there exists a unique solution ${\bf u} \in C([0, T],$ ${\bf L}^2(\T))$ of the Cauchy problem \eqref{Cauchy problem cal L}. 
In Section \ref{coniugio step 1} we have proved that the transformation ${\cal S}$ defined in \eqref{matrice autovettori inversa} conjugates the operator ${\cal L}$ defined in \eqref{operatore lineare generale} to the operator $ {\cal L}_0$ defined in \eqref{cal L1}. 
Hence ${\bf u}$ solves the Cauchy problem 
$$
{\cal L} {\bf u} = 0\,, \qquad {\bf u}(T, \cdot) = {\bf u}_T
$$
if and only if $\widetilde{\bf u}(t, x) := {\cal S}^{- 1}(t){\bf u}(t, x)$ solves 
${\cal L}_0 \widetilde{\bf u} = 0$, $\widetilde{\bf u}(T, \cdot) =  {\cal S}^{- 1}(T){\bf u}_T$.
By Lemma \ref{osservabilita cal L1a}, 
\begin{equation}\label{gargamella 0z}
\int_0^T \int_{\omega} |\widetilde{\bf u}(t, x)|^2\, d x\, dt \geq C_7 \| \widetilde{\bf u}_T\|_{0}^2\,.
\end{equation}
Applying \eqref{stima cal S pm 1} and the ansatz \eqref{ansatz operatore astratto}, together with Sobolev embeddings, there exists $\sigma \in \N$ large enough such that 
\begin{equation}\label{stima cal S nella osservabilita}
\| {\cal S}^{\pm 1} - {\mathbb I}_2 \|_{L^\infty} \lesssim N_T(\sigma) \leq C \eta\,,  \quad \| {\cal S}^{\pm 1}\|_{L^\infty} \leq 2
\end{equation}
for $\eta \in (0, 1)$ small enough.
Therefore, recalling \eqref{definizione cal A inverso} and performing the change of variable $x = y + \widetilde \alpha(T, y)$, provided that $\eta$ is small enough, one has 
\begin{align}
\| \widetilde{\bf u}_T\|_{0}^2 & = \int_\T  |{\cal S}^{- 1}(T, x){\bf u}_T(x)|^2\, d x  \nonumber\\
& \stackrel{\eqref{stima cal S nella osservabilita}}{\geq}(1 - C^2 \eta^2)\int_\T |{\bf u}_T(x)|^2\, d x \geq \frac12 \| {\bf u}_T\|_{0}^2 \,. \label{gargamella 1z}
\end{align}
Moreover, using again \eqref{stima cal S nella osservabilita},  
\begin{align}
\int_0^T \int_{\omega} |\widetilde{\bf u}(t, y)|^2\,d y \, d t & = \int_0^T \int_{\omega} |{\cal S}^{- 1}(t, x){\bf u}(t, x)|^2\,d y \, d t  \leq 2  \int_0^T \int_\omega |{\bf u}(t, x)|^2\,d x\,dt \,. \label{gargamella 2z}
\end{align}
The claimed inequality follows by \eqref{gargamella 0z}, \eqref{gargamella 1z}, \eqref{gargamella 2z} and taking $C_8 := C_7/4$. 
\end{proof}

\section{Controllability}\label{sezione controllabilita}
In this Section we prove the controllability of linear operators ${\cal L}$ of the form \eqref{operatore lineare generale}, namely
$$
{\cal L}= \partial_t {\mathbb I}_2 + \ii (\Sigma + {A_2} ) \partial_{xx} + \ii {A_1} \partial_x + \ii A_0
$$ 
where the vector field $L(t) = - \ii\big( (\Sigma + {A_2} ) \partial_{xx} +  {A_1} \partial_x +  A_0 \big) $ is Hamiltonian and $A_2, A_1, A_0$ satisfy hypotheses \eqref{Sigma Ak}-\eqref{ansatz operatore astratto}. 
We define the operator ${\cal L}^*$ as 
\begin{equation}\label{definizione operatore cal P2 trasposto}
{\cal L}^* := - \partial_t {\mathbb I}_2 - \ii (\Sigma +  [A_2]^*) \partial_{xx} -  \ii \widetilde A_1 \partial_x - \ii \widetilde A_0 \,,
\end{equation}
where 
\begin{equation}\label{widetilde A1(1) R2}
 \quad \widetilde A_1 := 2 \partial_x [A_2]^* - [A_1]^*\,, \quad \widetilde A_0 :=  \partial_{xx} [A_2]^*  + \partial_x [A_1]^*\,.
\end{equation}
We point out that by Lemma \ref{struttura Hamiltoniana operatore aggiunto}, the time-dependent vector field $L_2^*(t) := - \ii   [A_2]^* \partial_{xx} -  \ii \widetilde A_1 \partial_x - \ii \widetilde A_0$ is still a Hamiltonian operator. Note that 
$$
\max \{ \| \widetilde A_1 \|_{T,s_0 - 1}, \|\partial_t \widetilde A_1 \|_{T,s_0 - 1},
\| \widetilde A_0 \|_{T,s_0 - 2} \} \lesssim N_T(s_0)\,,
$$
so that the operator ${\cal L}^*$ satisfies the same hyphotheses as ${\cal L}$ and the reduction procedure of Section \ref{riduzione operatori lineari generali} can be applied also to ${\cal L}^*$.

\begin{lemma}\label{controllabilita cal P2}
Let $T > 0$, let $\omega \subset \T$ be an open set. 
Let ${\cal L}^*$ be the operator defined by \eqref{definizione operatore cal P2 trasposto}.
There exists $\eta \in (0, 1)$ small enough and $ \sigma \in \N$ large enough such that, if $N_T( \sigma) \leq \eta$, then for any ${\bf h}_{in}, {\bf h}_{end} \in {\bf L}^2(\T)$, ${\bf q} \in C([0, T], {\bf L}^2(\T))$ there exists a unique function 
${\bf f} \in C([0, T], {\bf L}^2(\T))$ that solves ${\cal L}^* {\bf f} = 0$
such that the only solution ${\bf h} \in C([0, T], {\bf L}^2(\T))$ of the Cauchy problem 
\begin{equation}\label{controllo per cal P2}
\begin{cases}
{\cal L} {\bf h} = \chi_\omega {\bf f} + {\bf q} \\
{\bf h}(0, \cdot) = {\bf h}_{in} \\
\end{cases}
\end{equation}
satisfies ${\bf h}(T, \cdot) ={\bf h}_{end}$. 
Furthermore
$$
\| {\bf f}\|_{T, 0} \lesssim \| {\bf h}_{in}\|_{0} + \| {\bf h}_{end}\|_{0} + \| {\bf q}\|_{T, 0}\,.
$$
%$(ii)$ 
%The control ${\bf f}$ in $(i)$ is the unique solution of  such that the solution ${\bf h}$ of the Cauchy problem \eqref{controllo per cal P2} 
%satisfies ${\bf h}(T, \cdot) = {\bf h}_{end}$. 
\end{lemma}

\begin{proof}
\emph{(Existence)}. 
For any ${\bf f}_1, {\bf g}_1 \in {\bf L}^2(\T)$, applying Lemma \ref{buona positura equazione lineare 6}, we consider the unique solutions ${\bf f}, {\bf g} \in C([0, T], {\bf L}^2(\T))$ of the Cauchy problems 
\begin{equation}\label{tartara}
\begin{cases}
{\cal L}^* {\bf f} = 0 \\
{\bf f}(T, \cdot) = {\bf f}_1 \,,
\end{cases} \qquad \begin{cases}
{\cal L}^* {\bf g} = 0 \\
{\bf g}(T, \cdot) = {\bf g}_1 
\end{cases}
\end{equation}
and we define the bilinear form 
$$
{ B}({\bf f}_1, {\bf g}_1) := \int_0^T \langle \chi_\omega {\bf f}, {\bf g} \rangle_{{\bf L}^2}  d t
$$
and the linear form 
$$
\Lambda({\bf g}_1) :=   \langle {\bf h}_{end}, {\bf g}(T, \cdot) \rangle_{{\bf L}^2} - \langle {\bf h}_{in}, {\bf g}(0, \cdot)   \rangle_{{\bf L}^2} - \int_0^T \langle  {\bf q}(t, \cdot), {\bf g}(t, \cdot)\rangle_{{\bf L}^2} \, dt\,,
$$
where the real scalar product $\langle \cdot, \cdot \rangle_{{\bf L}^2}$ is defined in \eqref{prodotto scalare sottospazio reale}. 
By \eqref{tartara} and Lemma \ref{buona positura equazione lineare 6} we have  
$$
| B({\bf f}_1, {\bf g}_1) | \lesssim \| {\bf f}_1\|_{0} \| {\bf g}_1\|_{0}\,, \quad |\Lambda({\bf g}_1)| \lesssim (\| {\bf h}_{in}\|_{0} + \| {\bf h}_{end}\|_{0} + \| {\bf q}\|_{T, 0} ) \| {\bf g}_1\|_{0}\,.
$$
By Lemma \ref{osservabilita cal L1}, the bilinear form $B$ is coercive and therefore, 
by Riesz representation theorem (or Lax-Milgram lemma), 
there exists a unique ${\bf f}_1 \in {\bf L}^2(\T)$ such that 
\begin{equation}\label{bistecca}
B({\bf f}_1, {\bf g}_1) = \Lambda({\bf g}_1) \quad \forall {\bf g}_1 \in {\bf L}^2(\T),
\end{equation}
satisfying $\| {\bf f}_1\|_{0} \lesssim \| \Lambda \|_{{\cal L}({\bf L}^2, \C)} \lesssim \| {\bf h}_{in}\|_{0} + \| {\bf h}_{end}\|_{0} + \| {\bf q}\|_{T, 0}$. 
Now let ${\bf f}_1$ be the only solution of \eqref{bistecca} and let ${\bf h}$ be the solution of the Cauchy problem \eqref{controllo per cal P2} 
(whose existence follows by Lemma \ref{buona positura equazione lineare 6}). 
We have 
\begin{align}
0 & = B({\bf f}_1, {\bf g}_1) - \Lambda({\bf g}_1) \nonumber\\
& = \int_0^T \langle \chi_\omega {\bf f}, {\bf g} \rangle_{{\bf L}^2}  d t  - \langle {\bf h}_{end}, {\bf g}(T, \cdot) \rangle_{{\bf L}^2} + \langle {\bf h}_{in}, {\bf g}(0, \cdot)   \rangle_{{\bf L}^2} + \int_0^T \langle  {\bf q}(t, \cdot), {\bf g}(t, \cdot)\rangle_{{\bf L}^2} \, dt \nonumber\\
& \stackrel{\eqref{controllo per cal P2}}{=} \int_0^T \langle {\cal L} {\bf h}, {\bf g} \rangle_{{\bf L}^2}\, d t  - \langle {\bf h}_{end}, {\bf g}(T, \cdot) \rangle_{{\bf L}^2} + \langle {\bf h}_{in}, {\bf g}(0, \cdot)   \rangle_{{\bf L}^2} \nonumber\\
& = \int_0^T \langle{\bf u}, {\cal L}^* {\bf g} \rangle_{{\bf L}^2}\,d t + \langle {\bf h}(T, \cdot), {\bf g}(T, \cdot) \rangle_{{\bf L}^2} - \langle {\bf h}(0, \cdot), {\bf g}(0, \cdot) \rangle_{{\bf L}^2} -  \langle {\bf h}_{end}, {\bf g}(T, \cdot) \rangle_{{\bf L}^2} + \langle {\bf h}_{in}, {\bf g}(0, \cdot) \rangle_{{\bf L}^2} \nonumber\\
& \stackrel{\eqref{tartara}}{=} \langle {\bf h}(T, \cdot) - {\bf h}_{end}, {\bf g}_1 \rangle_{{\bf L}^2}\,. \nonumber
\end{align}
Then for any ${\bf g}_1 \in {\bf L}^2(\T)$ we have that $\langle {\bf h}(T, \cdot) - {\bf h}_{end} , {\bf g}_1 \rangle_{{\bf L}^2} = 0$, implying that ${\bf h}(T, \cdot) = {\bf h}_{end}$ and then the lemma follows. 

\emph{(Uniqueness)}. 
Assume that $\widetilde {\bf f} \in C([0, T], {\bf L}^2(\T))$ satisfies ${\cal L}^* \widetilde{\bf f} = 0$, and that the solution ${\bf h}$ of the Cauchy problem 
${\cal L} {\bf h} = \chi_\omega \widetilde {\bf f} + {\bf q}$, ${\bf h}(0, \cdot) = {\bf h}_{in}$
satisfies ${\bf h}(T, \cdot) = {\bf h}_{end}$. 
Setting $\widetilde{\bf f}_1 := \widetilde{\bf f}(T, \cdot)$ and arguing as above, one sees that $B(\widetilde{\bf f}_1, {\bf g}_1) = \Lambda({\bf g}_1)$ for all ${\bf g}_1 \in {\bf L}^2(\T)$, 
and then, by uniqueness of the solution ${\bf f}_1$ of \eqref{bistecca}, 
we deduce $\widetilde{\bf f}_1 = {\bf f}_1$.
\end{proof}

\begin{lemma}[Higher regularity]\label{regolarita H2 per cal P2} 
Assume the hypotheses of Lemma \ref{controllabilita cal P2}, and $N_T(\s + 2) \leq 1$. 
Let $s \in [0, S - \s-1]$, and assume that $N_T(s + 1 + \s) < \infty$. 
If  ${\bf h}_{in}, {\bf h}_{end} \in {\bf H}^s(\T)$, ${\bf q} \in C ([0, T], {\bf H}^s(\T))$, 
then ${\bf h}, {\bf f} \in C([0, T], {\bf H}^s(\T))$ and  
$$
\| {\bf f}\|_{T, s}, \| {\bf h}\|_{T, s} 
\lesssim_s \| \phi\|_{T, s} + N_T(s + \sigma) \| \phi\|_{T, 0} \,, 
\qquad \phi := ({\bf q}, {\bf h}_{in}, {\bf h}_{end})\,. 
$$ 
Furthermore, if ${\bf h}_{in}, {\bf h}_{end} \in {\bf H}^{s + 4}(\T)$, ${\bf q} \in C\big([0, T], {\bf H}^{s + 4}(\T)  \big) \cap C^1([0, T], {\bf H}^s(\T))$, then 
$$
{\bf h}, {\bf f} \in C([0, T], {\bf H}^{s + 4}(\T)) \cap C^1([0, T], {\bf H}^{s + 2}(\T)) \cap C^2([0, T], {\bf H}^s(\T))\,, 
$$ 
and 
\begin{align}
\| {\bf h}, {\bf f}\|_{T, s+ 4},\| \partial_t {\bf h}, \partial_t {\bf f}\|_{T, s + 2 }, \| \partial_{tt} {\bf h}, \partial_{tt } {\bf f}\|_{T, s}
& \lesssim_s \| \phi\|_{T, s + 4} + \| \partial_t {\bf q}\|_{T, s} + N_T(s + \sigma) \| {\bf \phi}\|_{T, 4}\,.  \label{stima derivate bf h bf f}
\end{align}
\end{lemma}

\begin{proof}
Assume that ${\bf h}, {\bf f} \in C([0, T], {\bf L}^2(\T))$ are the solutions of 
\begin{equation}\label{gomorra scampia 10}
\begin{cases}
{\cal L} {\bf h} = \chi_\omega {\bf f} + {\bf q} \\
{\bf h}(0, \cdot) = {\bf h}_{in } \\
{\bf h}(T, \cdot) = {\bf h}_{end}\,,
\end{cases} \qquad {\cal L}^* {\bf f} = 0\,. 
\end{equation}
By the results of Section \ref{riduzione operatori lineari generali}, one has that 
\begin{equation}\label{espressione completa cal P2 cal L4}
{\cal L} = 
\Phi \mL_4 \Psi, \quad \ 
\Phi := 
{\cal S}({\cal A}{\mathbb I}_2)  ({\cal B} {\mathbb I}_2) \rho   
({\cal T} {\mathbb I}_2) {\cal M}, \quad 
\Psi := {\cal M}^{- 1} ({\cal T}^{- 1} {\mathbb I}_2) ({\cal B}^{- 1} {\mathbb I}_2) ({\cal A}^{- 1} {\mathbb I}_2) {\cal S}^{- 1},
\end{equation}
with ${\cal L}_4 = \partial_t {\mathbb I}_2 + \ii \mu \Sigma \partial_{xx} + {\cal R}$, and ${\cal R} \in C([0, T], {\bf H}^s(\T))$ is the multilplication operator given by \eqref{cal R r1 r2}. 
We define the \emph{adjoint operator}
$$
{\cal L}_4^* := - \partial_t - \ii \mu \Sigma \partial_{xx} + {\cal R}^*\,, 
$$
where ${\cal R}^*$ 
is the adjoint of the multiplication operator ${\cal R}$ with respect to the scalar product $\langle \cdot , \cdot \rangle_{{\bf L}^2}$, namely, recalling \eqref{cal R r1 r2}, 
\begin{equation}\label{espressione cal R*}
{\cal R}^* = 
\begin{pmatrix}
\overline r_1 & r_2 \\
\overline r_2 & r_1
\end{pmatrix}. 
\end{equation}
Now we define 
\begin{alignat}{3} \label{bacetti 20}
\widetilde{\bf h} & := \Psi h  \qquad &
\widetilde{\bf h}_{in} & := \Psi |_{t = 0} \, {\bf h}_{in}  \qquad &
\widetilde{\bf h}_{end} & := \Psi |_{t = T} \, {\bf h}_{end}  \\
\widetilde{\bf q} & := \Phi^{-1} {\bf q} \qquad & 
\widetilde{\bf f} & := \Phi_* {\bf f} \qquad &
K & := \Phi^{-1} \chi_\om (\Phi_*)^{-1}, 
\label{bacetti 21}
\end{alignat}
where $\Phi_*$ is the adjoint of $\Phi$ with respect to the \emph{time-space} scalar product 
$\langle \cdot , \cdot \rangle_{(t,x)} := \int_0^T \langle \cdot , \cdot \rangle_{{\bf L}^2} \, dt$. 
We call ``time-space adjoint'' the adjoint of an operator with respect to $\langle \cdot , \cdot \rangle_{(t,x)}$. 
By \eqref{cal S inverso aggiunto}, \eqref{definizione cal A inverso}, 
\eqref{operatori traslazione dello spazio}, \eqref{uguaglianza cal M aggiunto inverso}, 
the adjoint operators (with respect to the ${\bf L}^2$ scalar product)
of $\mS, \mA, \mT, \mM$ are
\begin{equation}\label{gomorra scampia}
{\cal S}^* = {\cal S}, \quad {\cal A}^* = {\cal A}^{-1}, \quad  
{\cal T}^* = {\cal T}^{-1}, \quad {\cal M}^* = {\cal M}^{-1}
\end{equation}
at each fixed $t \in [0,T]$, and therefore, integrating over $[0,T]$, 
the equalities in \eqref{gomorra scampia} also hold for the time-space adjoint operators
$\mS_*, \mA_*, \mT_*, \mM_*$.
The time-space adjoint of $\mB$ satisfies $\mB_* = \rho^{-1} \mB^{-1}$
(see \eqref{0807.2}), and therefore, from the definitions of $\Phi, \Psi$ in
\eqref{espressione completa cal P2 cal L4}, 
we calculate $\Phi_* = {\cal M}^{- 1} ({\cal T}^{- 1} {\mathbb I}_2) ({\cal B}^{- 1} {\mathbb I}_2) ({\cal A}^{- 1} {\mathbb I}_2) {\cal S}$.	
We also calculate 
$$
K = {\cal M}^{- 1} ({\cal T}^{- 1} {\mathbb I}_2) \rho^{- 1} ({\cal B}^{- 1} {\mathbb I}_2) ({\cal A}^{- 1} {\mathbb I}_2) {\cal S}^{- 1}\chi_\omega {\cal S} ({\cal A} {\mathbb I}_2) ({\cal B} {\mathbb I}_2) ({\cal T} {\mathbb I}_2) {\cal M}\,.
$$ 
Since $[\mS, \chi_\om \mathbb{I}_2] = 0$ 
and $[\mM , k \mathbb{I}_2] = 0$ for all real-valued functions $k(t,x)$, 
using the conjugation rules \eqref{regola coniugazione cal A}, \eqref{regola coniugazione cal B}, \eqref{regola coniugazione cal T}, 
and recalling also \eqref{definizione cal A}-\eqref{definizione cal A inverso},
one can easily see that $K$ is the multiplication operator 
\begin{equation}\label{espressione K moltiplicazione}
K = k(t,x) \mathbb{I}_2, \quad 
k(t,x) := ({\cal T}^{- 1} \rho^{-1})(t) \, 
(\mT^{-1} \mB^{-1} A_\a^{-1} \chi_\om)(t,x).
\end{equation}
By the estimates of Section \ref{riduzione operatori lineari generali}, we get 
\begin{equation}\label{stime K moltiplicazione}
\| K {\bf h}\|_{T, s} \lesssim_s 
\| {\bf h}\|_{T, s} + N_T(s + \sigma) \| {\bf h}\|_{T, 0} 
\quad \forall {\bf h} \in C([0, T], {\bf H}^s(\T))\,.
\end{equation}
Note that, by the estimates of Section \ref{riduzione operatori lineari generali}, one has that if ${\bf h}_{in}, {\bf h}_{end} \in {\bf H}^s(\T)$, ${\bf q} \in C([0, T], {\bf H}^s(\T))$, then $\widetilde{\bf h}_{in}, \widetilde{\bf h}_{end} \in {\bf H}^s(\T)$, $\widetilde{\bf q} \in C([0, T], {\bf H}^s(\T))$. Moreover using that ${\bf h}, {\bf f} \in C([0, T], {\bf L}^2(\T))$, one has that also $\widetilde{\bf h}, \widetilde{\bf f}, K \widetilde{\bf f} \in C([0, T], {\bf L}^2(\T))$. 
By construction, $\widetilde{\bf h}, \widetilde{\bf f}$ satisfy 
\begin{equation}\label{equazioni trasformate regolarita H2}
\begin{cases}
{\cal L}_4 \widetilde{\bf h} = K \widetilde{\bf f} + \widetilde{\bf q} \\
\widetilde{\bf h}(0, \cdot) = \widetilde{\bf h}_{in} \\
\widetilde{\bf h}(T, \cdot) = \widetilde{\bf h}_{end}\,,
\end{cases} \qquad {\cal L}_4^* \widetilde{\bf f} = 0\,.
\end{equation}
To prove that ${\cal L}_4^* \widetilde{\bf f} = 0$ it is enough to write it in its weak form, namely 
$$
\langle \widetilde{\bf f}(T, \cdot), {\bf v}(T, \cdot) \rangle_{{\bf L}^2} - \langle \widetilde{\bf f}(0, \cdot), {\bf v}(0, \cdot) \rangle_{{\bf L}^2} = \int_0^T \langle \widetilde{\bf f}, {\cal L}_4 {\bf v} \rangle_{{\bf L}^2}\,d t \quad \ \forall {\bf v} \in C^\infty([0, T] \times \T)
$$
and to apply the changes of coordinates in the integrals. 

Now we show that $\widetilde{\bf h}, \widetilde{\bf f} \in C([0, T], {\bf H}^s(\T))$. 
We adapt an argument used by Dehman-Lebeau \cite{Dehman-Lebeau}, 
also used in \cite{Laurent}, \cite{ABH}, \cite{BFH}. 
We split the proof into two parts.

\medskip

\noindent
{\sc Proof in the case ${\bf h}_{end} = 0$, ${\bf q} = 0$.} 
Define the map 
\begin{equation}\label{definizione mappa cal S regolarita}
S: {\bf L}^2(\T) \to {\bf L}^2(\T)\,, \qquad { S} \widetilde{\bf f}_1 := \widetilde{\bf h}(0, \cdot)\,,
\end{equation}
where $ \widetilde{\bf f}$ and $\widetilde{\bf h}$ are the solutions of the Cauchy problems
\begin{equation}\label{bacetti 0}
\begin{cases}
{\cal L}_4^* \widetilde{\bf f} = 0 \\
\widetilde{\bf f}(T, \cdot) = \widetilde{\bf f}_1\,,
\end{cases}
\qquad \begin{cases}
{\cal L}_4 \widetilde{\bf h} = K \widetilde{\bf f} \\
\widetilde{\bf h}(T, \cdot) = 0\,.
\end{cases}
\end{equation}
By existence and uniqueness in Lemma \ref{controllabilita cal P2}, it follows that ${S}$ is a linear isomorphism. Then for every initial datum $\widetilde{\bf h}_{in} \in {\bf L}^2(\T)$ there exists a unique $\widetilde{\bf f}_1 \in {\bf L}^2(\T)$ such that ${S} \widetilde{\bf f}_1 = \widetilde{\bf h}_{in}$. 
Note that $\| \Lambda^s \widetilde{\bf f}_1\|_{L^2_x} \lesssim \|{S} \Lambda^s \widetilde{\bf f}_1 \|_{L^2_x} $, since ${S} : {\bf L}^2(\T) \to {\bf L}^2(\T)$ is an isomorphism, where $\Lambda := {\rm Op}\big( (1 + \xi^2)^{\frac12}\big)$. 
To study the commutator $[\Lambda^s, {S}]$, we have to compare $(\Lambda^s \widetilde{\bf u}, \Lambda^s \widetilde{\bf f})$ with $(\underline{\bf h}, \underline{\bf f})$ solving the Cauchy problems 
\begin{equation}\label{Cauchy problem u f underline}
\begin{cases}
{\cal L}_4^* \underline{\bf f} = 0 \\
\underline{\bf f}(T, \cdot) = \Lambda^s \widetilde{\bf f}_1\,,
\end{cases} \qquad \begin{cases}
{\cal L}_4 \underline{\bf h} = K \underline{\bf f} \\
\underline{\bf h}(T, \cdot) = 0\,.
\end{cases}
\end{equation}
Since $[\mL_4^* , \Lm^s] = [\mR^*, \Lm^s]$, 
the difference $\Lambda^s \widetilde{\bf f} - \underline{\bf f}$ satisfies  
$$
\begin{cases}
{\cal L}_4^* \big( \Lambda^s \widetilde{\bf f} - \underline{\bf f} \big) = [{\cal R}^*, \Lambda^s] \widetilde{\bf f} \\
\big( \Lambda^s \widetilde{\bf f} - \underline{\bf f} \big)(T, \cdot) = 0\,.
\end{cases}
$$
By Lemma \ref{buona positura equazione lineare 1}, 
and then using Lemma \ref{commutatore moltiplicazione fourier multiplier}, \eqref{espressione cal R*}, \eqref{stima cal R linearizzato ridotto}, one gets the estimate 
\begin{align}
\|\Lambda^s \widetilde{\bf f} - \underline{\bf f}  \|_{T, 0} 
& \lesssim \| \, [{\cal R}^*, \Lambda^s] \widetilde{\bf f} \, \|_{T, 0} 
\lesssim_s
N_T(s + \sigma) \| \widetilde{\bf f}\|_{T, 0} + \| \widetilde{\bf f}\|_{T, s - 1}\,,  \label{uhhp}
\end{align}
for some constant $\sigma > 0$, 
where we have used that $N_T(\sigma) \lesssim 1$. 
The difference $\Lambda^s \widetilde{\bf h} - \underline{\bf h}$ satisfies the Cauchy problem 
$$
\begin{cases}
{\cal L}_4 \big( \Lambda^s \widetilde{\bf h} - \underline{\bf h}  \big) = K \big( \Lambda^s \widetilde{\bf f} - \underline{\bf f} \big) + [{\cal R}, \Lambda^s] \widetilde{\bf h} + [\Lambda^s , K] \widetilde{\bf f}  \\
(\Lambda^s \widetilde{\bf h} - \underline{\bf h})(T, \cdot) = 0\,.
\end{cases}
$$ 
Arguing as in \eqref{uhhp} one gets 
$$
\|  [{\cal R}, \Lambda^s] \widetilde{\bf h}\|_{T, 0} 
\lesssim_s N_T(s + \sigma) \| \widetilde{\bf h}\|_{T, 0} + \| \widetilde{\bf h}\|_{T, s - 1}\,.
$$
Since $K$ is a multiplication operator (see \eqref{espressione K moltiplicazione}), 
the commutator $[\Lambda^s , K]$ is of order $s - 1$.  
By \eqref{stime K moltiplicazione}, using again Lemma \ref{commutatore moltiplicazione fourier multiplier}, we deduce that 
$$
\|K ( \Lambda^s \widetilde{\bf f} - \underline{\bf f} ) \|_{T, 0} \lesssim \| \Lambda^s \widetilde{\bf f} - \underline{\bf f} \|_{T, 0}\,, \qquad \|  [\Lambda^s , K] \widetilde{\bf f}  \|_{T, 0} \lesssim \| \widetilde{\bf f}\|_{T, s - 1} + N_T(s + \sigma) \| \widetilde{\bf f}\|_{T, 0}\,.
$$
Therefore, by Lemma \ref{buona positura equazione lineare 1}, 
\begin{align}
\| \Lambda^s \widetilde{\bf h} - \underline{\bf h}\|_{T, 0} 
& \lesssim \|  [{\cal R}, \Lambda^s] \widetilde{\bf h}\|_{T, 0} 
+ \|K \big( \Lambda^s \widetilde{\bf f} - \underline{\bf f} \big) \|_{T, 0} 
+ \| [\Lambda^s , K] \widetilde{\bf f}  \|_{T, 0} 
\nonumber\\
& \lesssim  N_T(s + \sigma) \| \widetilde{\bf h}\|_{T, 0} 
+ \| \widetilde{\bf h}\|_{T, s - 1} 
+ \| \Lambda^s \widetilde{\bf f} - \underline{\bf f} \|_{T, 0} 
+ \| \widetilde{\bf f}\|_{T, s - 1} 
+ N_T(s + \sigma) \| \widetilde{\bf f}\|_{T, 0} 
\nonumber\\
& \stackrel{\eqref{uhhp}}{\lesssim}  \| \widetilde{\bf h}\|_{T, s - 1} +\|\widetilde{\bf f}\|_{T, s - 1} + N_T(s + \sigma)\big( \| \widetilde{\bf h}\|_{T, 0} +\|\widetilde{\bf f}\|_{T, 0} \big)  \,.  \label{stima tilde u - underline u 1}
\end{align}
Applying Lemma \ref{buona positura equazione lineare 1} to the Cauchy problems \eqref{bacetti 0},
and using also \eqref{stime K moltiplicazione}, we have
\begin{align} 
\| \widetilde{\bf f}\|_{T, s} 
& \lesssim \| \widetilde{\bf f}_1 \|_{ s} + N_T(s + \sigma) \| \widetilde{\bf f}_1 \|_{0}\,, 
\notag\\ 
\| \widetilde{\bf h}\|_{T, s} 
& \lesssim \|K \widetilde{\bf f} \|_{T, s} + N_T(s + \sigma) \|K \widetilde{\bf f} \|_{T, 0}  
\lesssim
\| \widetilde{\bf f}_1 \|_{ s} + N_T(s + \sigma) \| \widetilde{\bf f}_1 \|_{0} \,.
\label{bacetti 3}
\end{align}
Hence estimates \eqref{uhhp}, \eqref{stima tilde u - underline u 1} become
\begin{equation}\label{bacetti 1}
\|\Lambda^s \widetilde{\bf f} - \underline{\bf f}  \|_{T, 0}, \| \Lambda^s \widetilde{\bf h} - \underline{\bf h}\|_{T, 0} \lesssim_s \| \widetilde{\bf f}_1 \|_{ s - 1} + N_T(s + \sigma) \| \widetilde{\bf f}_1 \|_{0} \,.
\end{equation}
By the definition of the map ${S}$ in \eqref{definizione mappa cal S regolarita}, one has 
$\underline{\bf h}(0, \cdot) = {S} \Lambda^s \widetilde{\bf f}_1$. 
Also recall that we have fixed ${S} \widetilde{\bf f}_1 = \widetilde{\bf h}_{in } = \widetilde{\bf h}(0, \cdot)$. 
Using \eqref{bacetti 1} and triangular inequality, 
\begin{align}
\| {S} \Lambda^s \widetilde{\bf f}_1 \|_{0} 
& \lesssim \| \Lambda^s \widetilde{\bf h}(0, \cdot)\|_{0} + \| \Lambda^s \widetilde{\bf h}(0, \cdot) - \underline{\bf h}(0, \cdot)\|_{0} 
\nonumber \\ & 
\lesssim \| \widetilde{\bf h}_{in}\|_{s} + \| \Lambda^s \widetilde{\bf h} - \underline{\bf h}\|_{T, 0}  
\lesssim \| \widetilde{\bf h}_{in}\|_{s} +  \| \widetilde{\bf f}_1 \|_{ s - 1} + N_T(s + \sigma) \| \widetilde{\bf f}_1 \|_{0} \,.
\label{bacetti 2}
\end{align}
Since ${S} : {\bf L}^2(\T) \to {\bf L}^2(\T)$ is a linear isomorphism, we have $ \| \widetilde{\bf f}_1\|_s \simeq\| \Lambda^s \widetilde{\bf f}_1\|_{0} \lesssim \| {S} \Lambda^s \widetilde{\bf f}_1 \|_{0} $ and therefore, by \eqref{bacetti 2},
$$
\| \widetilde{\bf f}_1\|_{s} \lesssim \| \widetilde{\bf h}_{in}\|_{s} +  \| \widetilde{\bf f}_1 \|_{ s - 1} + N_T(s + \sigma) \| \widetilde{\bf f}_1 \|_{0}\,.
$$
Using again that ${S} :{\bf L}^2(\T) \to {\bf L}^2(\T)$ is an isomorphism, 
we have $\| \widetilde{\bf f}_1\|_{0} \lesssim \| \widetilde{\bf h}_{in}\|_{0}$, 
and the above inequality becomes 
\begin{equation} \label{1207.1}
\| \widetilde{\bf f}_1\|_{s} \lesssim \| \widetilde{\bf h}_{in}\|_{s} +  N_T(s + \sigma) \| \widetilde{\bf h}_{in} \|_{0} +  \| \widetilde{\bf f}_1 \|_{ s - 1} \,.
\end{equation}
If $0 < s \leq 1$, then $\| \widetilde{\bf f}_1 \|_{ s - 1} 
\leq \| \widetilde{\bf f}_1 \|_0$, and, as already observed, 
$\| \widetilde{\bf f}_1\|_{0} \lesssim \| \widetilde{\bf h}_{in}\|_{0}$, 
whence 
\begin{equation} \label{1207.2}
\| \widetilde{\bf f}_1\|_{s} \lesssim \| \widetilde{\bf h}_{in}\|_{s} +  N_T(s + \sigma) \| \widetilde{\bf h}_{in} \|_{0}\,.
\end{equation}
If $s > 1$, bound \eqref{1207.2} is proved by induction on $s$, 
applying \eqref{1207.1} repeatedly.
Hence, by \eqref{bacetti 3}, 
\begin{equation}\label{widetilde bf h bf f bf h in}
\| \widetilde{\bf h}\|_{T, s}, \| \widetilde{\bf f} \|_{T, s} \lesssim_s \| \widetilde{\bf h}_{in}\|_{s} +  N_T(s + \sigma) \| \widetilde{\bf h}_{in} \|_{0}\,.
\end{equation}
Finally, recalling \eqref{espressione completa cal P2 cal L4}, \eqref{bacetti 20}-\eqref{bacetti 21}, \eqref{gomorra scampia} and the estimates \eqref{stima cal S pm 1}, \eqref{stima cal a pm1}, \eqref{stima cal B pm1}, \eqref{stima rho step 2}, \eqref{stima cal T pm1}, \eqref{stima cal M} of Section \ref{riduzione operatori lineari generali}, 
we obtain the claimed estimate for ${\bf h}$ and ${\bf f}$, namely
\begin{equation}\label{bf h bf f bf h in}
\| {\bf h}\|_{T, s}, \| {\bf f} \|_{T, s} \lesssim_s \| {\bf h}_{in}\|_{s} +  N_T(s + \sigma) \| {\bf h}_{in} \|_{0}\,.
\end{equation}

\noindent
{\sc Proof of the general case.} Now we remove the hypothesis that ${\bf h}_{end}$ and ${\bf q}$ are zero. 
Assume that ${\bf h}, {\bf f}$ solve \eqref{gomorra scampia 10} and let ${\bf w}$ be the solution of the backward Cauchy problem 
\begin{equation}\label{gomorra scampia 0}
{\cal L} {\bf w} = {\bf q}, \quad 
{\bf w}(T, \cdot) = {\bf h}_{end}\,.
\end{equation}
Since ${\bf h}_{end} \in {\bf H}^s(\T)$ and ${\bf q} \in C([0, T], {\bf H}^s(\T))$, 
by Lemma \ref{buona positura equazione lineare 6} one has 
${\bf w} \in C([0, T], {\bf H}^s(\T))$ with 
\begin{equation}\label{gomorra scampia 1}
\| {\bf w} \|_{T, s} \lesssim_s \| {\bf q}\|_{T, s} + \| {\bf h}_{end}\|_{s} + N_T(s + \sigma)  \| {\bf h}_{end}\|_{0} \,.
\end{equation}
Let ${\bf v} := {\bf h} - {\bf w}$. Hence 
\begin{equation}\label{gomorra scampia 2}
{\cal L} {\bf v} = \chi_\omega {\bf f}, \quad 
{\bf v}(0, \cdot) = {\bf h}_{in} - {\bf w}(0, \cdot), \quad 
{\bf v}(T, \cdot) = 0
\end{equation}
and therefore ${\bf v}, {\bf f}$ solve \eqref{gomorra scampia 10} where $({\bf h}_{in}, {\bf h}_{end}, {\bf q})$ are replaced by $(0, {\bf h}_{in} - {\bf w}(0, \cdot), 0)$. 
Hence we can apply to ${\bf v}, {\bf f}$ the estimate \eqref{bf h bf f bf h in} proved in the previous step, obtaining that 
\begin{align}
\| {\bf v}\|_{T, s}\,,\, \| {\bf f} \|_{T, s} & \lesssim_s \| {\bf h}_{in} - {\bf w}(0, \cdot) \|_{s} + N_T(s + \sigma) \| {\bf h}_{in} - {\bf w}(0, \cdot) \|_{0} \nonumber\\
& \lesssim_s \| {\bf h}_{in} \|_{s} + \| {\bf w}(0, \cdot) \|_{s} + N_T(s + \sigma) \big( \| {\bf h}_{in} \|_{0}  + \| {\bf w}(0, \cdot) \|_{0} \big)  \nonumber\\
&\lesssim_s  \| {\bf h}_{in} \|_{s} + \| {\bf w} \|_{T, s} + N_T(s + \sigma) \big( \| {\bf h}_{in} \|_{0}  + \| {\bf w}\|_{T, 0} \big)\,.  \label{ciruzzo immortale 0}
\end{align}
Therefore \eqref{gomorra scampia 1}, \eqref{ciruzzo immortale 0} imply that 
\begin{equation}\label{ciruzzo immortale 2}
\| {\bf v}\|_{T, s}\,,\, \| {\bf f} \|_{T, s} \lesssim_s \| {\bf h}_{in}\|_{s} + \| {\bf h}_{end}\|_{s} + \| {\bf q}\|_{T, s} + N_T(s + \sigma) \big( \| {\bf h}_{in}\|_{0} + \| {\bf h}_{end}\|_{0} + \| {\bf q}\|_{T, 0}  \big)\,.
\end{equation}
The estimate for ${\bf h} = {\bf v} + {\bf w}$ follows by triangular inequality and 
by \eqref{gomorra scampia 1} and \eqref{ciruzzo immortale 2}. 
Estimate \eqref{stima derivate bf h bf f} is deduced from the fact that 
${\bf h}, {\bf f}$ solve the equations 
${\cal L} {\bf h} = \chi_\omega {\bf f} + {\bf q}$ and 
${\cal L}^* {\bf f} = 0$. 
\end{proof}

For any $s \in \R$, we consider the space 
$$
C([0, T], H^s(\T, \R^2)) = C([0, T], H^s(\T, \R)) \times C([0, T], H^s(\T, \R)) 
$$
and for $u = (u_1, u_2) \in C([0, T], H^s(\T, \R^2))$ we set 
$$
\| u \|_{T, s} := \| u_1\|_{T, s} + \| u_2\|_{T, s}\,.
$$
We define 
\begin{equation} \label{def Es}
E_s := X_s \times X_s , \quad 
\end{equation}
\begin{equation}\label{definizione Xs}
X_s := C([0,T], {H}^{s+ 4}(\T, \R^2)) \cap C^1([0,T], {H}^{s+2}(\T, \R^2)) \cap C^2([0,T], {H}^s(\T, \R^2))\,,
\end{equation}
and (recall notations in \eqref{2406.4}-\eqref{2706.1}),  
\begin{align} 
F_s := \Big\{  & z := (v, \alpha ,\beta) = (v_1, v_2, \a_1, \a_2, \b_1, \b_2) : 
\notag \\ & \ \ 
v \in C([0,T], {H}^{s+4}(\T, \R^2)) \cap C^1([0,T], {H}^s(\T, \R^2)), \ 
\alpha, \beta  \in H^{s + 4}(\T, \R^2) \Big\} 
\label{def Fs}
\end{align}
equipped with the norms 
\begin{equation} \label{def norm Es}
\| (u, f) \|_{E_s} := \| u \|_{X_s} + \| f \|_{X_s}, \quad 
\| {u} \|_{X_s} := \| { u} \|_{T,s+4} + \| \pa_t {u} \|_{T,s+2} + \| \pa_{tt} {u} \|_{T,s}\,, 
\end{equation}
and
\begin{equation} \label{def norm Fs}
\| z \|_{F_s} := \| v\|_{T, s + 4} + \| \partial_t v \|_{T, s}
+ \| \alpha\|_{s + 4} + \| \beta\|_{s + 4} \,.
\end{equation}
With this notation, we have proved the following linear inversion result.

\begin{theorem}[Right inverse of the linearized operator] \label{thm:inv}
Let $T>0$, and let $\om \subset \T$ be an open set.
There exist constants $\t \geq 6$, $\s \geq 3$ (independent of $T,\om$)
and $\d_* > 0$ (depending on $T,\om$) with the following property.

Let $s \in [0, r-\t]$, 
where $r$ is the regularity of the nonlinearity in \eqref{regolarita Hamiltoniana}.
Let $z = (v, \alpha, \beta) \in F_s$. 
If $(u , f) \in E_{s+\s}$, with $\| u \|_{X_\sigma}  \leq \d_*$, 
then there exists $(h, \vphi) := \Psi(u, f)[z] \in E_s$, such that 
\begin{equation} \label{3009}
P'(u)[h] - \chi_\om \vphi = v , \quad 
h(0, \cdot) = \alpha, \quad 
h(T, \cdot) = \beta,
\end{equation}
and 
\begin{equation}  \label{2809}
\| h, \vphi \|_{E_s} \leq C(s) \big(\| z \|_{F_s} + \| { u} \|_{X_{s+\s}} \| z \|_{F_0} \big)
\end{equation}
where the constant $C(s) > 0$ depends on $s,T,\om$.
\end{theorem}

\begin{proof}
Using the transformation $\mC$ defined in \eqref{2806.5}, the linear control problem \eqref{3009} for the operator $P'(u_1, u_2)$  is transformed into the linear control problem \eqref{2706.10} for the operator ${\cal L}(u_1, u_2) = \mC^{-1} P'(u_1, u_2) \mC$, where the operator ${\cal L} = {\cal L}(u_1, u_2)$ is given in \eqref{linearized operator}.  
We apply Lemma \ref{regolarita H2 per cal P2} to the control problem \eqref{2706.10}, since by definition \eqref{definizione NT sigma} and Lemma \ref{stime coefficienti linearizzato} 
the smallness condition $\| {\bf u}\|_{X_\sigma} \leq \delta_*$ implies that $N_T(\sigma') \lesssim \delta_*$, for some $\sigma' < \sigma$. 
Then the lemma follows by noticing that the map $\mC : {\bf H}^s(\T) \to H^s(\T, \R^2)$ 
is a unitary isomorphism.
\end{proof}

\section{Proofs} 
\label{sec:proof}

In this section we prove Theorems \ref{thm:1}, \ref{teorema controllo in coordinate reali} and \ref{thm:byproduct}, \ref{thm:byproduct real}. As explained in Section \ref{sezione functional setting}, Theorems \ref{thm:1} and \ref{thm:byproduct} follow by Theorems \ref{teorema controllo in coordinate reali}, \ref{thm:byproduct real}. 

\subsection{Proof of Theorems \ref{thm:1}, \ref{teorema controllo in coordinate reali}}\label{subsec:proof thm 1}

We check that all the assumptions of Theorem \ref{thm:NM} are verified.
The spaces $E_s, F_s$ defined in \eqref{def Es}-\eqref{def norm Fs}, with $s \geq 0$, 
form scales of Banach spaces. 
We define the smoothing operators $S_j$, $j = 0,1,2,\ldots$ as  
\[ 
S_j u (x) := \sum_{|k| \leq 2^j} \widehat u_k \, e^{\ii k x}
\qquad \text{where} \quad 
u(x) = \sum_{k \in \Z} \widehat u_k \, e^{\ii k x} \in L^2(\T).
\] 
The definition of $S_j$ extends in the obvious way to functions 
$u(t,x) = \sum_{k \in \Z} \widehat u_k(t) \, e^{\ii kx}$ 
depending on time. 
Since $S_j$ and $\pa_t$ commute, the smoothing operators $S_j$ are defined 
on the spaces $E_s$, $F_s$ defined in \eqref{def Es}-\eqref{def Fs} by setting 
$S_j(u,f) := (S_j u, S_j f)$ and similarly on $z = (v, \a, \b)$. 
One easily verifies that $S_j$ satisfies \eqref{S0}-\eqref{S4} and \eqref{2705.4} 
on $E_s$ and $F_s$. 

By \eqref{2406.2}, observe that $\Phi(u,f) := (P(u) - \chi_\om f, \, u(0), \, u(T) )$ 
belongs to $F_s$ 
when $(u,f) \in E_{s+2}$, $s \in [0, r-4]$, 
with $\| u \|_{T,3} \leq 1$. 
Its second derivative in the directions $(h,\ph) = (h_1, h_2, \ph_1, \ph_2)$ and 
$(w,\psi) = (w_1, w_2, \psi_1, \psi_2)$ 
is
\[
\Phi''(u,f)[(h, \ph), (w,\psi)] 
= \begin{pmatrix} P''(u)[h, w] \\ 0 \\ 0 \end{pmatrix}.
\]
For $u$ in a fixed ball $\| u \|_{X_1} \leq \d_0$, with $\d_0$ small enough, 
one has
\begin{equation} \label{stima Phi''}
\| P''(u)[h,w] \|_{F_s} 
\lesssim_s  \big( \| h \|_{X_1} \| w \|_{X_{s+2}} 
+ \| h \|_{X_{s+2}} \| w \|_{X_1}
+ \| u \|_{X_{s+2}} \| h \|_{X_1} \| w \|_{X_1} \big)
\end{equation}
for all $s \in [0, r-4]$. 
We fix $V = \{ (u,f) \in E_2 : \| (u,f) \|_{E_2} \leq \d_0 \}$,
$\d_1 = \d_*$,  
\begin{equation} \label{param.1}
a_0 = 1, \quad
\mu = 2, \quad
a_1 = \s, \quad
\a = \b > 2 \s, \quad
a_2 > 2\a - a_1,
\end{equation}
where $\d_*, \s, \t$ are given by Theorem \ref{thm:inv}, 
and $r \geq r_1 := a_2 + \t$ is the regularity of $G$ in Theorem \ref{teorema controllo in coordinate reali}. 
The right inverse $\Psi$ in Theorem \ref{thm:inv} satisfies the assumptions of Theorem \ref{thm:NM}.
Let $u_{in}, u_{end} \in H^{\b+4}(\T, \R^2)$, with $\| u_{in}, u_{end} \|_{H^{\b+4}_x}$ small enough.
Let $g := (0, u_{in}, u_{end})$, so that $g \in F_\b$ and $\| g \|_{F_\b} \leq \d$.
Since $g$ does not depend on time, it satisfies \eqref{2705.1}.

Thus by Theorem \ref{thm:NM} there exists a solution $(u,f) \in E_\a$ of the equation
$\Phi(u,f) = g$, with $\| u,f \|_{E_\a} \leq C \| g \|_{F_\b}$ 
(and recall that $\b = \a$). 
We fix $s_1 := \a + 4$, and \eqref{stimetta2} is proved. 

We have found a solution $(u,f)$ of the control problem \eqref{pb cauchy statement controllo reale}-\eqref{i101}. 
Now we prove that $u$ is the unique solution of the Cauchy problem \eqref{pb cauchy statement controllo reale}, 
with that given $f$. 
Let $u,v$ be two solutions of \eqref{i9} in $E_{s_1-4}$. 
We calculate
\[
P(u) - P(v) 
= \int_0^1 P'(v + \lm (u-v)) \, d\lm \, [u - v]. 
\]
Conjugating the operator $P'(v + \lm (u-v))$ 
by means of the unitary isomorphism $\mC : {\bf H}^s(\T) \to H^s(\T, \R^2)$ 
defined in \eqref{2806.5}, one gets 
$$
\mC^{-1} P'(v + \lm (u-v)) \, \mC = {\cal L}(v + \lm (u-v))\,,
$$ 
where ${\cal L}$ has the form \eqref{linearized operator}. 
Hence 
$$
\mC^{- 1} \int_0^1 P'(v + \lm (u-v)) \, d\lm \, \mC = \widetilde{\cal L}\,,
$$
where 
\[
\widetilde \mL := \pa_t + \ii (\Sigma  + \tilde A_2(t,x)) \pa_{xx} 
+ \ii \tilde A_1(t,x) \pa_{x}
+ \ii \tilde A_0(t,x),
\]
\[
\tilde A_i(t,x) := \int_0^1 A_i(v+\lm(u-v))(t,x) \, d\lm, \quad i = 0,1,2,
\]
and $A_i(u)$ is defined in \eqref{2706.12}-\eqref{2706.13}. 
Setting ${\bf u} := \mC^{- 1} u$, ${\bf v} := \mC^{- 1} v$ one has that the difference 
${\bf u}-{\bf v}$ satisfies $\widetilde \mL ({\bf u}-{\bf v}) = 0$, $({\bf u}- {\bf v})(0) = 0$.
We apply Lemma \ref{buona positura equazione lineare 6} to the operator $\widetilde{\cal L}$, 
and we obtain ${\bf u} - {\bf v} = 0$. Then $u - v = 0$. 
This completes the proof of Theorem \ref{teorema controllo in coordinate reali},
and therefore of Theorem \ref{thm:1}.
\qed

\subsection{Proof of Theorems \ref{thm:byproduct}, \ref{thm:byproduct real}}\label{subsec:proof thm byproduct}

We define 
\begin{align} \label{def Es bis}
E_s & := C([0,T], H^{s+4}(\T,\R^2)) \cap C^1([0,T], H^{s+2}(\T,\R^2)) 
\cap C^2([0,T], H^s(\T,\R^2)),
\\
F_s & := \{ (v, \a) : v \in C([0,T], H^{s+4}(\T, \R^2)) \cap C^1([0,T], H^s(\T, \R^2)), 
\a \in H^{s+4}(\T,\R^2) \}
\end{align}
equipped with norms 
\begin{align} \label{def norm Es bis}
\| u \|_{E_s} & := \| u \|_{T,s+4} + \| \pa_t u \|_{T,s+2} + \| \pa_{tt} u \|_{T,s}
\\
\| (v,\a) \|_{F_s} & := \| v \|_{T,s+4} + \| \pa_t v \|_{T,s} + \| \a \|_{s+4},
\end{align}
and $\Phi(u) := (P(u), u(0))$, where $P$ is defined in \eqref{2406.1}.
Given $g := (0,u_{in}) \in F_{s_0}$, 
the Cauchy problem \eqref{pb cauchy statement controllo reale pappa}
writes $\Phi(u) = g$.
We fix 
$V := \{ u \in E_2 : \| u \|_{E_2} \leq \d_0 \}$, 
where $\d_0$ is the same as in subsection \ref{subsec:proof thm 1}; 
we fix $a_0, \mu, a_1, \a, \b, a_2$ like in \eqref{param.1},
where the constants $\s,\t$ are now given in Lemma \ref{buona positura equazione lineare 6},
$r \geq r_0 := a_2 + \t$ is the regularity of $G$ in Theorem \ref{thm:byproduct real},
and $\d_1$ is small enough to satisfy both assumption \eqref{2107.1} in Lemma \ref{stime coefficienti linearizzato} and $N_T(\s) \leq \eta$ in Lemma \ref{buona positura equazione lineare 6}.

Assumption \eqref{tame in NM} about the right inverse of the linearized operator 
is satisfied by Lemmas \ref{buona positura equazione lineare 6} 
and \ref{stime coefficienti linearizzato}. 
We fix $s_0 := \a + 4$.
Then Theorem \ref{thm:NM} applies, giving the existence part of Theorem \ref{thm:byproduct real}. 
The uniqueness of the solution is proved exactly as in Subsection \ref{subsec:proof thm 1}.
This completes the proof of Theorem \ref{thm:byproduct real}, 
and therefore of Theorem \ref{thm:byproduct}.
\qed

\section{Appendix A. Quadratic Hamiltonians and linear Hamiltonian vector fields}
\label{sezione formalismo hamiltoniano}
Dealing with linear Hamiltonian equations, we develop Hamiltonian formalism only for quadratic Hamiltonians.
We consider real quadratic Hamiltonians ${\cal H} : {\bf H}^s(\T) \to \R$ of the form
\begin{equation}\label{generica hamiltoniana quadratica reale nel complesso}
{\mathcal H}(u,  \overline u) = \int_{\T} {R}_1[u]\overline u\, d x+ \frac12 \int_{\T} { R}_2[u]  u \, dx + \frac12 \int_\T \overline{{ R}_2}[ \overline u]\, \overline u\,d x\,,
\end{equation}
where $R_1, R_2 : H^s (\T)\to H^{s - 2}(\T)$ and  
\begin{equation}\label{condizione R1 R2 campo hamiltoniano complesso}
R_1= R_1^*\,, \qquad R_2 = R_2^T\,.
\end{equation}
the Hamiltonian equation associated to ${\cal H}$ is given by 
$$
\partial_t {\bf u} = \ii J \nabla_{\bf u} {\cal H}({\bf u})\,, \qquad {\bf u} = (u, \overline u) \in {\bf H}^s(\T)
$$
where 
$$
{\nabla_{\bf u}} {\cal H} := (\nabla_u {\cal H}, \nabla_{\bar u} {\cal H})\,, \qquad J := \begin{pmatrix}
0 & 1\\
- 1 & 0
\end{pmatrix}\,.
$$
Note that the Hamiltonian vector field associated to the Hamiltonian ${\cal H}$ has the form 
\begin{equation}\label{forma campo Hamiltoniano complesso}
{\cal R} = \ii J \nabla_{\bf u} {\cal H} = \ii \begin{pmatrix}
R_1 & R_2 \\
- \overline R_2 & - \overline R_1
\end{pmatrix}\,, \qquad R_1 = R_1^*\,, \quad R_2 = R_2^T\,.
\end{equation}
The symplectic form on the phase space ${\bf L}^2(\T)$ is defined as  
\begin{equation}\label{forma simplettica coordinate complesse}
{\cal W}[{\bf u}_1, {\bf u}_2] =  \ii \int_{\T} (u_1  \overline u_2 -   \overline u_1 u_2)\, dx\,, \quad \forall {\bf u}_1, {\bf u}_2 \in { \bf L}^2(\T)\,.
\end{equation}
\begin{definition}\label{definizione mappa simplettica complessa}
Let $\Phi_i = \Phi_i : H^s(\T) \to H^s(\T)$, $i = 1, 2$. We say that the map 
$$
\Phi = \begin{pmatrix}
\Phi_1 & \Phi_2 \\
\overline{\Phi_2} & \overline{\Phi_1}
\end{pmatrix}\,, 
$$ 
is symplectic if 
$$
{\cal W}[\Phi[{\bf u}_1], \Phi[{\bf u}_2]] = {\cal W}[{\bf u}_1, {\bf u}_2]\,, \quad \forall {\bf u}_1, {\bf u}_2 \in { \bf L}^2(\T)\,, 
$$
or equivalently $\Phi^T J \Phi= J$. 
\end{definition}
It is well known that if ${\mathcal R}$ is an operator of the form \eqref{forma campo Hamiltoniano complesso}, then the operators $ {\rm exp}(\pm{\mathcal R})$ are symplectic maps. 
In the next lemma we state some properties of some particular Hamiltonian vector fields. 
\begin{lemma}\label{caratterizzazione campi lineari Hamiltoniani ordine 2}
Let $a_i, b_i \in H^s(\T)$, $i = 0,1,2$ and 
$$
A_i := \begin{pmatrix}
a_i & b_i \\
- \overline b_i & - \overline a_i
\end{pmatrix}\,, \qquad i = 0,1,2\,.
$$
If the vector field ${\cal R} := \ii \big( A_2 \partial_{xx} + A_1 \partial_x + A_0 \big) : {\bf H}^s(\T) \to {\bf H}^{s - 2}(\T)$ is Hamiltonian then the following holds: 
$$
a_2 = \overline a_2\,, \quad a_1 = 2 (\partial_x a_2) - \overline a_1\,, \quad a_0 = \overline a_0 + (\partial_{xx} a_2) - (\partial_x \overline a_1)\,,\quad b_1 = (\partial_x b_2)
$$
\end{lemma}
\begin{lemma}\label{struttura Hamiltoniana operatore aggiunto}
Assume that ${\cal R}$ is a Hamiltonian operator of the form \eqref{forma campo Hamiltoniano complesso}. Then its adjoint ${\cal R}^*$ with respect to the complex scalar product $\langle \cdot, \cdot \rangle_{{\bf L}^2}$ is still a Hamiltonian operator.
\end{lemma}
\begin{proof}
Let ${\cal R}$ be a Hamiltonian operator 
$$
{\cal R} = \ii \begin{pmatrix}
R_1 & R_2 \\
- \overline R_2 & - \overline R_1
\end{pmatrix}\,, \qquad R_1 = R_1^*\,, \quad R_2 = R_2^T\,.
$$
A direct calculation shows that the adjoint ${\cal R}^*$ with respect to the complex scalar product $\langle \cdot, \cdot \rangle_{{\bf L}^2}$ is given by 
$$
{\cal R}^* =  \ii \begin{pmatrix}
 Q_1 &  Q_2 \\
- \overline Q_2 & - \overline Q_1
\end{pmatrix}\,, \qquad Q_1 := - \overline R_1^T\,, \qquad Q_2 := R_2^T\,.
$$
using that $R_1$ is selfadjoint and $\overline R_1^T = R_1^*$, we get that $Q_1= - R_1$ and therefore $Q_1 = Q_1^*$. Moreover since $R_2 = R_2^T$, we get that $Q_2 = R_2$ and therefore $Q_2 = Q_2^T$. This implies that 
$$
{\cal R}^* = \ii \begin{pmatrix}
- R_1 & R_2 \\
- \overline R_2 & \overline R_1
\end{pmatrix}
$$
is still Hamiltonian. 
\end{proof}

\section{Appendix B. Classical tame estimates}
\label{appendice paradiff}
In this appendix we recall some classical interpolation estimates used in this paper. We introduce the following notation: given $k \in \R$, we denote 
$$
\Z_{\geq k} := \{ n \in \Z : n \geq k \}, \quad 
\R_{\geq k} := \{ s \in \R : s \geq k \}, \quad 
\R_{>k} := \{ s \in \R : s > k \}.
$$

\begin{lemma}\label{lemma interpolazione}
$(i)$ \emph{(Embedding).}
For any $s \in  \Z_{\geq 0}$, the space $H^{s + 1}(\T)$ is compactly embedded in $C^s(\T)$ and 
\begin{equation}\label{Sobolev embedding}
\| u \|_{C^s} \lesssim_s \| u \|_{s + 1} \quad \forall u \in H^{s + 1}(\T)\,.
\end{equation}
$(ii)$ \emph{(Tame product).}
Let $s \in \R_{\geq 1}$ and $u_1, u_2 \in H^s(\T)$. Then
\begin{equation}\label{prtame}
\| u_1 u_2 \|_{s} 
\lesssim_s \| u_1 \|_{1} \| u_2 \|_{s} + \| u_1 \|_{s} \| u_2 \|_{1}.
\end{equation}
In particular 
\begin{equation}\label{prtame2}
\| u_1 u_2 \|_{s} \lesssim_s \| u_1 \|_{s} \| u_2 \|_{s}.
\end{equation}
$(iii)$ \emph{(Interpolation).}
Let $a_0, b_0, p, q \in \R_{\geq 0}$. Then 
\begin{align} 
& \| u_1 \|_{a_0 + p} \| u_2 \|_{b_0 + q} \leq   \| u_1 \|_{a_0 + p + q} \| u_2 \|_{b_0} + \| u_1 \|_{a_0} \| u_2 \|_{b_0 + p + q}\,. \label{interpolation estremi fine} 
\end{align}

\end{lemma}
\begin{lemma}[Composition]\label{lemma Moser}
Let $s \in \R_{\geq 0}$, $m \in \N$, with $m > s+1$.
Let $F : \C^n \to \R$ be a function of $C^m$ class in the real sense. 
Let $u \in H^s(\T, \C^n) \cap H^1(\T,\C^n)$, with $\| u \|_{1} \leq 1$. 
Then 
\begin{equation} \label{0707.2}
\| F(u) \|_{s} \lesssim_s \, 1 + \| u \|_{s}\,. 
\end{equation}
Moreover, if $F(0) = 0$, then
\begin{equation} \label{0707.1}
\| F(u) \|_s \lesssim_s \, \| u \|_s.
\end{equation}
\end{lemma}

\begin{proof}
For $s \in \N$ see \cite[p.\,272--275]{Moser-Pisa-66}
and \cite[Lemma 7, p.\,202--203]{Rabinowitz-tesi-1967}.
For the more general case of real $s$ see \cite[Theorem 5.2.6]{Metivier},
\cite[Proposition 2.2, p.\,87]{AG}, 
and \cite[Proposition 7.3 $iii$]{ABZ-ul}. 
The result in \cite{ABZ-ul} is stated in the \emph{uniformly local} Sobolev spaces 
$H^s_{ul}(\R^d)$, which contain the periodic Sobolev spaces $H^s(\T^d)$. 
The result in \cite{Metivier} is stated for $F \in C^\infty$, 
but, in fact, the proof in \cite{Metivier} only uses the assumption 
that $F$ has derivatives up to order $m > s+1$  
that are bounded on compact sets.
The proof in \cite{Metivier} is on $\R^d$,  
but it also holds on the torus $\T$ and, more generally, $\T^d$. 
The only nontrivial point when adapting that proof to $\T^d$ 
is equation (5.2.10) of \cite{Metivier}, 
which is also ``Bernstein inequality'' (4.1.8), 
which follows from Lemma 4.1.6 of \cite{Metivier}. 

We explain how to adapt Lemma 4.1.6 of \cite{Metivier} to $\T^d$.
Let $\chi \in C^\infty(\R^d,\R)$, with $0 \leq \chi \leq 1$, 
supported on $\{ |\xi| \leq 2 \}$ and such that $\chi = 1$ on $|\xi| \leq 1$.
Let $\Op(\chi_\lm)$ be the Fourier multiplier of symbol $\chi_\lm(\xi) := \chi(\xi / \lm)$, 
$\lm \geq 1$.
Let $\ph_\lm := \mF_{\R^d}^{-1} \chi_\lm$, where 
$\mF_{\R^d}^{-1}$ denotes the inverse Fourier transform on $\R^d$, 
so that the Fourier transform of $\ph_\lm$ is $\widehat{\ph_\lm} = \chi_\lm$. 
Thus for functions $u \in L^2(\R^d)$ we have
\[
\Op(\chi_\lm) u (x) 
= \int_{\R^d} \hat u(\xi) \chi_\lm(\xi) e^{\ii \xi \cdot x} \, d\xi
= \int_{\R^d} u(x-y) \ph_\lm(y) \, dy 
= (u \ast_{\R^d} \ph_\lm) (x),
\]
where $\hat u$ if the Fourier transform of $u$ and 
$\ast_{\R^d}$ denotes the convolution on $\R^d$. 
%
%The Fourier multiplier $\Op(\chi_\lm)$ also applies to distributions in the usual way.
%Periodic functions are tempered distributions, and, using Poisson summation formula, 
%one proves that 
%$\Op(\chi_\lm) u (x) 
%= \sum_{k \in \Z^d} \hat u_k \chi_\lm(k) e^{\ii k \cdot x}$
%on periodic functions $u \in L^2(\T^d)$, 
%where $\hat u_k$ are the Fourier coefficients of $u$. 
%
Similarly, for periodic functions $u \in L^2(\T^d)$ one has 
\[
\Op(\chi_\lm) u (x) 
= \sum_{k \in \Z^d} \hat u_k \chi_\lm(k) e^{\ii k \cdot x} 
= \int_{\T^d} u(x-y) \psi_\lm(y) \, dy 
= (u \ast_{\T^d} \psi_\lm) (x),
\]
where $\hat u_k$ are the Fourier coefficients of $u$, 
$\ast_{\T^d}$ denotes the convolution on $\T^d$, 
and $\psi_\lm(x) := \sum_{k \in \Z^d} \chi_\lm(k) e^{\ii k \cdot x}$.
With elementary calculations (imitating Section 13.4 of \cite{AB}),
one proves that $\psi_\lm$ is the periodization of $\ph_\lm$, namely 
\[
\psi_\lm(x) = \sum_{m \in \Z^d} \ph_\lm(x + 2\pi m), \quad \text{and} \quad 
\widehat{(\psi_\lm)}_k = \widehat{\ph_\lm}(k) \quad \forall k \in \Z^d,
\]
where $\widehat{(\psi_\lm)}_k$ are Fourier coefficients, 
and $\widehat{\ph_\lm}(k)$ is the Fourier transform. 
As a consequence, one proves that, for $u \in L^\infty(\T^d)$, 
$u \ast_{\R^d} \ph_\lm = u \ast_{\T^d} \psi_\lm$ 
(see equation (13.19) of \cite{AB}). 
We deduce that
\[
\| \Op(\chi_\lm) u \|_{L^\infty(\T^d)}
= \| u \ast_{\R^d} \ph_\lm \|_{L^\infty(\T^d)}
\leq \| u \|_{L^\infty(\T^d)} \| \ph_\lm \|_{L^1(\R^d)}
\]
and the bounds for $\ph_\lm$ over $\R^d$ proved in \cite{Metivier} 
can still be used. 
The periodization trick makes it possible to safely bypass a change of the variable $\xi$
which does not seem to be applicable when $\xi \in \Z^d$. 
\end{proof}

We recall also the standard commutator estimate between a multiplication operator and a Fourier multiplier. 
\begin{lemma}\label{commutatore moltiplicazione fourier multiplier}
Let $s \in \R_{> 0}$. 
Let $\vphi_s(D)$ be a Fourier multiplier of order $s$ and $a \in H^{s + 1}(\T) \cap H^2(\T)$. Then 
$$
\| [a , \vphi_s(D)] u\|_{0} \lesssim_s \| a \|_{s + 1} \| u \|_{0} + \| a \|_{2} \| u \|_{s - 1} \quad \forall u \in H^{s - 1}(\T) \cap L^2(\T). 
$$
\end{lemma}

We now state a lemma on changes of variables induced by diffeomorphisms of the torus. 

\begin{lemma}[Change of variables]\label{Lemma astratto cambio di variabile0}
$(i)$ Let $s \in \Z_{\geq 1}$ and $\alpha \in C^{s}(\T)$, with $\| \alpha\|_{C^1} \leq 1/2$. 
Then the operator ${\cal A} u(x) := u(x + \alpha(x))$ satisfies the estimate 
\begin{align} \label{orlando 0}
\| {\cal A} u\|_{0} 
& \lesssim \| u \|_{0} \quad \forall u \in L^2(\T),
\\ 
\label{orlando 1}
\| {\cal A} u\|_{s} 
& \lesssim_s \| u \|_{s} + \| \alpha\|_{C^s} \| u \|_{1} 
\quad \forall u \in H^s(\T), \ \ s \in \Z_{\geq 1}\,.
\end{align}
Moreover, for any $s \in \R_{\geq 0}$, 
if $\alpha \in H^{s + 2}(\T)$, with $\| \alpha\|_{2} \leq 1$, 
then
\begin{equation}\label{stima interpolata cambio di variabile0}
\| {\cal A} u \|_{s} \lesssim_s \| u \|_{s} + \| \alpha\|_{s +2} \| u \|_{0}
\quad \forall u \in H^s(\T), \ \ s \in \R_{\geq 0}\,. 
\end{equation}
\noindent
$(ii)$ Let $s \in \Z_{\geq 1}$ and $\alpha \in C^{s}(\T)$, with $\| \alpha\|_{C^1} \leq 1/2$. 
The map $\T \to \T$, $x \mapsto x + \alpha(x)$ is invertible 
and the inverse diffeomorphism $\T \to \T$, $y \mapsto y + \tilde \alpha(y)$ satisfies 
\begin{equation}\label{1307.1}
\| \widetilde \alpha \|_{{ C}^s} \lesssim_s \| \alpha\|_{{ C}^s}\,,
\quad s \in \Z_{\geq 1}\,.
\end{equation}

\noindent
$(iii)$ The inverse operator ${\cal A}^{- 1}$ defined as ${\cal A}^{- 1} u(y) := u(y + \widetilde \alpha(y))$ satisfies the same estimates \eqref{orlando 0}-\eqref{orlando 1} as ${\cal A}$ in $(i)$. Moreover there exists $\d \in (0,1)$ such that, 
for any $s \in \R_{\geq 0}$, if $\alpha \in H^{s + 4}(\T)$ with $\| \alpha\|_4 \leq \delta$, 
then 
\begin{equation}\label{stima interpolata cambio di variabile0 inverso}
\| {\cal A}^{- 1} u\|_s \lesssim_s \| u \|_s + \| \alpha\|_{s + 4} \| u \|_0 
\quad \forall u \in H^{s}(\T), \ \ s \in \R_{\geq 0}\,.
\end{equation}
\end{lemma}

\begin{proof}
{\sc Proof of $(i)$.}
Estimates \eqref{orlando 0}-\eqref{orlando 1} are classical; they are proved, e.g., 
in \cite{Baldi-Benj-Ono}, Lemma B.4. 
Let us prove \eqref{stima interpolata cambio di variabile0}. 
Applying \eqref{orlando 1} for $s = 1$ and recalling \eqref{orlando 0} one has 
\begin{equation}\label{stima fino H1 cal A}
\| {\cal A} u \|_0 \lesssim \| u \|_0\,, \quad 
\| {\cal A} u \|_1 \lesssim \| u \|_1\,.
\end{equation}
Now let $u \in H^2(\T)$ and assume that $\alpha \in H^2(\T)$, with $\| \alpha\|_2 \leq 1$. 
Then, using \eqref{stima fino H1 cal A}, \eqref{prtame2} and the bound $\| \alpha\|_2 \leq 1$, 
\begin{align}
\| {\cal A} u \|_2 
\simeq \| {\cal A} u\|_0 + \|\partial_x ({\cal A} u) \|_1 
& \stackrel{\eqref{stima fino H1 cal A}}{\lesssim} 
\| u \|_0 + \| (1 + \alpha_x) {\cal A}(u_x) \|_1 
\nonumber\\ & 
\stackrel{\eqref{prtame2}}{\lesssim} 
\| u \|_0 + \| {\cal A}(u_x) \|_1 (1 + \| \alpha\|_2) 
\stackrel{\eqref{stima fino H1 cal A}}{\lesssim} 
\| u \|_2 \,.
\label{orlando 3}
\end{align}
By \eqref{orlando 0} and \eqref{orlando 3}, 
using a classical interpolation result, one has 
\begin{equation}\label{stima fino H2 cal A s non interi}
\| {\cal A} u\|_s \lesssim \| u \|_s 
\quad \forall u \in H^s(\T), \ \ s \in [0, 2]. 
\end{equation}
Now we argue by induction on $s$. 
Assume that the claimed estimate holds for $s \in \R_{\geq 1}$ 
and let us prove it for $s + 1$. 
Using the bound $\| \a \|_2 \leq 1$, we have 
\begin{align*}
\| {\cal A} u \|_{s + 1} 
\simeq \| {\cal A} u\|_0 + \| \partial_x ({\cal A} u) \|_s 
& \stackrel{\eqref{orlando 0}}{\lesssim} 
\| u \|_0 + \| (1 + \alpha_x) {\cal A}(\partial_x u) \|_s 
\\ & 
\stackrel{\eqref{prtame}}{\lesssim_s} 
\| u \|_0 + \| {\cal A}(u_x) \|_s + \| \alpha_x\|_s \| {\cal A}(u_x) \|_1\,. 
\end{align*}
By the inductive hyphothesis, we deduce that 
\begin{align}
\| {\cal A} u \|_{s + 1} & \lesssim_s \| u \|_{s + 1} + \| \alpha \|_{s + 2} \| u \|_1 + \| \alpha\|_{s + 1} \| u\|_2\,. \label{pappone 0}
\end{align}
By \eqref{interpolation estremi fine}, applied with $u_1 = \alpha, u_2 = u$, $a_0 = 2, b_0 = 0, p= s, q = 1$, one gets 
\begin{align}
 \| \alpha \|_{s + 2} \| u \|_1 \leq \| \alpha\|_{s + 3} \| u \|_0 + \| \alpha\|_{2} \| u \|_{s + 1}\,.\label{pappone 1}
\end{align}
Using again \eqref{interpolation estremi fine}, 
applied with $u_1 = \alpha, u_2 = u$, $a_0 = 2, b_0 = 0, p= s - 1, q = 2$, one gets 
\begin{align}
 \| \alpha \|_{s + 1} \| u \|_2 \leq \| \alpha\|_{s + 3} \| u \|_0 + \| \alpha\|_{2} \| u \|_{s + 1}\,.\label{pappone 2}
\end{align}
Then \eqref{pappone 0}-\eqref{pappone 2}, using that $\| \alpha\|_2 \leq 1$, imply that 
$$
\|{\cal A} u \|_{s + 1} \lesssim_s \| u \|_{s + 1} + \| \alpha\|_{s + 3} \| u \|_0\,,
$$
which is estimate \eqref{stima interpolata cambio di variabile0} at the Sobolev index $s + 1$. 

{\sc Proof of $(ii)$.} It is proved in \cite{Baldi-Benj-Ono}, Lemma B.4. 

{\sc Proof of $(iii)$.} The fact that ${\cal A}^{- 1}$ satisfies the estimate \eqref{orlando 0}-\eqref{orlando 1} is proved in \cite{Baldi-Benj-Ono}, Lemma B.4. 
Let us prove \eqref{stima interpolata cambio di variabile0 inverso}. 
For any real $s \geq 0$, we denote by $[s]$ the integer part of $s$. One has 
\begin{equation}\label{orlando 10}
\| \widetilde \alpha\|_{s + 2} 
\leq \| \widetilde \alpha\|_{[s] + 3} 
\lesssim \| \widetilde \alpha\|_{C^{[s] + 3}}  
\stackrel{(ii)}{\lesssim_s} \| \alpha\|_{C^{[s] + 3}} 
\stackrel{\eqref{Sobolev embedding}}{\lesssim_s} \| \alpha \|_{{[s] + 4}} 
\lesssim_s \| \alpha \|_{{s + 4}}\,.
\end{equation}
Hence, for $s = 0$, 
one has $\| \widetilde \alpha\|_2 \leq C \| \alpha\|_4 \leq 1$ 
by taking $\| \alpha\|_4$ small enough. 
Therefore we can apply \eqref{stima interpolata cambio di variabile0} to ${\cal A}^{- 1}$ 
and the claimed estimate follows by \eqref{orlando 10}. 
\end{proof}

We also study the action of the operators induced by diffeomorphisms of the torus on the spaces $C([0, T], H^s(\T))$. 
For any function $\a : [0, T] \times \T \to \R$ and any $h : \T \to \C$, we define the $t$-dependent family ${\cal A}(t) h(x) := h(x + \a(t, x))$. 
Then, given $h : [0, T] \times \T \to \R$, we define 
\begin{equation}\label{definizione cambi di variabile funzioni spazio-tempo}
{\cal A} h(t, x) := {\cal A}(t)h(t, x) = h(t, x + \a(t, x))\,. 
\end{equation}

\begin{lemma}\label{lemma derivata alpha tilde alpha astratto}
Let $s \in  \Z_{\geq 1}$, $\alpha \in C([0, T], C^s(\T))$ with $\| \alpha_x\|_{L^\infty} \leq 1/2$. Let $y \mapsto y + \widetilde \alpha (t, y)$ be the inverse diffeomorphism of $x \mapsto x + \alpha(t, x)$. Then $\widetilde \alpha \in C([0, T], C^s(\T))$ and 
\begin{equation}\label{shutter island}
\| \widetilde \alpha\|_{C([0, T], C^s(\T))} \lesssim_s \|  \alpha\|_{C([0, T], C^s(\T))}\,, 
\quad s \in  \Z_{\geq 1}\,. 
\end{equation}
Moreover, for any $s \in \R_{\geq 0}$, 
if $\alpha \in C([0, T], H^{s + 2}(\T))$, then $\widetilde \alpha \in C([0, T], H^s(\T))$, 
with 
\begin{equation}\label{shutter island 0}
\| \widetilde \alpha\|_{T, s} \lesssim_s \| \alpha\|_{T, s + 2}\,, 
\quad s \in \R_{\geq 0}\,.
\end{equation}
\end{lemma}

\begin{proof}
{\sc Proof of \eqref{shutter island}.}
Let $y \mapsto y + \widetilde \alpha(t, y)$ be the inverse diffeomorphism of $x \mapsto x + \alpha (t, x)$. Since 
$$
\widetilde \alpha(t, y) + \alpha (t, y + \widetilde \a(t, y)) = 0,
$$
one can directly check that if $\alpha \in C([0, T], C^1(\T))$ then also $\widetilde \alpha \in C([0, T], C^1(\T))$ and 
$$
 \widetilde \alpha_y(t, y) = - \frac{\alpha_x(t, y + \widetilde \alpha(t, y))}{ 1 + \widetilde \alpha_y(t, y)}\,.
$$
Using the above formula and a bootstrap argument, one can show that for any integer $s \geq 1$, if $\alpha \in C([0, T], C^s(\T))$, then $\widetilde \alpha \in C([0, T], C^s(\T))$. 
By \eqref{1307.1}, one has $\| \widetilde \alpha(t, \cdot) \|_{C^s} \lesssim_s \|  \alpha(t, \cdot) \|_{C^s}$. Then \eqref{shutter island} follows by taking the sup over $t \in [0, T]$. 

{\sc Proof of \eqref{shutter island 0}.} Let $\alpha \in C([0, T], H^{s + 2}(\T))$. 
Since $[s] \leq s$, one has $C([0, T], H^{s + 2}(\T))$ 
$\subseteq C([0, T], H^{[s] + 2}(\T))$. 
Using \eqref{Sobolev embedding}, $C([0, T], H^{[s] + 2}(\T)) \subseteq C([0, T], C^{[s] + 1}(\T))$.
As a consequence, $\alpha$ $\in C([0, T], C^{[s] + 1}(\T))$, with 
\begin{equation}\label{orlando 20}
\| \alpha\|_{C([0, T], C^{[s] + 1}(\T))} \lesssim_s \| \alpha\|_{T, s + 2}\,. 
\end{equation}
By \eqref{shutter island}, $\widetilde \alpha \in C([0, T], C^{[s] + 1}(\T))$ and using that $C([0, T], C^{[s] + 1}(\T)) \subseteq C([0, T], H^s(\T))$, we get that $\widetilde \alpha \in C([0, T], H^s(\T))$, with $\| \widetilde \alpha\|_{T, s} \lesssim_s \| \widetilde \alpha\|_{C([0, T], C^{[s] + 1}(\T))}$. The claimed inequality \eqref{shutter island 0} follows by recalling \eqref{orlando 20}. 
\end{proof}

\begin{lemma}[Change of variables]\label{Lemma astratto cambio di variabile}
There exists $\d \in (0,1)$ with the following properties.

$(i)$ Let $s \in \R_{\geq 0}$ and $\alpha \in C([0, T], H^{s + 2}(\T))$, 
with $\| \alpha\|_{T, 2} \leq \delta$. 
Then the operator ${\cal A} u(t, x) := u(t, x + \alpha(t, x))$ is a linear and continuous operator $C([0, T], H^s(\T)) \to C([0, T], H^s(\T))$, with 
\begin{equation}\label{stima interpolata cambio di variabile}
\| {\cal A} u\|_{T, s} \lesssim_s \| u \|_{T, s} + \| \alpha\|_{T, s + 2} \| u \|_{T, 0}
\quad \forall u \in C([0, T], H^s(\T))\,.
\end{equation}

$(ii)$
Let $s \in \R_{\geq 0}$ and $\alpha \in C([0, T], H^{s + 4}(\T))$, with $\| \alpha\|_{T, 4} \leq \delta$. 
Then the inverse operator ${\cal A}^{- 1}$, 
defined by ${\cal A}^{- 1} u(t, y) := u(t, y + \widetilde \alpha(t, y))$, 
maps $C([0, T], H^s(\T))$ into itself, with 
$$
\| {\cal A}^{- 1} u \|_{T, s} \lesssim_s \| u \|_{T, s} + \| \alpha\|_{T, s + 4} \| u \|_{T, 0}
\quad \forall u \in C([0, T], H^s(\T))\,. 
$$
\end{lemma}

\begin{proof}
First, we prove $(i)$.
Let $s \in \R_{\geq 0}$ and $u \in C([0, T], H^s(\T))$. 
We have to prove that ${\cal A} u \in C([0, T], H^s(\T))$, namely, for any $t_0 \in [0,T]$, 
we have to prove that $\| (\mA u)(t) - (\mA u)(t_0) \|_s \to 0$ as $t \to t_0$. 
By triangular inequality, 
\begin{equation} \label{1307.2}
\| (\mA u)(t) - (\mA u)(t_0) \|_s 
\leq \| \mA(t) [u(t) - u(t_0)] \|_s + \| (\mA(t) - \mA(t_0)) [u(t_0)] \|_s
\end{equation}
(where, in short, $u(t)$ means $u(t,\cdot)$).
The first term is estimated using \eqref{stima interpolata cambio di variabile0}, which gives
\[
\| \mA(t) [u(t) - u(t_0)] \|_s 
\lesssim_s \, \| u(t) - u(t_0) \|_s + \| \a \|_{T,s+2} \| u(t) - u(t_0) \|_{0}  
\to 0 \quad (t \to t_0).
\]
To prove that the last term in \eqref{1307.2} also vanishes as $t \to t_0$ is equivalent 
to prove that, for every $h \in H^s(\T)$, 
the map $[0,T] \to H^s(\T)$, $t \mapsto \mA(t) h$ is continuous.
Let $h \in H^s(\T)$, and let $\widehat h (k)$ be its Fourier coefficients. 
Let 
\[
\Pi_n h(x) := \sum_{|k| \leq n} \widehat h(k) e^{\ii k x}, \quad 
\Pi_n^\bot h(x) := (I - \Pi_n) h(x) 
= \sum_{|k| > n} \widehat h (k) e^{\ii k x},
\]
and 
\[
f_n(t) := \mA(t) \Pi_n h, \quad 
f(t) := \mA(t) h.
\]
The sequence $(f_n)$ converges to $f$ uniformly in $t \in [0,T]$ 
in the space $H^s(\T)$, because, 
using \eqref{stima interpolata cambio di variabile0} and the assumption $h \in H^s(\T)$, 
\[
\sup_{t \in [0,T]} \| f_n(t) - f(t) \|_s 
= \| f_n - f \|_{T,s} 
= \| \mA \Pi_n^\bot h \|_{T,s}
\lesssim_s \| \Pi_n^\bot h \|_s + \| \a \|_{T,s+2} \| \Pi_n^\bot h \|_0
\to 0 \quad (n \to \infty).
\]
Since continuity is preserved by uniform limits, 
we have to prove that all $f_n$ are continuous. 
For any $n$, the function $f_n$ is 
\[
f_n(t,x) 
= \mA(t) \Pi_n h(x)
= \sum_{|k| \leq n} \widehat h(k) \, \psi_k(t,x), \quad 
\psi_k(t,x)  := e^{\ii k (x + \a(t,x))} = \mA(t)[e^{\ii k x}].
\]
Hence $f_n$ is a finite linear combination of functions $\psi_k$. 
It remains to prove that, for all $k \in \Z$, 
the function $\psi_k$ belongs to $C([0,T], H^s(\T))$. 
Fix $k \in \Z$, and consider the functions $G(u) := e^{\ii k u}$ 
and $F(u) := e^{\ii k u} - 1$. 
Split 
\[
\psi_k(t) - \psi_k(t_0) 
= e^{\ii k x} e^{\ii k \a(t_0,x)} \{ e^{\ii k [\a(t,x) - \a(t_0,x)] } - 1 \},
\]
and estimate each factor.
First, $\| e^{\ii k x} \|_s = \langle k \rangle^s$. 
Second, using \eqref{0707.2} and the assumption $\| \a \|_{T,1} \leq 1$, 
\[
\| e^{\ii k \a(t_0,x)} \|_s 
= \| G(\a(t_0, \cdot)) \|_s 
\leq C_{k,s} (1 + \| \a(t_0, \cdot) \|_s)
\leq C_{k,s} (1 + \| \a \|_{T,s}).
\]
Third, by \eqref{0707.1}, 
\[
\| e^{\ii k [\a(t,\cdot) - \a(t_0, \cdot)]} - 1 \|_s 
= \| F(\a(t, \cdot) - \a(t_0, \cdot)) \|_s 
\leq C_{k,s} \| \a(t, \cdot) - \a(t_0, \cdot) \|_s\,. 
\]
Hence
\[
\| \psi_k(t, \cdot) - \psi_k(t_0, \cdot) \|_s 
\leq C_{k,s} (1 + \| \a \|_{T,s}) \| \a(t, \cdot) - \a(t_0, \cdot) \|_s \to 0 
\quad (t \to t_0) 
\]
because $\a \in C([0,T], H^s(\T))$. 
Hence, we have proved that ${\cal A} : C([0, T], H^s(\T)) \to C([0, T], H^s(\T))$. 
Estimate \eqref{stima interpolata cambio di variabile} then follows by applying 
\eqref{stima interpolata cambio di variabile0} at any fixed $t \in [0, T]$ 
and taking the supremum. 

Finally, $(ii)$ follows by $(i)$ and \eqref{shutter island 0}. 
\end{proof}

\section{Appendix C. Well-posedness of linear equations} 
\label{sec:WP}

\begin{lemma}\label{buona positura equazione lineare 0}
Let $T > 0$, $t_0 \in [0, T]$, $\mu \in \R$. Let $S \geq 1$, ${\bf h}_{in} \in {\bf H}^S(\T)$, ${\bf g} \in C([0, T], {\bf H}^S(\T))$ and let ${\cal R}$ be the multiplication operator 
\begin{equation}\label{forma operatori moltiplicazione buona positura}
{\cal R} := \begin{pmatrix}
r_1 & r_2 \\
\overline r_2 & \overline r_1
\end{pmatrix}\,, \qquad r_1, r_2 \in C([0, T], H^{S + 1}(\T))\,.
\end{equation}
 There exists $\eta > 0$ small enough depending on $T$ such that if 
\begin{equation}\label{piccolezza cal R buona positura}
\| {\cal R}\|_{T, 1} = {\rm max}\{ \| r_1\|_{T, 1}, \| r_2\|_{T, 1} \} \leq \eta\,, 
\end{equation}
then there exists a unique solution ${\bf h} \in C([0, T], {\bf H}^S(\T))$ of the Cauchy problem 
\begin{equation}\label{problema di cauchy cal L4}
\begin{cases}
\partial_t {\bf h} + \ii \mu \Sigma \partial_{xx} {\bf h} + {\cal R} {\bf h} = {\bf g} \\
{\bf h}(t_0, \cdot) = {\bf h}_{in}
\end{cases}
\end{equation}
satisfying for any $0 \leq s \leq S$, the estimate 
$$
\| {\bf h}\|_{T, s} \lesssim_s \| {\bf h}_{in}\|_{s} + \| {\bf g} \|_{T, s} + \| {\cal R}\|_{T, s + 1} \| {\bf h}_{in}\|_{0}\,.
$$
\end{lemma}
\begin{proof}
Since ${\bf h}_0 = (h_0, \overline h_0)$, ${\bf g} = (g , \overline g)$, ${\bf h} = (h, \overline h)$ and ${\cal R}$ has the form \eqref{forma operatori moltiplicazione buona positura},
it is enough to study the Cauchy problem 
\begin{equation}\label{problema di Cauchy nel lemma}
\begin{cases}
\partial_t h + \ii \mu \partial_{xx} h + {\cal Q}(h) = g \\
h (t_0, \cdot) = h_0\,,
\end{cases}\qquad {\cal Q}(h) :=r_1 h + r_2 \overline h\,.
\end{equation}
Note that for any $0 \leq s \leq S$, by Lemma \ref{lemma interpolazione}-$(ii)$, applying \eqref{interpolation estremi fine}, with $v = (r_1, r_2)$, $u = h$, $a_0 = 1$, $b_0 = 0$, $p = s - 1$, $q = 1$ and using the smallness condition \eqref{piccolezza cal R buona positura}, one gets that 
\begin{equation}\label{stima cal Q buona positura lemma}
\| {\cal Q} h\|_{T, s} \lesssim_s \eta \| h \|_{T, s} + \| {\cal R}\|_{T, s+ 1} \| h\|_{T, 0}\,, \quad \forall h \in C([0, T], H^s(\T))\,.  
\end{equation} 
We split in \eqref{problema di Cauchy nel lemma}, $h =  v + \vphi$, where 
\begin{equation}\label{problema di cauchy nel lemma 2}
\begin{cases}
\partial_t v + \ii \mu \partial_{xx} v = g \\
v(t_0, \cdot) = h_{in}\,,
\end{cases}\qquad \begin{cases}
\partial_t \vphi + \ii \mu \partial_{xx} \vphi + {\cal Q}(\vphi) + {\cal Q}( v) = 0 \\
\vphi(t_0, \cdot) = 0\,.
\end{cases}
\end{equation}
The first Cauchy problem in \eqref{stima cal Q buona positura lemma} can be solved explicitly and since $h_{in} \in H^S(\T)$, $g\in C([0, T],$ $H^S(\T))$ there exists a unique solution $v \in C([0, T], H^S(\T))$ satisfying 
\begin{equation}\label{stima flusso libero lemma appendice}
\| v \|_{T, s} \leq \| h_{in}\|_{T, s} + T \| g \|_{T, s}\,, \qquad \forall 0 \leq s \leq S\,. 
\end{equation}
Then, we construct iteratively the solution of the second Cauchy problem in \eqref{problema di cauchy nel lemma 2}, by setting 
$$
\vphi_0 := 0\,, \qquad \vphi_{n + 1} := \Phi(\vphi_n)\,, \qquad n \geq 0\,,
$$
where 
\begin{align}
\Phi(\vphi)&  := - \int_{t_0}^t e^{- \ii \mu \partial_{xx}(t - \tau)}[{\cal Q}(v)(\tau) ]\, d \tau  - \int_{t_0}^t e^{- \ii \mu \partial_{xx}(t - \tau)}[ {\cal Q}(\vphi)(\tau) ]\, d \tau\,. \label{definizione mappa punto fisso appendice}
\end{align}
We prove the following claim: 
for any $0 \leq  s \leq S$ there exists a constant $K_T(s) > 0$ (depending on $T$ and $s$) such that for any $n \geq 0$, $\vphi_n \in C([0, T], H^s(\T))$ and
\begin{equation}\label{stima induttiva hn lemma}
\| \vphi_n\|_{T, s} \leq R(s)\,, \qquad R(s) := K_T(s) \Big(  \eta \| v \|_{T, s} + \| {\cal R}\|_{T, s + 1} \| v \|_{T, 0} \Big)\,.
\end{equation}
We argue by induction on $n$. For $n = 0$ the statement is trivial. Then assume that the claim holds for some $n \geq 0$ and let us prove it for $n + 1$. By the definition of the map $\Phi$ in \eqref{definizione mappa punto fisso appendice}, using the inductive hyphothesis, one has immediately that $\vphi_{n + 1} = \Phi(\vphi_n) \in C([0, T], H^s(\T))$, for any $0 \leq s \leq S$. 
Moreover, using that for any $t, \tau \in [0, T]$, $\| e^{- \ii \mu \partial_{xx}(t - \tau)} \|_{{\cal L}(H^s(\T))} \leq 1$ and by estimate \eqref{stima cal Q buona positura lemma}, one gets 
\begin{align}
\| \vphi_{n + 1}\|_{T, s} &  \leq C(s) T \big( \eta \| v \|_{T, s} + \| {\cal R}\|_{T, s + 1} \| v \|_{T, 0} \big) 
+ C(s) T \big( \eta \| \vphi_n \|_{T, s} + \| {\cal R}\|_{T, s + 1} \| \vphi_n \|_{T, 0} \big)
 \nonumber\\ &
\stackrel{\eqref{stima induttiva hn lemma}}{\leq} C(s) T \big( \eta \| v \|_{T, s} + \| {\cal R}\|_{T, s + 1} \| v \|_{T, 0} \big)
+ C(s) T K_T(s) \eta\Big(  \eta \| v \|_{T, s} + \| {\cal R}\|_{T, s + 1} \| v \|_{T, 0} \Big) \nonumber\\
& \qquad + C(s) T \| {\cal R}\|_{T, s + 1} K_T(0) \big( \eta \| v \|_{T, 0} + \| {\cal R}\|_{T,  1} \| v \|_{T, 0} \big) \nonumber\\
& \stackrel{\eqref{piccolezza cal R buona positura}}{\leq}\Big( C(s) T \eta + C(s) K_T(s) T \eta^2 \Big) \| v \|_{T, s} \nonumber\\
& \qquad + \Big( C(s) T  + C(s) K_T(s) T \eta + 2 T C(s) K_T(0) \eta  \Big) \| {\cal R}\|_{T, s + 1} \| v \|_{T, 0}   \nonumber\\
& \leq K_T(s) \Big(\eta \| v \|_{T, s} + \| {\cal R}\|_{T, s + 1} \| v \|_{T, 0} \Big)\,,
\end{align}
provided that
$$
C(s) T + C(s) K_T(s) T \eta \leq K_T(s)\,, \qquad C(s) T  + C(s) K_T(s) T \eta + 2 T C(s) K_T(0) \eta \leq K_T(s)\,. 
$$
The above conditions are fulfilled by taking $K_T(s) > 0$ large enough and $\eta \in (0, 1)$ small enough, therefore \eqref{stima induttiva hn lemma} has been proved at the step $n + 1$. 

\bigskip

\noindent
{\sc Convergence of $\vphi_n$.} We prove that for any $0 \leq s \leq S$, there exists a constant $J_T(s) > 0$ such that for any $n \geq 0$
\begin{equation}\label{stima iterativa convergenza vphi n}
\| \vphi_{n + 1} - \vphi_n \|_{T, s} \leq J_T(s) \Big(\frac{1}{2^{n + 1}} \| v \|_{T, s} + \frac{1}{2^n} \| {\cal R}\|_{T, s + 1}\| v \|_{T, 0} \Big)\,. 
\end{equation}
We argue by induction on $n$. For $n = 0$, since $\vphi_0 = 0$, the estimate follows by \eqref{stima induttiva hn lemma} applied for $n = 1$ and by taking $J_T(s) \geq K_T(s)$ and $\eta \leq 1/2$. Now let us assume that \eqref{stima iterativa convergenza vphi n} holds for some $n \geq 0$ and let us prove it for $n + 1$. Recalling \eqref{definizione mappa punto fisso appendice} and the definition of $\cal Q$ in \eqref{problema di Cauchy nel lemma}, one has  
$$
\vphi_{n + 2} - \vphi_{n + 1} = \Phi(\vphi_{n + 1}) - \Phi(\vphi_n) = - \int_{t_0}^t e^{- \ii \mu \partial_{xx}(t - \tau)}[ {\cal Q}(\vphi_{n + 1} - \vphi_n)(\tau) ]\, d \tau\,.
$$
Using estimates \eqref{stima cal Q buona positura lemma}, \eqref{piccolezza cal R buona positura}, \eqref{stima iterativa convergenza vphi n}, one gets 
\begin{align}
\| \vphi_{n + 2} - \vphi_{n + 1}\|_{T, s} & \leq C(s) T \Big( \eta \| \vphi_{n + 1} - \vphi_n \|_{T, s} + \|{\cal R} \|_{T, s + 1} \| \vphi_{n + 1} - \vphi_n\|_{T, 0} \Big) \nonumber\\
& \leq C(s) T \eta J_T(s) \Big( \frac{1}{2^{n + 1}} \| v \|_{T, s} + \frac{1}{2^n} \| {\cal R}\|_{T, s + 1} \| v \|_{T, 0} \Big)
\nonumber\\ & \qquad
+ C(s) T \| {\cal R}\|_{T, s + 1} J_T(0)  \Big( \frac{1}{2^{n + 1}} + \eta \frac{1}{2^n} \Big) \| v \|_{T, 0} \nonumber\\
& \leq J_T(s) \Big( \frac{1}{2^{n + 2}} \| v \|_{T, s} + \frac{1}{2^{n + 1}} \| {\cal R}\|_{T, s + 1} \| v \|_{T, 0} \Big) \nonumber
\end{align}
by taking $J_T(s) > 0$ large enough and $\eta \in (0, 1)$ small enough. Thus \eqref{stima iterativa convergenza vphi n} at the step $n + 1$ has been proved. Using a telescoping argument one has that there exists $\vphi \in C([0, T], H^{S}(\T))$ such that 
$$
\vphi_n \to \vphi\,, \qquad \text{in} \qquad C([0, T], H^s(\T))\,, \quad \forall 0 \leq s \leq S\,.  
$$
Moreover, $\Phi(\vphi_n) \to \Phi(\vphi)$ in $C([0, T], H^s(\T))$, for any $0 \leq s \leq S$, implying that $\Phi(\vphi) = \vphi$. Since $\| \vphi\|_{T, s}= \lim_{n \to + \infty} \| \vphi_n\|_{T, s}$, by \eqref{stima induttiva hn lemma} one deduces that $\vphi$ satisfies 
\begin{equation}\label{stima vphi punto fisso}
\| \vphi \|_{	T, s} \lesssim_s \eta \| v \|_{T, s} + \| {\cal R}\|_{T, s + 1} \| v \|_{T, 0}\,. 
\end{equation}
Recalling that $h = \vphi + v$ and using estimates \eqref{stima flusso libero lemma appendice}, \eqref{stima vphi punto fisso}, one gets
$$
\| h \|_{T, s} \lesssim_s \| h_{in}\|_{s} + \| g \|_{T, s} + \| {\cal R}\|_{T, s + 1} \| h_{in}\|_{0}\,,
$$
and the lemma is proved. 
\end{proof}
\begin{lemma}[Well posedness of the operator ${\cal L}_4$ in \eqref{cal L5}]\label{buona positura equazione lineare 1}
Let $T > 0$, $t_0 \in [0, T]$ and let ${\cal L}_4 = \partial_t {\mathbb I}_2 + \ii \mu \partial_{xx} \Sigma + {\cal R}$ be the operator defined in \eqref{cal L5}. There exists $\eta \in (0, 1)$ small enough and universal constants $\sigma ,\t > 0$ large enough such that if $N_T(\sigma) \leq \eta$ (see the definition \eqref{definizione NT sigma}), then for any $s \in [0,r -\t]$, ${\bf h}_{in} \in {\bf H}^s(\T)$, ${\bf g} \in C([0, T], {\bf H}^s(\T))$, there exists a unique solution ${\bf h} \in C([0, T], {\bf H}^s(\T))$ such that 
$$
\begin{cases}
{\cal L}_4 {\bf h} = {\bf g} \\
{\bf h}(t_0, \cdot) = {\bf h}_{in}
\end{cases}
$$
satisfying the estimate 
\begin{equation}\label{stima buona positura equazione lineare 1}
\| {\bf h}\|_{T, s} \lesssim_s \| {\bf h}_{in}\|_{H^s_x} + \| {\bf g}\|_{T, s} + N_T(s + \sigma) \| {\bf h}_{in}\|_{L^2_x}\,. 
\end{equation}
\end{lemma}
\begin{proof}
The lemma follows by applying Lemmas \ref{stime coefficienti linearizzato}, \ref{stime cal L5} and \ref{buona positura equazione lineare 0}. Indeed, by \eqref{cal R r1 r2}-\eqref{stima cal R linearizzato ridotto}, using that $N_T(\sigma) \leq \eta$ for some $\eta \in (0, 1)$ small enough and $\sigma \in \N$ large enough, the smallness condition \eqref{piccolezza cal R buona positura} is fulfilled.  
\end{proof}
\begin{lemma}[Well posedness of the operator ${\cal L}_3$ in \eqref{forma finale cal L4}]\label{buona positura equazione lineare 2}
Let $T > 0$, $t_0 \in [0, T]$ and let ${\cal L}_3 = \partial_t {\mathbb I}_2 + \ii \mu \Sigma \partial_{xx} + \ii A_1^{(3)} \partial_x + \ii A_0^{(3)}$ be the operator defined in \eqref{forma finale cal L4}. There exists $\eta \in (0, 1)$ small enough and universal constants $\sigma ,\t > 0$ large enough such that if $N_T(\sigma) \leq \eta$ (see the definition \eqref{definizione NT sigma}), then for any $s \in [0, r-\t]$, ${\bf h}_{in} \in {\bf H}^s(\T)$, ${\bf g} \in C([0, T], {\bf H}^s(\T))$, there exists a unique solution ${\bf h} \in C([0, T], {\bf H}^s(\T))$ such that 
\begin{equation}\label{problema di Cauchy cal L3 appendice}
\begin{cases}
{\cal L}_3 {\bf h} = {\bf g} \\
{\bf h}(t_0, \cdot) = {\bf h}_{in}
\end{cases}
\end{equation}
satisfying the estimate 
$$
\| {\bf h}\|_{T, s} \lesssim_s \| {\bf h}_{in}\|_{H^s_x} + \| {\bf g}\|_{T, s} + N_T(s + \sigma) \| {\bf h}_{in}\|_{L^2_x}\,. 
$$
\end{lemma}

\begin{proof}
Let ${\cal M}$ be the transformation defined in \eqref{definizione cal m}. By \eqref{cal L5}, defining $\widetilde{\bf h}(t, \cdot) := {\cal M}^{- 1}(t) {\bf h}(t, \cdot)$, $\widetilde{\bf g} :={\cal M}^{- 1}(t){\bf g}(t, \cdot)$, the Cauchy problem \eqref{problema di Cauchy cal L3 appendice} transforms into the Cauchy problem 
$$
\begin{cases}
{\cal L}_4 \widetilde{\bf h} = \widetilde{\bf g} \\
\widetilde{\bf h}(t_0, \cdot) = \widetilde{\bf h}_{in}
\end{cases}\,.
$$
Then the statement follows by Lemma \ref{buona positura equazione lineare 1} and by estimate \eqref{stima cal M} on the transformation ${\cal M}$.
\end{proof}

\begin{lemma}[Well posedness of the operator ${\cal L}_2$ in \eqref{cal L3}]\label{buona positura equazione lineare 3}
Let $T > 0$, $t_0 \in [0, T]$ and let ${\cal L}_2 = \partial_t {\mathbb I}_2 + \ii \mu \Sigma \partial_{xx} + \ii A_1^{(2)} \partial_x + \ii A_0^{(2)}$ be the operator defined in \eqref{cal L3}. There exists $\eta \in (0, 1)$ small enough and universal constants $\sigma ,\t > 0$ large enough such that if $N_T(\sigma) \leq \eta$ (see the definition \eqref{definizione NT sigma}), then for any $s \in [0, r- \t]$, ${\bf h}_{in} \in {\bf H}^s(\T)$, ${\bf g} \in C([0, T], {\bf H}^s(\T))$, there exists a unique solution ${\bf h} \in C([0, T], {\bf H}^s(\T))$ such that 
\begin{equation}\label{problema di Cauchy cal L2 appendice}
\begin{cases}
{\cal L}_2 {\bf h} = {\bf g} \\
{\bf h}(t_0, \cdot) = {\bf h}_{in}
\end{cases}
\end{equation}
satisfying the estimate 
$$
\| {\bf h}\|_{T, s} \lesssim_s \| {\bf h}_{in}\|_{s} + \| {\bf g}\|_{T, s} + N_T(s + \sigma) \| {\bf h}_{in}\|_{0}\,. 
$$
\end{lemma}

\begin{proof}
Let ${\cal T}$ be the transformation defined in \eqref{operatori traslazione dello spazio}. By \eqref{cal L4}, defining $\widetilde{\bf h}(t, \cdot) := {\cal T}^{- 1}(t) {\bf h}(t, \cdot)$, $\widetilde{\bf g} :={\cal T}^{- 1}(t){\bf g}(t, \cdot)$, the Cauchy problem \eqref{problema di Cauchy cal L2 appendice} transforms into the Cauchy problem 
$$
\begin{cases}
{\cal L}_3 \widetilde{\bf h} = \widetilde{\bf g} \\
\widetilde{\bf h}(t_0, \cdot) = \widetilde{\bf h}_{in}
\end{cases}\,.
$$
Then the statement follows by Lemma \ref{buona positura equazione lineare 2} and by estimate \eqref{stima cal T pm1} on the transformation ${\cal T}$.
\end{proof}

\begin{lemma}[Well posedness of the operator ${\cal L}_1$ in \eqref{forma finale cal L2}]\label{buona positura equazione lineare 4}
Let $T > 0$, $t_0 \in [0, T]$ and let ${\cal L}_1 = \partial_t {\mathbb I}_2 + \ii m_2 \Sigma \partial_{yy} + \ii A_1^{(1)} \partial_y + \ii A_0^{(1)} $ be the operator defined in \eqref{forma finale cal L2}. There exists $\eta \in (0, 1)$ small enough and universal constants $\sigma ,\t > 0$ large enough such that if $N_T(\sigma) \leq \eta$ (see the definition \eqref{definizione NT sigma}), then for any $s \in [0, r-\t]$, ${\bf h}_{in} \in {\bf H}^s(\T)$, ${\bf g} \in C([0, T], {\bf H}^s(\T))$, there exists a unique solution ${\bf h} \in C([0, T], {\bf H}^s(\T))$ such that 
\begin{equation}\label{problema di Cauchy cal L1 appendice}
\begin{cases}
{\cal L}_1 {\bf h} = {\bf g} \\
{\bf h}(t_0, \cdot) = {\bf h}_{in}
\end{cases}
\end{equation}
satisfying the estimate 
$$
\| {\bf h}\|_{T, s} \lesssim_s \| {\bf h}_{in}\|_{s} + \| {\bf g}\|_{T, s} + N_T(s + \sigma) \| {\bf h}_{in}\|_{0}\,. 
$$
\end{lemma}

\begin{proof}
Let ${\cal B}$ be the transformation defined in \eqref{definizione cal B}. By \eqref{cal L3}, defining $\widetilde{\bf h}(t, \cdot) := {\cal B}^{- 1}(t) {\bf h}(t, \cdot)$, $\widetilde{\bf g} := \rho^{- 1}{\cal B}^{- 1}(t){\bf g}(t, \cdot)$, the Cauchy problem \eqref{problema di Cauchy cal L1 appendice} transforms into the Cauchy problem 
$$
\begin{cases}
{\cal L}_2 \widetilde{\bf h} = \widetilde{\bf g} \\
\widetilde{\bf h}(t_0, \cdot) ={\bf h}_{in}
\end{cases}\,
$$
(note that for a function $h (x)$ depending only on the variable $x$, ${\cal B}^{- 1} h = h$).
Then the statement follows by Lemma \ref{buona positura equazione lineare 3} and by estimate \eqref{stima cal B pm1} on the transformation ${\cal B}$.
\end{proof}

\begin{lemma}[Well posedness of the operator ${\cal L}_0$ in \eqref{cal L1}]\label{buona positura equazione lineare 5}
Let $T > 0$, $t_0 \in [0, T]$ and let ${\cal L}_0 = \partial_t {\mathbb I}_2 + \ii (\Sigma + {A_2^{(0)}} ) \partial_{xx} + \ii {A_1^{(0)}} \partial_x + \ii {A_0^{(0)}}$ be the operator defined in \eqref{cal L1}. There exists $\eta \in (0, 1)$ small enough and universal constants $\sigma ,\t > 0$ large enough such that if $N_T(\sigma) \leq \eta$ (see the definition \eqref{definizione NT sigma}), then for any $s \in [0, r-\t]$, ${\bf h}_{in} \in {\bf H}^s(\T)$, ${\bf g} \in C([0, T], {\bf H}^s(\T))$, there exists a unique solution ${\bf h} \in C([0, T], {\bf H}^s(\T))$ such that 
\begin{equation}\label{problema di Cauchy cal L0 appendice}
\begin{cases}
{\cal L}_0 {\bf h} = {\bf g} \\
{\bf h}(t_0, \cdot) = {\bf h}_{in}
\end{cases}
\end{equation}
satisfying the estimate 
$$
\| {\bf h}\|_{T, s} \lesssim_s \| {\bf h}_{in}\|_{s} + \| {\bf g}\|_{T, s} + N_T(s + \sigma) \| {\bf h}_{in}\|_{0}\,. 
$$
\end{lemma}

\begin{proof}
Let ${\cal A}$ be the transformation defined in \eqref{definizione cal A}. By \eqref{cal L2}, defining $\widetilde{\bf h}(t, \cdot) := {\cal A}^{- 1}(t) {\bf h}(t, \cdot)$, $\widetilde{\bf g} :={\cal A}^{- 1}(t){\bf g}(t, \cdot)$, the Cauchy problem \eqref{problema di Cauchy cal L0 appendice} transforms into the Cauchy problem 
$$
\begin{cases}
{\cal L}_1 \widetilde{\bf h} = \widetilde{\bf g} \\
\widetilde{\bf h}(t_0, \cdot) =\widetilde{\bf h}_{in}\,.
\end{cases}
$$
Then the statement follows by Lemma \ref{buona positura equazione lineare 4} and by estimate \eqref{stima cal a pm1} on the transformation ${\cal A}$.
\end{proof}

\begin{lemma}[Well posedness of the operator ${\cal L}$ in \eqref{operatore lineare generale}]\label{buona positura equazione lineare 6}
Let $T > 0$, $t_0 \in [0, T]$ and let ${\cal L} = \partial_t {\mathbb I}_2 + \ii (\Sigma + A_2 ) \partial_{xx} + \ii A_1 \partial_x + \ii A_0$ be the operator defined in \eqref{operatore lineare generale}. There exists $\eta \in (0, 1)$ small enough and universal constants $\sigma, \t > 0$ large enough such that if $N_T(\sigma) \leq \eta$ (see the definition \eqref{definizione NT sigma}), then for any $s \in [0, r - \t]$, ${\bf h}_{in} \in {\bf H}^s(\T)$, ${\bf g} \in C([0, T], {\bf H}^s(\T))$, there exists a unique solution ${\bf h} \in C([0, T], {\bf H}^s(\T))$ such that 
\begin{equation}\label{problema di Cauchy cal L appendice}
\begin{cases}
{\cal L} {\bf h} = {\bf g} \\
{\bf h}(t_0, \cdot) = {\bf h}_{in}
\end{cases}
\end{equation}
satisfying the estimate 
$$
\| {\bf h}\|_{T, s} \lesssim_s \| {\bf h}_{in}\|_{s} + \| {\bf g}\|_{T, s} + N_T(s + \sigma) \| {\bf h}_{in}\|_{0}\,. 
$$
\end{lemma}

\begin{proof}
Let ${\cal S}$ be the transformation defined in \eqref{matrice autovettori inversa}. By \eqref{cal L1}, defining $\widetilde{\bf h}(t, \cdot) := {\cal S}^{- 1}(t) {\bf h}(t, \cdot)$, $\widetilde{\bf g} :={\cal S}^{- 1}(t){\bf g}(t, \cdot)$, the Cauchy problem \eqref{problema di Cauchy cal L appendice} transforms into the Cauchy problem 
$$
\begin{cases}
{\cal L}_0 \widetilde{\bf h} = \widetilde{\bf g} \\
\widetilde{\bf h}(t_0, \cdot) =\widetilde{\bf h}_{in}\,.
\end{cases}
$$
Then the statement follows by Lemma \ref{buona positura equazione lineare 5} and by estimate \eqref{azione stima tame cal S} on the transformation ${\cal S}$.
\end{proof}

\section{Appendix D. Nash-Moser-H\"ormander theorem} 
\label{sec:NM}

We state here the Nash-Moser-H\"ormander theorem, proved in \cite{BH}, 
which we use in Section \ref{sec:proof} to prove Theorems \ref{thm:1} and \ref{thm:byproduct}.

Let $(E_a)_{a \geq 0}$ be a decreasing family of Banach spaces with continuous injections  
$E_b \hookrightarrow E_a$, 
\begin{equation} \label{S0}
\| u \|_a \leq \| u \|_b \quad \text{for} \  a \leq b.	
\end{equation}
Set $E_\infty = \cap_{a\geq 0} E_a$ with the weakest topology making the 
injections $E_\infty \hookrightarrow E_a$ continuous. 
Assume that $S_j : E_0 \to E_\infty$ for $j = 0,1,\ldots$ are linear operators 
such that, with constants $C$ bounded when $a$ and $b$ are bounded, 
and independent of $j$,
\begin{alignat}{2}
\label{S1} 
\| S_j u \|_a 
& \leq C \| u \|_a 
&& \text{for all} \ a;
\\
\label{S2} 
\| S_j u \|_b 
& \leq C 2^{j(b-a)} \| S_j u \|_a 
&& \text{if} \ a<b; 
\\
\label{S3} 
\| u - S_j u \|_b 
& \leq C 2^{-j(a-b)} \| u - S_j u \|_a 
&& \text{if} \ a>b; 
\\ 
\label{S4} 
\| (S_{j+1} - S_j) u \|_b 
& \leq C 2^{j(b-a)} \| (S_{j+1} - S_j) u \|_a 
\quad && \text{for all $a,b$.}
\end{alignat}
Set 
\begin{equation}  \label{new.24}
R_0 u := S_1 u, \qquad 
R_j u := (S_{j+1} - S_j) u, \quad j \geq 1.
\end{equation}
Thus 
\begin{equation} \label{2705.3}
\| R_j u \|_b \leq C 2^{j(b-a)} \| R_j u \|_a \quad \text{for all} \ a,b.
\end{equation}
Bound \eqref{2705.3} for $j \geq 1$ is \eqref{S4}, 
while, for $j=0$, it follows from \eqref{S0} and \eqref{S2}.

We also assume that 
\begin{equation} \label{2705.4}
\| u \|_a^2 \leq C \sum_{j=0}^\infty \| R_j u \|_a^2	\quad \forall a \geq 0,
\end{equation}
with $C$ bounded for $a$ bounded (a sort of ``orthogonality property'' of the smoothing operators).

Now let us suppose that we have another family $F_a$ of decreasing Banach spaces with smoothing operators having the same properties as above. We use the same notation also for the smoothing operators.

\begin{theorem} \label{thm:NM}
Let $a_1, a_2, \a, \b, a_0, \mu$ be real numbers with 
\begin{equation} \label{ineq 2016}
0 \leq a_0 \leq \mu \leq a_1, \qquad 
a_1 + \frac{\b}{2} \, < \a < a_1 + \b , \qquad 
2\a < a_1 + a_2. 
\end{equation}
Let $V$ be a convex neighborhood of $0$ in $E_\mu$. 
Let $\Phi$ be a map from $V$ to $F_0$ such that $\Phi : V \cap E_{a+\mu} \to F_a$ 
is of class $C^2$ for all $a \in [0, a_2 - \mu]$, with 
\begin{equation}\label{Phi sec}
\|\Phi''(u)[v,w] \|_a \leq C \big( \| v \|_{a+\mu} \| w \|_{a_0} 
+ \| v \|_{a_0} \| w \|_{a+\mu}
+ \| u \|_{a+\mu} \| v \|_{a_0} \| w \|_{a_0} \big)
\end{equation}
for all $u \in V \cap E_{a+\mu}$, $v,w \in E_{a+\mu}$.
Also assume that $\Phi'(v)$, for $v \in E_\infty \cap V$ 
belonging to some ball $\| v \|_{a_1} \leq \d_1$,
has a right inverse $\Psi(v)$ mapping $F_\infty$ to $E_{a_2}$, and that
\begin{equation}  \label{tame in NM}
\|\Psi(v)g\|_a\leq C(\|g\|_{a + \b - \a} + \| g \|_0 \| v \|_{a + \b}) 
\quad \forall a \in [a_1, a_2].
\end{equation}
For all $A > 0$ there exist $\d, C_1 > 0$ such that, 
for every $g \in F_\b$ satisfying
\begin{equation} \label{2705.1}
\sum_{j=0}^\infty \| R_j g \|_\b^2 \leq A \| g \|_\b^2,
\quad \| g \|_\b \leq \d,
\end{equation}
there exists $u \in E_\a$, with $\| u \|_\a \leq C_1 \| g \|_\b$, 
solving $\Phi(u) = \Phi(0) + g$.

Moreover, let $c > 0$
and assume that \eqref{Phi sec} holds for all $a \in [0, a_2 + c - \mu]$,
$\Psi(v)$ maps $F_\infty$ to $E_{a_2 + c}$, 
and \eqref{tame in NM} holds for all $a \in [a_1, a_2 + c]$. 
If $g$ satisfies \eqref{2705.1} and, in addition, $g \in F_{\b+c}$ with
\begin{equation} \label{0406.1}
\sum_{j=0}^\infty \| R_j g \|_{\b+c}^2 \leq A_c \| g \|_{\b+c}^2 
\end{equation}
for some $A_c$, then the solution $u$ belongs to $E_{\a + c}$, 
with $\| u \|_{\a+c} \leq C_{1,c} \| g \|_{\b+c}$.
\end{theorem}

\medskip

\begin{small}
\noindent
\textbf{Acknowledgements}.
Baldi and Haus were supported by the European Research Council under FP7 (ERC Project 306414),
by PRIN 2012 ``Variational and perturbative aspects of nonlinear differential problems'',
and partially by Programme STAR (UniNA and Compagnia di San Paolo).
Montalto was partially supported by the Swiss National Science Foundation.
\end{small}

\begin{footnotesize}

\end{footnotesize}

\bigskip

Pietro Baldi

Dipartimento di Matematica e Applicazioni ``R. Caccioppoli''

Universit\`a di Napoli Federico II  

Via Cintia, 80126 Napoli, Italy

\texttt{pietro.baldi@unina.it} 

\bigskip

Emanuele Haus 

Dipartimento di Matematica e Applicazioni ``R. Caccioppoli''

Universit\`a di Napoli Federico II  

Via Cintia, 80126 Napoli, Italy

\texttt{emanuele.haus@unina.it}

\bigskip

Riccardo Montalto

Institut f\"ur Mathematik

Universit\"at Z\"urich

Winterthurerstrasse 190

CH-8057 Z\"urich

\texttt{riccardo.montalto@math.uzh.ch}

\end{document}